\documentclass[11pt]{amsart}

\pdfoutput=1

\usepackage{import}
\usepackage[text={420pt,660pt},centering]{geometry}
\usepackage{lipsum}
\usepackage{tikz}
\graphicspath{ {./images/} }
\usepackage{cancel}
\usepackage{esint,amssymb} 
\usepackage{graphicx}
\usepackage{amsmath}
\usepackage{mathtools} 
\usepackage[colorlinks=true, pdfstartview=FitV, linkcolor=blue, citecolor=blue, urlcolor=blue,pagebackref=false]{hyperref}
\usepackage{microtype}

\usepackage{bm}
\usepackage{dsfont}
\usepackage{mathrsfs}
\usepackage{xcolor}
\usepackage{enumitem} 

\usepackage[normalem]{ulem}

\newtheorem{proposition}{Proposition}
\newtheorem{theorem}[proposition]{Theorem}
\newtheorem{lemma}[proposition]{Lemma}
\newtheorem{corollary}[proposition]{Corollary}

\newtheorem{innerresult}{Result}
\newenvironment{result}[1]
  {\renewcommand\theinnerresult{#1}\innerresult}
  {\endinnerresult}

\theoremstyle{remark}
\newtheorem{remark}[proposition]{Remark}

\theoremstyle{definition}

\numberwithin{equation}{section}
\numberwithin{proposition}{section}
\numberwithin{figure}{section}
\numberwithin{table}{section}

\newcommand{\Z}{\mathbb{Z}}
\newcommand{\N}{\mathbb{N}}

\newcommand{\R}{\mathbb{R}}

\newcommand{\E}{\mathbb{E}}
\renewcommand{\P}{\mathbb{P}}

\renewcommand{\S}{\mathbb{S}}
\newcommand{\C}{\mathbb{C}}
\newcommand{\cle}{\mathrm{CLE}}
\newcommand{\sle}{\mathrm{SLE}}

\newcommand{\eps}{\varepsilon}

\renewcommand{\leq}{\leqslant}
\renewcommand{\geq}{\geqslant}

\renewcommand{\subset}{\subseteq}
\renewcommand{\bar}{\overline}
\renewcommand{\tilde}{\widetilde}

\renewcommand{\hat}{\widehat}

\newcommand{\Ll}{\left}
\newcommand{\Rr}{\right}
\renewcommand{\d}{\mathrm{d}}

\DeclareMathOperator{\diam}{diam}

\newcommand{\A}{{\mathcal{A}}}

\newcommand{\cro}{{\mathrm{crs}}}
\newcommand{\sep}{{\mathrm{sep}}}
\newcommand{\good}{{\mathrm{smpl}}}
\newcommand{\lop}{{\mathrm{loop}}}

\newcommand{\D}{D}

\newcommand{\bb}{\mathbf{b}}
\newcommand{\crit}{\mathrm{c}}

\newcommand{\VV}{\mathtt{V}}
\newcommand{\EE}{\mathtt{E}}

\newcommand{\peel}{\mathrm{peel}}
\newcommand{\most}{\mathrm{most}}
\newcommand{\conc}{\mathrm{conc}}
\newcommand{\bbeta}{\boldsymbol{\beta}}

\newcommand{\bt}{\mathbf{t}}

\newcommand{\bn}{{\mathbf n}}

\begin{document}
\usetikzlibrary{decorations.pathreplacing}
\usetikzlibrary {patterns,patterns.meta}

\author[H.-B. Chen]{Hong-Bin Chen}
\address[H.-B. Chen]{Institut des Hautes \'Etudes Scientifiques, France}
\email{hbchen@ihes.fr}

\author[J. Xia]{Jiaming Xia}
\address[J. Xia]{Institut des Hautes \'Etudes Scientifiques, France}
\email{xiajiam@ihes.fr}

\title[Conformal invariance of random currents]{Conformal invariance of random currents:\\ a stability result}

\begin{abstract}
We show the convergence of the single sourceless critical random current to a limit identifiable with the nested $\cle(3)$. Our approach is based on viewing the random current as a perturbation of the Ising interface, which is known to converge to $\cle(3)$. Instead of focusing solely on the random current, we provide a general framework for the stability of scaling limits under the perturbation by superimposing an independent Bernoulli percolation.
\end{abstract}

\maketitle

\section{Introduction}
In the past few decades, the understanding of conformally invariant random objects has advanced rapidly, thanks to the development of the Schramm--Loewner Evolution (SLE) and its related subjects such as the Gaussian Free Field (GFF).
A consensus is that if the scaling limit of a discrete lattice model is conformally invariant, then most of the quantities interesting to mathematical physicists can be computed exactly in the continuum. 

Conjectures have been made regarding a large class of discrete lattice models, suggesting that their interfaces converge to $\sle(\kappa)$-type curves in the scaling limit for $\kappa\in(0,8]$. Such convergence results have been established for the loop-erased random walk~\cite{schramm2000scaling} ($\kappa = 2$), the uniform spanning tree~\cite{lawler2011conformal} ($\kappa=8$), the Ising model~\cite{chelkak2012universality} ($\kappa=3$), the FK--Ising random cluster model~\cite{smirnov2010conformal} ($\kappa=16/3$), Bernoulli percolation~\cite{smirnov2001critical} ($\kappa=6$), and more recently, the double random current \cite{duminil2021conformalI} ($\kappa=4$). The development of SLE leads to the introduction of the Conformal Loop Ensembles (CLE) \cite{sheffield2009exploration}, which is expected to describe the limit of the full set of interfaces. For $\kappa\in (\frac{8}{3},8]$, $\cle(\kappa)$ is a conformally invariant probability measure on collections of loops that locally behave like $\sle(\kappa)$. 
The convergence of the full set of interfaces has been established for Bernoulli percolation \cite{camia2006two} ($\kappa= 6$), the Ising model \cite{benoist2019scaling} ($\kappa=3$), the FK--Ising model \cite{kemppainen2019conformal} ($\kappa=16/3$), and the double random current \cite{duminil2021conformalI,duminil2021conformal} ($\kappa=4$). The connection between $\cle(4)$ and the GFF is established in \cite{aru2019bounded,miller2011cle}.

The proofs of convergence to SLE curves usually involve computing certain discrete observables and proving their convergence in the scaling limit to holomorphic or harmonic functions that satisfy conformally covariant boundary value problems. To achieve the convergence of interfaces through the conformal covariance of observables, additional ingredients are often necessary. The spatial Markov properties of discrete models, including the Ising model, the FK--Ising model, and Bernoulli percolation, are used as crucial tools in proving the conformal invariance of interfaces in their respective models. However, many models do not possess sufficiently favorable discrete properties, and the scaling limit of their interfaces is not an easy consequence of the conformal invariance of certain observables. Examples of such models are the double dimer model \cite{kenyon2014conformal,dubedat2018double,basok2021tau} and the double random current \cite{duminil2021conformalI,duminil2016random}, which require different strategies.

The random current model is an expansion of the Ising model that has proven to be powerful. Applications include correlation inequalities \cite{griffiths1970concavity}, exponential decay \cite{aizenman1987phase,duminil2020exponential,duminil2016new}, Gibbs states \cite{raoufi2020translation}, continuity of the phase transition \cite{aizenman2015random}, and etc. Other recent progress utilizing this tool includes \cite{aizenman2019emergent,duminil2019double,lupu2016note}. We refer to the survey \cite{duminil2016random} for more details on the random current. 
The scaling limit of the random current model is expected to be conformally invariant (see \cite[Question~3]{duminil2016random}).

In this work, we prove the convergence of the single sourceless critical random current to a limit identifiable with the nested $\cle(3)$. We emphasize that the convergence is for the configuration itself, \textit{not} for the interfaces or boundaries of clusters.

Our approach to the convergence of the random current starts by viewing it as a perturbation of the high-temperature expansion via adding an independent Bernoulli percolation.
Since the high-temperature expansion can be identified with the Ising interface, which is known to converge to $\cle(3)$, we recast the convergence problem of the random current as a stability problem for the convergence of the Ising interface under a Bernoulli perturbation.
Instead of solely focusing on the random current, we develop a general framework for addressing the stability problem.

\subsection{Main results}

Let $\D$ be a Jordan domain in $\R^2$ and for each $\delta>0$, let $\D_\delta$ be a natural discretization of $\D$ on the lattice $\delta\Z^2$. We consider a general percolation model $\eta_\delta$ on $\D_\delta$ that is sourceless and, as a consequence, can be decomposed into disjoint discrete loops. On top of $\eta_\delta$, we add a Bernoulli percolation $\bb_\delta$ (where the probability of being open can be different among edges) on $\D_\delta$ and denote the resulting perturbed percolation model by $\omega_\delta$.

We assume that $\eta_\delta$ converges in distribution in the space $\mathscr{L}$ of loop collections on $\D$ to some limit $\eta_0$. We want to show that $\omega_\delta$ converges in distribution to the same limit given that the perturbation is small. Since $\omega_\delta$ does not admit a natural loop representation in general, we encapsulate the information of $\omega_\delta$ in terms of crossings of topological annuli in $\D$. We say that an annulus $A$ is crossed by $S\subset \R^2$ if $S$ contains a path connecting the inner boundary of $A$ to its outer boundary. We put a type of Schramm--Smirnov topology \cite{SchrammSmirnov} on the space $\mathscr{H}$ of \textit{hereditary} collections of crossed annuli.

We emphasize again that we consider the percolation configurations, instead of the interfaces. 
In our setting, the projection of a percolation configuration into $\mathscr{L}$ or $\mathscr{H}$ is defined using the open edges in the configuration, \textit{not} interfaces. Hence, we are in a situation different from \cite{garban2013pivotal} where crossings of quads are made by open edges but the loops are made of interfaces. The relation between crossings and loops in \cite{garban2013pivotal} is measure-theoretic (see \cite[Proposition~2.7]{garban2013pivotal} which uses ideas from \cite{camia2006two}), whereas in our case, the relation is deterministic.
We pass the information of loops to that of crossings by the deterministic map $F:\mathscr{L}\to \mathscr{H}$ defined by
\begin{align*}
    F:\quad L=\{\text{loops}\}\quad\mapsto \quad \{A\subset D: \text{$A$ is crossed by some loop in $L$}\}.
\end{align*}
We show in Section~\ref{s.cts_bij} that its restriction $\hat F:\mathscr{L}^\good \to F(\mathscr{L}^\good)$ is a continuous bijection where $\mathscr{L}^\good$ is the subspace of $\mathscr{L}$ consisting of collections of simple disjoint loops.

A few comments on $\mathscr{H}$ are due: the Schramm--Smirnov topology is used in \cite{SchrammSmirnov} on crossings of quads (i.e.\ topological rectangles), whereas we consider crossings of annuli. We choose annuli instead of quads for two reasons:
\begin{itemize}
    \item For quads, a crossing is defined to be a path within the quad connecting its left and right boundaries. Hence, to realize a crossing, a path cannot exit the top and bottom edges, which requires additional care. For crossings of annuli, there is no such issue.
    \item The crossing of an annulus is dual to the separation of its inner and outer boundaries. A natural geometric object that separates an annulus is a loop, which enables us to relate loops to crossings.
\end{itemize}
Additionally, by construction, the compactness of the Schramm-Smirnov topology is a major advantage.

\begin{result}{A}\label{resultA}
    Suppose that $\eta_\delta$ satisfies some mild conditions and the perturbation intensity is low so that $\omega_\delta$ is close to $\eta_\delta$. If $\eta_\delta$ converges in distribution in $\mathscr{L}$ to some $\eta_0$ as $\eta\to0$, then $(\eta_\delta,\omega_\delta)$ converges in distribution in $\mathscr{L}\times \mathscr{H}$ to $(\eta_0,F(\eta_0))$.

    If, in addition, $\eta_0\in\mathscr{L}^\good$ a.s., then this gives a coupling between $\eta_0$ and the limit $\omega_0$ of $\omega_\delta$ such that $\omega_0 = \hat F(\eta_0)$ and $\hat F^{-1}(\omega_0) = \eta_0$ a.s. 
\end{result}

\noindent The rigorous version of Result~\ref{resultA} is Theorem~\ref{t.(eta,omega)_id_lim}.
The conditions on $(\eta_\delta,\omega_\delta)$ are stated in~\ref{i.H_basics}--\ref{i.H_bdy_conn}.

Then, we apply this to the setting of the random current. We put a critical Ising model $\sigma_\delta$ on the vertices of $\D_\delta$. Let $\eta_\delta$ and $\hat\bn_\delta$ be, respectively, the percolation configuration associated with the high-temperature expansion of $\sigma_\delta$ and the trace of random current associated with $\sigma_\delta$. By a coupling in \cite{aizenman2019emergent}, 
we represent $\hat\bn_\delta=\omega_\delta$ for some $\bb_\delta$ in the above notation.

By the Kramers--Wannier duality \cite{kramers1941statistics}, $\eta_\delta$ is also the interface of a critical Ising model $\sigma^*_\delta$ on the dual graph $\D_\delta^*$ with $+$ boundary condition. It is established in \cite{benoist2019scaling} that $\eta_\delta$ converges to the nested $\cle(3)$ on $\D$.
We also verify that $(\eta_\delta,\hat\bn_\delta)$ satisfies the conditions for Result~\ref{resultA}. Let $L^{\cle(3)}$ be distributed as the nested $\cle(3)$ on $\D$, then by definition $L^{\cle(3)}\in \mathscr{L}^\good$ a.s. Hence, we obtain the following from Result~\ref{resultA}.

\begin{result}{B}\label{resultB}
    As $\delta\to0$, $(\eta_\delta,\hat\bn_\delta)$ converges in distribution in $\mathscr{L}\times\mathscr{H}$ to $(L^{\cle(3)},\hat\bn_0)$ where $\hat\bn_0$ satisfies $\hat\bn_0 = \hat F(L^{\cle(3)})$ and $\hat F^{-1}(\hat\bn_0) = L^{\cle(3)}$ a.s. 
    In particular, $\hat\bn_0$ is conformally invariant.
\end{result}

\noindent 
Therefore, the limit of $\hat\bn_\delta$ is identified with the nested $\cle(3)$ via the deterministic bijection $\hat F$. 
The notion of conformal invariance for an $\mathscr{H}$-valued random variable is clarified in Section~\ref{s.conformal_inv}.

The rigorous version of Result~\ref{resultB} is Corollary~\ref{c.cvg_rc} of Theorem~\ref{t.cvg_rc_cle}. This theorem proves a more general result that the same convergence holds for any perturbation of $\eta_\delta$ with intensity below a certain level $t^\star$. This applies to the random current because it has intensity $t_\crit<t^\star$. In the proof, we utilize a key estimate on the boundary connectivity of double random currents from \cite[Theorem~1.2]{duminil2021conformal}.

\subsection{Outline}

This work is divided into two parts. The first part, Sections~\ref{s.prelim}--\ref{s.id_limits}, concerns the general setting and the proof of Result~\ref{resultA}. The second part, Section~\ref{s.app_random_current}, applies the result to the random current setting and proves Result~\ref{resultB}. We briefly describe each section.

In Section~\ref{s.prelim}, we give various definitions and state the three conditions~\ref{i.H_basics},~\ref{i.H_eta_cross_similar}, and~\ref{i.H_bdy_conn} on $(\eta_\delta,\omega_\delta)$ to be assumed throughout. We also comment on these conditions.

In Section~\ref{s.loop_decomp}, given a sourceless configuration $\kappa$ (applicable to $\eta_\delta$), we construct a loop decomposition of $\kappa$ that gives a collection of loops consisting of open edges in $\kappa$. The construction is built on two procedures: peeling and concatenation. In the end, we obtain loops that are mutually disjoint and simple in a weaker sense. The construction is summarized in Lemma~\ref{l.kappa^lop_loop_decomposition}. This loop decomposition construction is needed in Section~\ref{s.compare_cross} and Section~\ref{s.app_random_current}.

In Section~\ref{s.compare_cross}, we use an exploration process to show that, under the assumed conditions, $\omega_\delta$ and $\eta_\delta$ crosses asymptotically the same collection of annuli, which is used in Section~\ref{s.id_limits} to identify the limits in the annulus-crossing space.

In Section~\ref{s.topologies}, we define relevant topologies. We start by introducing the annulus-crossing space and recalling the results on the Schramm--Smirnov topology from \cite{SchrammSmirnov}.
We then recall the metric topology on the space of collections of loops, on which $\cle(\kappa)$ lives for $\kappa\in(\frac{8}{3},4]$. Later, we prove that the aforementioned map $F$ from loops to crossed annuli is continuous and injective, and thus its restriction $\hat F$ is a continuous bijection.
Lastly, we clarify the meaning of conformal invariance in our context.

In Section~\ref{s.id_limits}, we combine results from previous sections to prove Result~\ref{resultA}, which completes the first part.

In Section~\ref{s.app_random_current}, we start the second part. We begin with the definitions of $\eta_\delta$ and $\hat\bn_\delta$ for the critical Ising model and recall useful couplings. Next, using results in \cite{benoist2019scaling}, we show that the loop decomposition of $\eta_\delta$ converges to $\cle(3)$. In \cite{benoist2019scaling}, it is the collection of outermost loops in $\eta_\delta$ that converges, which is different from our loop decompositions. Lastly, we verify that under this setting, the conditions are satisfied and apply Result~\ref{resultA} to get Result~\ref{resultB}. 

\subsection{Acknowledgements}
The authors are grateful to Hugo Duminil-Copin for the introduction to this subject, stimulating discussions, consistent encouragement, and feedback on the preliminary version of the manuscript. The authors warmly thank Dmitry Chelkak, Fran\c cois Jacopin, Piet Lammers, and Yiling Wang for answering many questions. This project has received funding from the European Research Council (ERC) under the European Union’s Horizon 2020 research and innovation programme (grant agreement No.\ 757296).

\section{Preliminaries}\label{s.prelim}

\subsection{Definitions}

\subsubsection{Topological objects}\label{s.def_topo_obj}

A \textbf{path} $\gamma$ is a continuous function from $[0,1]$ to $\R^2$.
Let $\mathbb{S}^1$ be the unit circle.
A \textbf{loop} is a continuous function $l:\S^1\to\R^2$. We call $l$ \textbf{simple} if $l(\S^1)$ is homeomorphic to $\S^1$. We call $l$ \textbf{non-self-crossing} if there is a sequence $(l_n)_{n\in\N}$ of simple loops converging to $l$ uniformly.

For a simple loop $l$, a subset $R\subset\R^2$ is said to be \textbf{inside} (resp.\ \textbf{outside}) $l$ if $R$ is in the closed bounded (resp.\ the unbounded) component of $\R^2\setminus l$. For a non-self-crossing loop $l$,
$R$ is said to be \textbf{inside} (resp.\ \textbf{outside}) $l$ if there is a sequence $(l_k)_{k\in\N}$ of simple loops converging to $l$ such that $R$ is inside (resp.\ outside) $l_k$ for each $k$.

A \textbf{domain} is a subset of $\R^2$ that is homeomorphic to the closed unit disc. In particular, its boundary is a simple loop.

An \textbf{(topological) annulus} $A$ is a homeomorphism from $\mathbb{S}^1\times [0,1]$ onto its image in $\R^2$ such that the \textbf{inner boundary} $\partial_0 A = A(\S^1\times \{0\})$ is inside the \textbf{outer boundary} $\partial_1 A = A(\S^1\times \{1\})$. Then, the usual boundary of $A$ is given by $\partial A = \partial_0 A\cup \partial_1 A$.

We often identify $\gamma$, $l$, and $A$ with their images.

\subsubsection{Discrete objects}\label{s.discrete_objects}
Throughout, for points $z=(z_1,z_2)\in \R^2$, we denote its Euclidean norm by $|z| = \sqrt{|z_1|^2+|z_2|^2}$ and its sup-norm by $|z|_\infty= \max\{|z_1|,|z_2|\}$. These norms naturally induce the Euclidean distance and the sup-distance, respectively.

For $\delta>0$, we view $\delta\Z^2$ as a graph with vertices being points in $\delta\Z^2$ and edges being the line segments between two points in $\delta\Z^2$ at a distance $\delta$ of each others.
Two vertices $x,y\in \delta\Z^2$ is said to be \textbf{neighbors} if $|x-y| =\delta$ and we write $x\sim y$. We denote by $xy$ the edge with endpoints $x$ and $y$. 

We can view subgraphs of $\delta\Z^2$ as subsets of $\R^2$. Indeed, we identify each vertex as a singleton set in $\R^2$ and identify edges $xy$ with $\Ll\{sx+(1-s)y:\:s\in[0,1]\Rr\}$.
We define the \textbf{continuous version} of a subgraph of $\delta\Z^2$ to be the union of its vertices and edges under this identification. 

A \textbf{$\delta$-discrete path} $\gamma$ is a sequence $(x_0,x_1,\dots,x_n)$ of neighboring vertices in $\delta\Z^2$.
We denote also by $\gamma$ the subgraph with vertices $\{x_i\}_{i=0}^n$ and edges $\{x_{i-1}x_{i}\}_{i=1}^n$. 
By viewing it as a subset of $\R^2$, $\gamma$ can be identified with a continuous path, which is the continuous version of $\gamma$.

If a $\delta$-discrete path $l = (x_0,x_1\dots,x_n)$ satisfies $x_0=x_n$, then $l$ is called a \textbf{$\delta$-discrete loop}. 
The loop $l$ is said to be \textbf{simple} if $x_0,x_1,\dots,x_{n-1}$ are distinct; and \textbf{non-self-crossing} if the oriented edges $(\overrightarrow{x_ix_{i+1}})_{i=0}^n$ are distinct and there is no vertices $x_i$ and $x_j$ such that $(x_{i-1},x_i,x_{i+1})$ and $(x_{j-1},x_j,x_{j+1})$ are on two orthogonal lines (namely, $l$ does not go through itself). Note that a $\delta$-discrete loop is simple or non-self-crossing if and only if its continuous version is so. A subset of $\R^2$ is said to be inside (resp.\ outside) a $\delta$-discrete non-self-crossing (in particular, simple) loop if it is so with respect to the continuous version of the loop.

For a subgraph $G$ of $\delta\Z^2$, we denote by $\VV(G)$ and $\EE(G)$ its vertex set and edge set, respectively. We define the \textbf{(inner) boundary} of $G$ by
\begin{align}\label{e.boundary}
    \partial G = \Ll\{x\in \VV(G):\:\exists y\not\in \VV(G),\, x\sim y\Rr\}.
\end{align}

A subgraph $K$ is said to be a \textbf{$\delta$-discrete domain} if $\partial K$ is traversed by a simple loop; and a \textbf{$\delta$-discrete disc} if $\partial K$ is traversed by a non-self-crossing loop.
A subgraph $A$ of $\delta\Z^2$ is a \textbf{$\delta$-discrete annulus} if $\partial A$ consists of two non-self-crossing loops, one inside the other.
Note that $\delta$-discrete discs and annuli can be degenerate.

When the discreteness is obvious from the context, we sometimes drop the qualifier ``$\delta$-discrete''.

\subsubsection{Percolation configurations}\label{s.kappa_prop}

On a subgraph $G$ of $\delta \Z^2$, we call an element $\kappa =(\kappa_e)_{e\in \EE(G)}$ in $\{0,1\}^{\EE(G)}$ a \textbf{percolation configuration (on $G$)}. An edge $e$ is said to be \textbf{open} (resp.\ \textbf{closed}) in $\kappa$ if $\kappa_e= 1$ (resp.\ $\kappa_e= 0$).

We view $\kappa$ as a graph by taking $\EE(\kappa) = \{e: \kappa_e=1\}$ and $\VV(\kappa)$ to consist of endpoints of edges in $\EE(\kappa)$. When $\kappa$ is involved in a set-theoretical operation ($\cup$, $\cap$, $\setminus$, and etc.), $\kappa$ is viewed as the continuous version of the graph $(\VV(\kappa), \EE(\kappa))$.

For any subgraph $R$ of $G$, we define the \textbf{restriction} of $\kappa$ to $R$ by
\begin{align}\label{e.kappa_K}
    \kappa_{R} = (\kappa_e)_{e\in \EE(R)}.
\end{align}
We define the \textbf{trivial extension} of $\kappa$ to a percolation configuration $\bar\kappa$ on $\delta\Z^2$ by declaring, for every $e\in \EE(\delta\Z^2)$,
\begin{align}\label{e.trivial_ext}
    \bar\kappa_e = 1 \qquad\Longleftrightarrow\qquad e\in \EE(G),\quad \kappa_e = 1.
\end{align}

For any subgraph $R$ of $\delta\Z^2$, we define the \textbf{source} set of $\kappa$ relative to $R$ to be the set
\begin{align}\label{e.source}
    \partial_R \kappa = \Ll\{x\in \VV(\delta\Z^2):\sum_{y:\:  xy\in \EE(R)} \bar\kappa_{xy}\text{ is odd}\Rr\}.
\end{align}
It is possible that $\partial_R\kappa\not\subset \VV(R)$ but every vertex in $\partial_R\kappa$ is an endpoint to some edge in $\EE(R)$.

\subsubsection{The main domain and its discretization}
Throughout, we fix a domain $\D\subset \R^2$ and a family $(\D_\delta)_{\delta>0}$ of $\delta$-domains such that the continuous version of the loop $\partial\D_\delta$ converges to $\partial\D$ as $\delta\to0$ in the topology of uniform convergence up to reparametrization.

\subsubsection{Dyadic annuli}\label{s.dyadic_annuli}

For $k\in \N$, let $\mathcal{A}^k_\D$ be the collection of topological annuli in the interior of $\D$ with boundaries on the grid of $2^{-k}\Z^2$. 
Set $\mathcal{A}_{\D}^{\mathrm{dya}} = \cup_{k\in\N} \mathcal{A}_\D^k $ to be the collection of \textbf{dyadic annuli}. 
Since every $A\in\mathcal{A}_{\D}^{\mathrm{dya}}$ is away from $\partial D$, for sufficiently small $\delta$, we have that $A$ is inside $\partial D_\delta$.

\subsubsection{Probabilities measures}
\label{s.def_eta^loop}
We introduce the main objects to work with.
For each $\eta_\delta \in \{0,1\}^{\EE(\D_\delta)}$, we are given a real number $w(\eta_\delta)\geq 0$ which can be thought of as the weight on $\eta_\delta$. 
For every subgraph $G\subset\D_\delta$, we define a probability measure $\P_G^\emptyset$ on the sourceless percolation configuration by
\begin{align}\label{e.P^emptyset_G}
    \P_G^\emptyset \Ll\{\eta_\delta\in \mathscr{E}\Rr\} = \frac{\sum_{\eta_\delta:\:\eta_\delta\in \mathscr{E},\, \partial_G \eta_\delta =\emptyset }w\Ll(\eta_\delta\Rr)}{\sum_{\eta_\delta:\:\partial_G \eta_\delta=\emptyset }w\Ll(\eta_\delta\Rr)}
\end{align}
for every event $\mathscr{E}$ depending only on edges in $\EE(G)$.
We need the following version of the \textbf{Markov property}: 
\begin{align}\label{e.markov_prop}
\begin{cases}
    \textit{for every $\delta$-disc $K\subset\D_\delta$ and every $\kappa \in \{0,1\}^{\EE(\D_\delta\setminus K)}$ satisfying $\partial_{\D_\delta\setminus K} \kappa=\emptyset$,}
    \\
    \textit{it holds that }\P^\emptyset_{\D_\delta}\Ll\{\eta_\delta\in \mathscr{E} \,\big|\,(\eta_\delta)_{\D_\delta\setminus K}=\kappa\Rr\} = \P^\emptyset_K\Ll\{\eta_\delta\in \mathscr{E}\Rr\} 
    \\
    \textit{for every event $\mathscr{E}$ depending only on the edges in $\EE(K)$.}
\end{cases}
\end{align}

\subsubsection{Bernoulli perturbation}\label{s.bernoulli_pert}
For each $\delta>0$, let $\bt_\delta \in [0,1]^{\EE(\D_\delta)}$ be a collection of Bernoulli coefficients on the edges of $\D_\delta$.
Let $\bb_\delta^{\bt_\delta}$ be the Bernoulli percolation on $\EE(\D_\delta)$ with coefficients $\bt_\delta$. More precisely, for each $e\in \EE(\D_\delta)$, $(\bb_\delta^{\bt_\delta})_e$ is an independent Bernoulli random variable with coefficient $(\bt_\delta)_e$.
If all the entries in $\bt_\delta$ are equal to some $t\in [0,1]$, we use the shorthand notation $\bb_\delta^t$.

Then, we describe a perturbation of $\eta_\delta$ by opening additional edges according to $\bb_\delta^{\bt_\delta}$. For real numbers $a, b$, we write
\begin{align}\label{e.aveeb}
    a\vee b = \max\{a,b\}.
\end{align}
For each $\delta$, we define the perturbed configuration
\begin{align}\label{e.omega_delta}
    \omega_\delta = \eta_\delta\vee \bb^{\bt_\delta}_\delta= ((\eta_\delta)_e \vee (\bb^{\bt_\delta}_\delta)_e)_{e\in \EE(\D_\delta)}.
\end{align}
By extension, we assume that $\eta_\delta$ and $\bb^{\bt_\delta}_\delta$ are defined on the same probability space and we denote the joint probability measures still by $(\P^\emptyset_G)_{G\subset \D_\delta}$.

\subsubsection{Connection}

Let $B_0$, $B_1$, and $S$ be subsets of $\R^2$. We write $B_0\stackrel{S}{\longleftrightarrow} B_1$ and say $S$ \textbf{connects} $B_0$ and $B_1$ if $S$ contains a connected compact subset intersecting both $B_0$ and $B_1$. For any annulus $A$, we write
\begin{align}\label{e.circledcirc_A^s}
    \circledcirc_A^S = \Ll\{\partial_0 A \stackrel{S\cap A}{\longleftrightarrow} \partial_1 A\Rr\}.
\end{align}

\subsection{Conditions}

For any subset $S\subset \R^2$ and $r>0$, we set $(S)_r = \{y \in \R^2: d(y,S)\leq r\}$ where $d$ is the Euclidean distance. 
Let $\eta_\delta$ and $\omega_\delta$ be given in Sections~\ref{s.def_eta^loop} and~\ref{s.bernoulli_pert} respectively.
We state the following conditions on $(\eta_\delta,\omega_\delta)$:

\begin{enumerate}[start=0,label={\rm(H\arabic*)}]
\item \label{i.H_basics}
For every $\delta>0$, the Markov property~\eqref{e.markov_prop} holds for $\eta_\delta$.

\item 
\label{i.H_eta_cross_similar} 
    For every $A \in \mathcal{A}_{\D}^{\mathrm{dya}}$, there is a family $(A^\eps)_{\eps>0}$ of piecewise smooth annuli such that
    \begin{itemize}
        \item for sufficiently small $\eps$, $\partial_0 A^\eps$ is in the interior of $A$, $\partial_0A$ is inside $\partial_0 A^\eps$, and $\partial_1 A^\eps = \partial_1 A$;
        \item the following holds:
        \begin{align*}
            \lim_{\eps\to 0}\limsup_{\delta\to 0} \P_{\D_\delta}^\emptyset\Ll( \circledcirc_{A^\eps}^{\eta_\delta}\setminus\circledcirc_{A}^{\eta_\delta} \Rr)=0.
        \end{align*}
    \end{itemize}

\item 
\label{i.H_bdy_conn}
For every $R>0$ and $x\in\R^2$,
\begin{align*}
    \lim_{r\to0}\limsup_{\delta\to0}\sup_K \P_K^\emptyset\Ll( x+[-R,R]^2 \stackrel{\omega_\delta\cap K}{\longleftrightarrow} (\partial K)_r \Rr) = 0
\end{align*}
where the supremum is taken over every $\delta$-disc $K\subset \D_\delta$ satisfying $x+[-2R,2R]^2 \subset K\not\subset x+[-3R,3R]^2  $.
\end{enumerate}

The first two conditions are imposed only on $\eta_\delta$. Condition~\ref{i.H_basics} is compatible with the Markovian restriction property for $\cle(\kappa)$ with $\kappa\in(\frac{8}{3},4]$ (see \cite[Section~1.2]{sheffield2012conformal}). Condition~\ref{i.H_eta_cross_similar} is weaker than the continuity of the probability of crossing events in terms of the annulus. We only need the probability of crossing $A$ to be approximated by that of $A_\eps$ inside $A$. 
Also, $\mathcal{A}^\mathrm{dya}_\D$ can be replaced by any dense (in the metric in~\eqref{e.d_(A_D)}) collection of piecewise smooth annuli.

The only condition on $\omega_\delta$ is~\ref{i.H_bdy_conn}, which requires $\eta_\delta$ and $\omega_\delta$ to have similar connectivity. Under $\P^\emptyset_K$, assuming that the limit of $\eta_\delta$ consists of disjoint simple loops, we expect that the loops of $\eta_\delta$ in $K$ do not touch $\partial K$. Condition~\ref{i.H_bdy_conn} also ensures that a similar event does not happen for $\omega_\delta$, meaning that the additional edges in $\omega_\delta$ are not numerous enough to make macroscopic disconnected sets in $\eta_\delta$ become connected in $\omega_\delta$.

The key consequence of~\ref{i.H_basics}--\ref{i.H_bdy_conn} is Lemma~\ref{l.discon_bdy} which implies that the probability of $\eta_\delta$ crossing any annulus $A$ is comparable to that of $\omega_\delta$. This is further used in Section~\ref{s.id_limits} to show that they have the same law in the limit in the annulus-crossing space.

\section{Loop decomposition of sourceless configurations}\label{s.loop_decomp}

Fix $\delta>0$, a typical realization $\eta_\delta$ under $\P^\emptyset_{\D_\delta}$ satisfies $\partial_{\D_\delta} \eta_\delta = \emptyset$. We want to decompose $\eta_\delta$ into a collection of loops. We approach this in a more general setting.

\subsection{More definitions}
\label{s.loop_properties}

\subsubsection{Equivalent classes}Throughout, loops are identified up to reparametrization. Continuous loops $l$ and $l'$ are identified with each other if there is a homeomorphism $\phi:\S^1 \to \S^1$ with counterclockwise orientation such that $l = l'\circ\phi$. For discrete loops, besides using this definition for their continuous versions, we have a simpler criterion: $l=(x_0,\dots,x_n)$ is identified with $l'=(x'_0,\dots,x'_n)$ if there is $k\in\Z$ such that $x_i=x'_{i+k}$ for all $i\in\Z$ where $i$ and $i+k$ are integers modulo $n$.

\subsubsection{Weakly simple loops}

A $\delta$-discrete loop $l=(x_0,x_1,\dots,x_n)$ with $x_0=x_n$ is called \textbf{weakly simple} if the edges $(x_ix_{i+1})_{i=1}^n$ are distinct and, whenever for some $x_k\in \VV(l)$, $yx_k \in \EE(l)$ for all $y\sim x_k$, it holds that $x_{k-1}x_k$ is perpendicular to $x_kx_{k+1}$. Note that the vertices of $l$ are not required to be distinct and the same vertex can be visited twice, meaning that being ``weakly simple'' is weaker than being ``simple'', but is stronger than being ``non-self-crossing''. 

For a weakly simple $\delta$-discrete loop $l$, its \textbf{orientation} (\textbf{counterclockwise} or \textbf{clockwise}) is defined in the standard way. For instance, one can find a sequence of simple loops that approximates the continuous version of $l$ in the uniform topology; and then use complex analysis to define the orientation for the simple loops; and finally declare the orientation of $l$ as the limit. 

\subsubsection{Loops as planar subsets}\label{s.loop_subset_R^2}

For two loops $l$ and $l'$ we write $l=_{\R^2}l'$ if their continuous versions are the same in $\R^2$, which is equivalent to $\EE(l) = \EE(l')$ for discrete loops.

Whenever a loop $l$ is involved in a set-theoretic operations ($\cup$, $\cap$, $\setminus$, and etc.), $l$ is viewed as a subset of $\R^2$.

For two collections $L$ and $L'$ of loops, we write $L\subset_{\R^2} L'$ if for every $l\in L$ there is $l'\in L'$ such that $l=_{\R^2} l'$; and write $L=_{\R^2} L'$ if $L\subset_{\R^2} L'$ and $L'\subset_{\R^2} L$.

\subsubsection{Dual graphs}\label{s.dual_graphs}
We set $(\Z^2)^* = (\frac{1}{2},\frac{1}{2})+\Z$. We give $(\Z^2)^*$ the obvious graph theoretic structure and represent the dual lattice of $\delta \Z^2$ by $\delta (\Z^2)^*$. For objects on $\delta(\Z^2)^*$, we use the same notation $\EE(\cdot)$, $\VV(\cdot)$, $\sim$, $\partial$ and the same terminology, the definitions of which are adapted in the obvious way.

For each edge $e=xy \in \delta \Z^2$, its \textbf{dual edge}, denoted by $e^*=(xy)^*$, is the unique element in $\EE(\delta (\Z^2)^*)$ that intersects $e$ as a subset of $\R^2$.

As usual, we call $\delta \Z^2$ the primal lattice.
In the following, discrete objects are assumed to be on the primal lattice if not specified.

Given a $\delta$-discrete disc $K$, we define its \textbf{dual disc} $K^*$ by first taking $\gamma$ to be the simple loop in the dual lattice that surrounds the loop $\partial K$ (see Section~\ref{s.surround} for ``surround'') and then taking $K^*$ to be the $\delta$-discrete disc on the dual lattice with $\partial K^* =\gamma$.

\subsubsection{Surrounding loops}\label{s.surround}

Given a $\delta$-discrete weakly simple loop $l$ on the primal lattice $\delta\Z^2$ and a $\delta$-discrete simple loop $l'$ on the dual lattice $\delta (\Z^2)^*$, we say that $l'$ \textbf{surrounds} $l$ if for every $e\in \EE(l)$ the endpoint of $e^*$ outside $l$ is in $\VV(l')$.

\subsubsection{Disjointedness}
Two loops $l$ and $l'$ are said to be \textbf{disjoint} if $l\cap l'=\emptyset$, understood as in Section~\ref{s.loop_subset_R^2}. If they are $\delta$-discrete, then it is equivalent to $\VV(l)\cap \VV(l')=\emptyset$. Two $\delta$-discrete loops $l$ and $l'$ are said to be \textbf{weakly disjoint} if $\EE(l)\cap \EE(l')=\emptyset$, meaning that they may share vertices. A collection is said to contain disjoint (resp.\ weakly disjoint) loops if those loops are pairwise disjoint (resp.\ weakly disjoint).

\subsection{Constructions}

Let $K$ be a $\delta$-discrete domain and $\kappa$ be a percolation configuration on $K$ satisfying the sourceless condition $\partial_K \kappa=\emptyset$.

\subsubsection{Loop decompositions}

A \textbf{loop decomposition (on $K$)} of $\kappa$ is a collection $L$ of disjoint weakly simple loops which satisfies
\begin{align}\label{e.kappa_e=1<=>...}
    \forall e\in \EE(K):\quad \Ll(\ \kappa_e = 1\quad \Longleftrightarrow\quad \text{$e \in \EE(l)$ for some $l\in L$} \ \Rr).
\end{align}
Note that
\begin{align}\label{e.decompsoable=>sourceless}
    \textit{a loop decomposition of $\kappa$ exists }\quad\Longrightarrow \quad \partial_{G} \kappa = \emptyset
\end{align}
because any endpoint $x$ to some open edge $e$ is an endpoint to one more or three more open edges since $e$ must be on a loop and the loops the loop decomposition are disjoint.

\subsubsection{Basic loop collections}\label{s.loop_construct_1}

We define the dual configuration $\kappa^*$ on the dual domain $K^*$ by setting $\kappa^*_{e^*} = 1-\kappa_e$ for edges $e^*\in \EE(K^*)$ whose primal counterparts are in $\EE(K)$; setting $\kappa^*_{e'}=1$ for the remaining edges in $\EE(K^*)$. As in the primal case, we view $\kappa^*$ as a graph with edges open in $\kappa^*$ and vertices being all the endpoints.

We start by taking all possible loops on open edges in $\kappa$. We set
\begin{align*}
    L^\mathrm{all} = \Ll\{\text{$\delta$-discrete weakly simple loop } l:\: e\in \EE(l)\ \Longleftrightarrow \ \kappa_e =1\Rr\}.
\end{align*}
This is a large collection of loops which are not disjoint in general. Then, a weakly simple loop $l\in L^\mathrm{all}$ is said to be \textbf{outmost} if there is a simple loop $l'\in L^*$ that surrounds $l$. Hence, no loop in $L^\mathrm{all}$ can touch $l$ from the outside of $l$.
We now refine the collection in steps.

We iteratively define the level of loops and the corresponding collections (see Figure~\ref{fig:picl^all}).
Since $K$ is a disc, its boundary $\partial K$ is a simple loop. We set $\tilde L_0 = \{\partial K\}$ as a reference collection and we emphasize that $\partial K$ may not be a loop in $\kappa$.
We proceed inductively. Given a collection $\tilde L_i$, we construct $\tilde L_{i+1}$ as follows.
Let $l$ be any loop in $L^\mathrm{all}\setminus \cup_{j=1}^i \tilde L_j$. 
We put $l$ in $\tilde L_{i+1}$ if
\begin{itemize}
    \item there is some $l'\in \tilde L_i$ such that $l$ is inside $l'$,
    \item and there are $e \in l$ and $e'\in l'$ such that $e^*\stackrel{\kappa^*}{\longleftrightarrow} (e')^*$
\end{itemize}
(note that $\kappa^*=0$ on $e^*$ and $(e')^*$; hence, $\kappa^*$ connects one endpoint of $e^*$ to one of $(e')^*$). Iteratively, we exhaust $L$ and construct $(\tilde L_i)_{i=1}^\infty$. Since we work on a discrete lattice, $\tilde L_i$ is empty for sufficiently large $i$. 

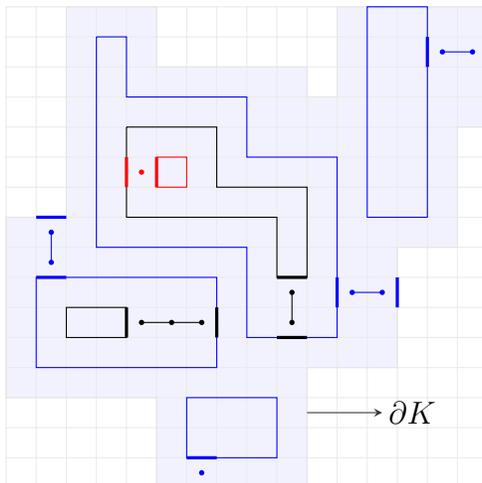
\begin{figure}
    \centering
    \begin{tikzpicture}
\fill[blue!5]
(-3.2,-2) -| (-1.2,-3.2) -| (0.8,-1.6) -| (2,0) -| (2.8,1.6) -| (3.2,3.2) -| (1.2,2) -| (0.8,2.4) -| (-1.2,3.2) -| (-2.4,0.4) -| (-3.2,-2);
\draw[black!8, step=4mm] (-3.21,-3.21) grid (3.2,3.2);
\draw[black!10] (-3.2,-2) -| (-1.2,-3.2) -| (0.8,-1.6) -| (2,0) -| (2.8,1.6) -| (3.2,3.2) -| (1.2,2) -| (0.8,2.4) -| (-1.2,3.2) -| (-2.4,0.4) -| (-3.2,-2);
\draw[blue!]
(-2.8,-1.6) rectangle (-0.4,-0.4)
(-2,2.8) -| (-1.6,2) -| (0,1.2) -| (1.2,-1.2) -| (0,0) -| (-2,2.8)
(1.6,0.4) rectangle (2.4,3.2)
(-0.8,-2.8) rectangle (0.4,-2)
;
\draw[black!]
(-2.4,-1.2) rectangle (-1.6,-0.8)
(-1.6,1.6) -| (-0.4,0.8) -| (0.8,-0.4) -| (0.4,0.4) -| (-1.6,1.6)
;
\draw[red]
(-1.2,0.8) rectangle (-0.8,1.2);
\draw[very thick, red]
(-1.2,0.8)--(-1.2,1.2)
(-1.6,0.8)--(-1.6,1.2)
;
\fill[red](-1.4,1) circle (1pt);
\draw[very thick, black]
(-1.6,-1.2)--(-1.6,-0.8)
(-0.4,-1.2)--(-0.4,-0.8)
(0.4,-0.4)--(0.8,-0.4)
(0.4,-1.2)--(0.8,-1.2)
;
\draw[very thick, blue]
(-2.8,-0.4)--(-2.4,-0.4)
(-2.8,0.4)--(-2.4,0.4)
(-0.8,-3.2)--(-0.4,-3.2)
(-0.8,-2.8)--(-0.4,-2.8)
(1.2,-0.8)--(1.2,-0.4)
(2,-0.8)--(2,-0.4)
(2.4,2.4)--(2.4,2.8)
(3.2,2.4)--(3.2,2.8)
;
\draw[black]
(-1.4,-1)--(-0.6,-1)
(0.6,-0.6)--(0.6,-1)
;
\draw[blue]
(-2.6,-0.2)--(-2.6,0.2)
(1.8,-0.6)--(1.4,-0.6)
(2.6,2.6)--(3,2.6)
;
\fill[black]
(-1.4,-1) circle (1pt)
(-1,-1) circle (1pt)
(-0.6,-1) circle (1pt)
(0.6,-0.6) circle (1pt)
(0.6,-1) circle (1pt)
;
\fill[blue]
(-2.6,-0.2) circle (1pt)
(-2.6,0.2) circle (1pt)
(-0.6,-3) circle (1pt)
(1.8,-0.6) circle (1pt)
(1.4,-0.6) circle (1pt)
(2.6,2.6) circle (1pt)
(3,2.6) circle (1pt)
;
\draw[-stealth,black!80] (0.8,-2.2) -- (1.8,-2.2);
\draw[] (2.2,-2.2) node{\large $\partial K$};
\end{tikzpicture}

    \caption{The shaded region is the $\delta$-discrete domain. The loops in blue, black, and red are in $\tilde L_1$, $\tilde L_2$, and $\tilde L_3$, respectively. For each loop $l \in \tilde L_i$ in the figure, there is a path (in the respective color) on the dual domain connecting $l$ to some $l'\in \tilde L_{i-1}$.}
    \label{fig:picl^all}
\end{figure}

The final touch is to fix orientations to avoid repetition of loops identical up to orientation. Let $L_i$ be the collections of the loops in $\tilde L_i$ with the following convention of orientation:
\begin{align}\label{e.orientation}
    \text{clockwise if $i$ is odd; and counterclockwise if $i$ is even.}
\end{align}
For $i\geq 1$, the loops in $L_i$ are said to be \textbf{level-$i$}. For each $i$, let $L_i^\most \subset L_i$ be the collection of level-$i$ outmost loops.

Similar to $L^\mathrm{all}$, we set
\begin{align*}
    L^* = \Ll\{\text{$\delta$-discrete loop $l'$ on the dual lattice}:\: e'\in \EE(l')\ \Longleftrightarrow \ \kappa^*_{e'} =1\Rr\}.
\end{align*}

\subsubsection{Peeling procedure}\label{s.loop_construct_2}

For each level $i$, we need a more refined collection $L_i^\peel$ than $L_i$. The procedure is to inductively ``peel'' loops from $\kappa$ for each level (see the first row in Figure~\ref{fig:peelconc}). 

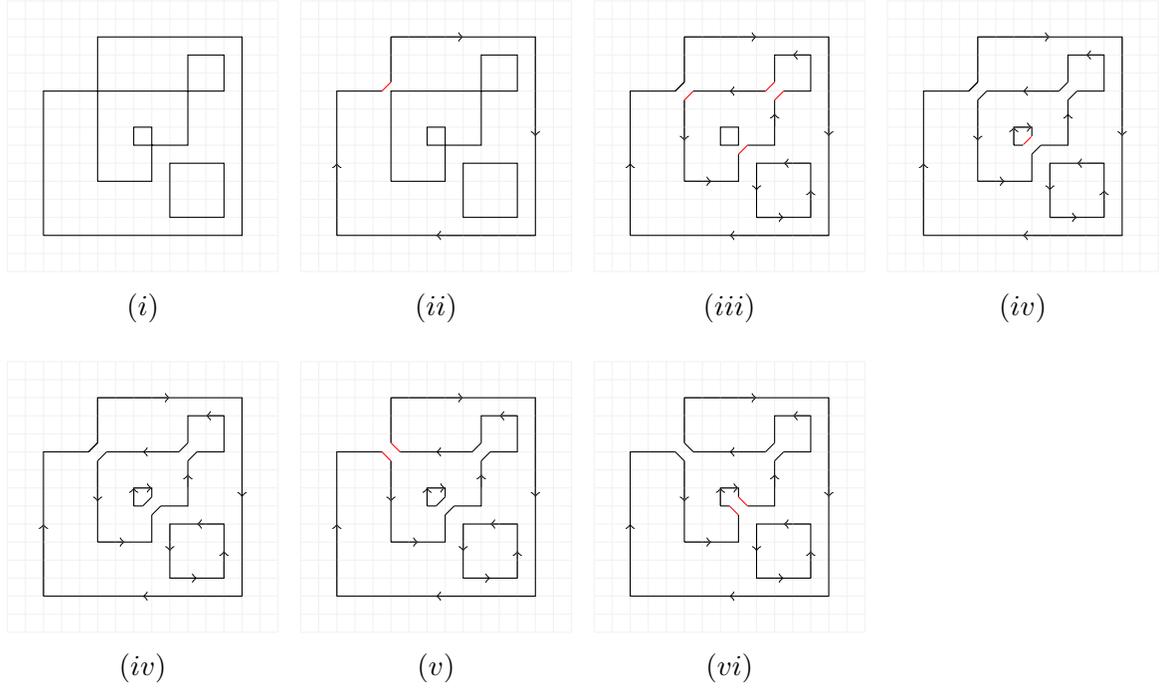
\begin{figure}
    \centering
    \begin{tikzpicture}[scale = 0.6]
\node (A) at (0.8, 0.8) { };
\node (B) at (5.2, 5.2) { };
\node (C) at (2,4) { };
\node (D) at (3.2,2) { };
\node (E) at (2.8,2.8) { };
\node (F) at (3.2,3.2) { };
\node (G) at (4,3.8) { };
\node (H) at (4.8,5) { };
\node (I) at (4,5.8) { };
\node (J) at (3.6,1.2) { };
\node (K) at (4.8,2.4) { };
\draw[black!5, step=4mm, xshift=-70mm] (0,0) grid (6,6);
\draw[black,xshift=-70mm]
(0.8,0.8) -| (5.2,5.2) -| (2,4) -| (0.8,0.8) 
(2,2) |- (4,4) |- (4.8,4.8) |- (4,4) |- (3.2,2.8) |- (2.8,3.2) |- (3.2,2.8) |- (2,2)
(3.6,1.2) rectangle (4.8,2.4)
;

\draw[-to,xshift=-5mm] (0.8,0.8) -- (0.8,2.4);
\draw[-to,xshift=-5mm] (2,4.2) |- (3.6,5.2);
\draw[-to,xshift=-5mm] (5.2,5.2) -- (5.2,3);
\draw[-to,xshift=-5mm] (5.2,0.8) -- (3,0.8);
\draw[black!5, step=4mm,xshift=-5mm] (0,0) grid (6,6);
\draw[black,xshift=-5mm] (0.8,0.8) -- (5.2,0.8) -- (5.2,5.2) -- (2,5.2) -- (2,4.2)
(1.8,4) -| (0.8,0.8)
(2,2) |- (4,4) |- (4.8,4.8) |- (4,4) |- (3.2,2.8) |- (2.8,3.2) |- (3.2,2.8) |- (2,2)
(3.6,1.2) rectangle (4.8,2.4);

\draw[red,xshift=-5mm] (1.8,4) -- (2,4.2);

\draw[-to,xshift=60mm] (0.8,0.8) -- (0.8,2.4);
\draw[-to,xshift=60mm] (2,4.2) |- (3.6,5.2);
\draw[-to,xshift=60mm] (5.2,5.2) -- (5.2,3);
\draw[-to,xshift=60mm] (5.2,0.8) -- (3,0.8);
\draw[-to,xshift=60mm] (3.6,1.2) -- (4.2,1.2);
\draw[-to,xshift=60mm] (4.8,1.2) -- (4.8,1.8);
\draw[-to,xshift=60mm] (4.8,2.4) -- (4.2,2.4);
\draw[-to,xshift=60mm] (3.6,2.4) -- (3.6,1.8);
\draw[-to,xshift=60mm] (2,3.8) -- (2,2.9);
\draw[-to,xshift=60mm] (2,2) -- (2.6,2);
\draw[-to,xshift=60mm] (4,3.2) -- (4,3.5);
\draw[-to,xshift=60mm] (4.8,4.8) -- (4.4,4.8);
\draw[-to,xshift=60mm] (3.8,4) -- (3,4);
\draw[black!5, step=4mm, xshift=60mm] (0,0) grid (6,6);
\draw[black,xshift=60mm] (0.8,0.8) -- (5.2,0.8) -- (5.2,5.2) -- (2,5.2) -- (2,4.2)
(2,4.2) -- (1.8,4)
(1.8,4) -| (0.8,0.8)
(2,2) -- (2,3.8)
(2.2,4) -- (3.8,4)

(4,4.2) |- (4.8,4.8) |- (4.2,4)
(4,3.8) |- (3.4,2.8)
(3.2,2.6) |- (2,2)
(2.8,2.8) rectangle (3.2,3.2)
(3.6,1.2) rectangle (4.8,2.4);

\draw[black,xshift=60mm] (1.8,4) -- (2,4.2);
\draw[red,xshift=60mm] (2,3.8) -- (2.2,4)
(3.2,2.6) -- (3.4,2.8)
(3.8,4) -- (4,4.2)
(4,3.8) -- (4.2,4)
;

\draw[red,xshift=125mm] (3,2.8) -- (3.2,3);
\draw[-to,xshift=125mm] (0.8,0.8) -- (0.8,2.4);
\draw[-to,xshift=125mm] (2,4.2) |- (3.6,5.2);
\draw[-to,xshift=125mm] (5.2,5.2) -- (5.2,3);
\draw[-to,xshift=125mm] (5.2,0.8) -- (3,0.8);
\draw[-to,xshift=125mm] (3.6,1.2) -- (4.2,1.2);
\draw[-to,xshift=125mm] (4.8,1.2) -- (4.8,1.8);
\draw[-to,xshift=125mm] (4.8,2.4) -- (4.2,2.4);
\draw[-to,xshift=125mm] (3.6,2.4) -- (3.6,1.8);
\draw[-to,xshift=125mm] (2,3.8) -- (2,2.9);
\draw[-to,xshift=125mm] (2,2) -- (2.6,2);
\draw[-to,xshift=125mm] (4,3.2) -- (4,3.5);
\draw[-to,xshift=125mm] (4.8,4.8) -- (4.4,4.8);
\draw[-to,xshift=125mm] (3.8,4) -- (3,4);
\draw[-to,xshift=125mm] (2.8,2.8) -- (2.8,3.2);
\draw[-to,xshift=125mm] (2.8,3.2) -- (3.2,3.2);
\draw[black!5, step=4mm, xshift=125mm] (0,0) grid (6,6);
\draw[black,xshift=125mm] (0.8,0.8) -- (5.2,0.8) -- (5.2,5.2) -- (2,5.2) -- (2,4.2)
(2,4.2) -- (1.8,4)
(1.8,4) -| (0.8,0.8)
(2,2) -- (2,3.8)
(2.2,4) -- (3.8,4)

(4,4.2) |- (4.8,4.8) |- (4.2,4)
(4,3.8) |- (3.4,2.8)
(3.2,2.6) |- (2,2)
(2.8,2.8) |- (3.2,3.2) -- (3.2,3)
(3,2.8) -- (2.8,2.8)
(3.6,1.2) rectangle (4.8,2.4);
\draw[black,xshift=125mm] (1.8,4) -- (2,4.2)
(2,3.8) -- (2.2,4)
(3.2,2.6) -- (3.4,2.8)
(3.8,4) -- (4,4.2)
(4,3.8) -- (4.2,4)
;

\draw[-to,yshift=-80mm,xshift=-70mm] (0.8,0.8) -- (0.8,2.4);
\draw[-to,yshift=-80mm,xshift=-70mm] (2,4.2) |- (3.6,5.2);
\draw[-to,yshift=-80mm,xshift=-70mm] (5.2,5.2) -- (5.2,3);
\draw[-to,yshift=-80mm,xshift=-70mm] (5.2,0.8) -- (3,0.8);
\draw[-to,yshift=-80mm,xshift=-70mm] (3.6,1.2) -- (4.2,1.2);
\draw[-to,yshift=-80mm,xshift=-70mm] (4.8,1.2) -- (4.8,1.8);
\draw[-to,yshift=-80mm,xshift=-70mm] (4.8,2.4) -- (4.2,2.4);
\draw[-to,yshift=-80mm,xshift=-70mm] (3.6,2.4) -- (3.6,1.8);
\draw[-to,yshift=-80mm,xshift=-70mm] (2,3.8) -- (2,2.9);
\draw[-to,yshift=-80mm,xshift=-70mm] (2,2) -- (2.6,2);
\draw[-to,yshift=-80mm,xshift=-70mm] (4,3.2) -- (4,3.5);
\draw[-to,yshift=-80mm,xshift=-70mm] (4.8,4.8) -- (4.4,4.8);
\draw[-to,yshift=-80mm,xshift=-70mm] (3.8,4) -- (3,4);
\draw[-to,yshift=-80mm,xshift=-70mm] (2.8,2.8) -- (2.8,3.2);
\draw[-to,yshift=-80mm,xshift=-70mm] (2.8,3.2) -- (3.2,3.2);
\draw[black!5, step=4mm, yshift=-80mm,xshift=-70mm] (0,0) grid (6,6);
\draw[black,yshift=-80mm,xshift=-70mm] (0.8,0.8) -- (5.2,0.8) -- (5.2,5.2) -- (2,5.2) -- (2,4.2)
(2,4.2) -- (1.8,4)
(1.8,4) -| (0.8,0.8)
(2,2) -- (2,3.8)
(2.2,4) -- (3.8,4)

(4,4.2) |- (4.8,4.8) |- (4.2,4)
(4,3.8) |- (3.4,2.8)
(3.2,2.6) |- (2,2)
(2.8,2.8) |- (3.2,3.2) -- (3.2,3)
(3,2.8) -- (2.8,2.8)
(3.6,1.2) rectangle (4.8,2.4);
\draw[black,yshift=-80mm,xshift=-70mm] (1.8,4) -- (2,4.2)
(2,3.8) -- (2.2,4)
(3.2,2.6) -- (3.4,2.8)
(3.8,4) -- (4,4.2)
(4,3.8) -- (4.2,4)
(3,2.8) -- (3.2,3);

\draw[-to,yshift=-80mm,xshift=-5mm] (0.8,0.8) -- (0.8,2.4);
\draw[-to,yshift=-80mm,xshift=-5mm] (2,4.2) |- (3.6,5.2);
\draw[-to,yshift=-80mm,xshift=-5mm] (5.2,5.2) -- (5.2,3);
\draw[-to,yshift=-80mm,xshift=-5mm] (5.2,0.8) -- (3,0.8);
\draw[-to,yshift=-80mm,xshift=-5mm] (3.6,1.2) -- (4.2,1.2);
\draw[-to,yshift=-80mm,xshift=-5mm] (4.8,1.2) -- (4.8,1.8);
\draw[-to,yshift=-80mm,xshift=-5mm] (4.8,2.4) -- (4.2,2.4);
\draw[-to,yshift=-80mm,xshift=-5mm] (3.6,2.4) -- (3.6,1.8);
\draw[-to,yshift=-80mm,xshift=-5mm] (2,3.8) -- (2,2.9);
\draw[-to,yshift=-80mm,xshift=-5mm] (2,2) -- (2.6,2);
\draw[-to,yshift=-80mm,xshift=-5mm] (4,3.2) -- (4,3.5);
\draw[-to,yshift=-80mm,xshift=-5mm] (4.8,4.8) -- (4.4,4.8);
\draw[-to,yshift=-80mm,xshift=-5mm] (3.8,4) -- (3,4);
\draw[-to,yshift=-80mm,xshift=-5mm] (2.8,2.8) -- (2.8,3.2);
\draw[-to,yshift=-80mm,xshift=-5mm] (2.8,3.2) -- (3.2,3.2);

\draw[black!5, step=4mm, yshift=-80mm,xshift=-5mm] (0,0) grid (6,6);
\draw[black,yshift=-80mm,xshift=-5mm] (0.8,0.8) -- (5.2,0.8) -- (5.2,5.2) -- (2,5.2) -- (2,4.2)
(1.8,4) -| (0.8,0.8)
(2,2) -- (2,3.8)
(2.2,4) -- (3.8,4)
(4,4.2) |- (4.8,4.8) |- (4.2,4)
(4,3.8) |- (3.4,2.8)
(3.2,2.6) |- (2,2)
(2.8,2.8) |- (3.2,3.2) -- (3.2,3)
(3,2.8) -- (2.8,2.8)
(3.6,1.2) rectangle (4.8,2.4);
\draw[black,yshift=-80mm,xshift=-5mm] 
(3.2,2.6) -- (3.4,2.8)
(3.8,4) -- (4,4.2)
(4,3.8) -- (4.2,4)
(3,2.8) -- (3.2,3);
\draw[red,yshift=-80mm,xshift=-5mm] (2,3.8) -- (1.8,4) (2.2,4) -- (2,4.2);

\draw[-to,yshift=-80mm,xshift=60mm] (0.8,0.8) -- (0.8,2.4);
\draw[-to,yshift=-80mm,xshift=60mm] (2,4.2) |- (3.6,5.2);
\draw[-to,yshift=-80mm,xshift=60mm] (5.2,5.2) -- (5.2,3);
\draw[-to,yshift=-80mm,xshift=60mm] (5.2,0.8) -- (3,0.8);
\draw[-to,yshift=-80mm,xshift=60mm] (3.6,1.2) -- (4.2,1.2);
\draw[-to,yshift=-80mm,xshift=60mm] (4.8,1.2) -- (4.8,1.8);
\draw[-to,yshift=-80mm,xshift=60mm] (4.8,2.4) -- (4.2,2.4);
\draw[-to,yshift=-80mm,xshift=60mm] (3.6,2.4) -- (3.6,1.8);
\draw[-to,yshift=-80mm,xshift=60mm] (2,3.8) -- (2,2.9);
\draw[-to,yshift=-80mm,xshift=60mm] (2,2) -- (2.6,2);
\draw[-to,yshift=-80mm,xshift=60mm] (4,3.2) -- (4,3.5);
\draw[-to,yshift=-80mm,xshift=60mm] (4.8,4.8) -- (4.4,4.8);
\draw[-to,yshift=-80mm,xshift=60mm] (3.8,4) -- (3,4);
\draw[-to,yshift=-80mm,xshift=60mm] (2.8,2.8) -- (2.8,3.2);
\draw[-to,yshift=-80mm,xshift=60mm] (2.8,3.2) -- (3.2,3.2);
\draw[] 
(-40mm,-8mm) node {$(i)$}
(25mm,-8mm) node {$(ii)$}
(90mm,-8mm) node {$(iii)$}
(155mm,-8mm) node {$(iv)$}
(-40mm,-88mm) node {$(iv)$}
(25mm,-88mm) node {$(v)$}
(90mm,-88mm) node {$(vi)$}
;
\draw[black!5, step=4mm, yshift=-80mm,xshift=60mm] (0,0) grid (6,6);
\draw[black,yshift=-80mm,xshift=60mm] (0.8,0.8) -- (5.2,0.8) -- (5.2,5.2) -- (2,5.2) -- (2,4.2)
(1.8,4) -| (0.8,0.8)
(2,2) -- (2,3.8)
(2.2,4) -- (3.8,4)
(4,4.2) |- (4.8,4.8) |- (4.2,4)
(4,3.8) |- (3.4,2.8)
(3.2,2.6) |- (2,2)
(2.8,2.8) |- (3.2,3.2) -- (3.2,3)
(3,2.8) -- (2.8,2.8)
(3.6,1.2) rectangle (4.8,2.4);
\draw[black,yshift=-80mm,xshift=60mm] 
(3.8,4) -- (4,4.2)
(4,3.8) -- (4.2,4)
(2,3.8) -- (1.8,4) 
(2.2,4) -- (2,4.2)
;
\draw[red,yshift=-80mm,xshift=60mm]
(3.2,2.6) -- (3,2.8)
(3.4,2.8) -- (3.2,3)
;
\end{tikzpicture}
    \caption{The first row $(i)\to(ii)\to(iii)\to(iv)$ shows a peeling procedure. The second row $(iv)\to(v)\to(vi)$ shows a concatenation procedure. The $45^\circ$- and $135^\circ$-edges are purely symbolic, indicating the direction of turns taken by a loop. Changes in turns are marked in red.}
    \label{fig:peelconc}
\end{figure}

We start by setting $L_0^\peel =\emptyset$ and $\kappa^{0}=\kappa$. Clearly, $\partial_K \kappa^0 =\emptyset$

Given $L_i^\peel$ and $\kappa^{i}$ satisfying $\partial_K \kappa^i =\emptyset$, we define the dual configuration $(\kappa^{i})^*$ in the same way as before.
We set
\begin{align*}
    E = \Ll\{e\in \EE(\kappa^{i}):\quad \exists e' \in \EE(\partial K^*) :\: e^*\stackrel{(\kappa^{i})^*}{\longleftrightarrow} e'\Rr\},
\end{align*}
where we view $\partial K^*$ as a loop.
Heuristically, $E$ contains exactly the open edges in $\kappa^{i}$ that are ``exposed'' to $\partial K^*$. We peel $E$ into loops.

Let us fix an index on the edges in $E$. We construct a loop $l$ by the following algorithm:
\begin{enumerate}
    \item Start from the edge $e \in E$ with the smallest index and let $x_0$, $x_1$ be the endpoints of $e$.
    \item Assuming that $x_j$ is chosen and $x_{j-1}x_j\in E$, we choose $x_{j+1}$ as follows. Let $\gamma$ be a path in $(\kappa)^*$ such that $(x_{j-1}x_j)^*\stackrel{\gamma}{\longleftrightarrow} e'$ for some $e'\in \EE(\partial K^*)$. Due to $\partial_K \kappa^i =\emptyset$, the number $n_j$ of open edges in $\kappa^i$ with endpoint $x_j$ is even, so there are two cases (see Figure~\ref{fig:picpeel}):
    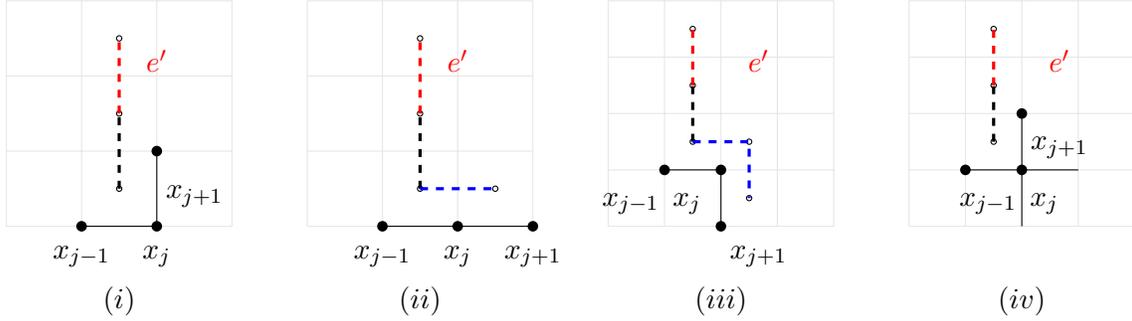
\begin{figure}
    \centering
    \begin{tikzpicture}
\draw[black!10,step=10mm] (0,0) grid (3,3);
\fill[](1,0) circle (2pt) (2,0) circle (2pt) (2,1) circle (2pt);
\draw[] (1,0) -- (2,0) (2,0) -- (2,1)
(1.5,0.5) circle (1pt)
(1.5,1.5) circle (1pt)
(1.5,2.5) circle (1pt)
(1,-0.4) node {$x_{j-1}$}
(2,-0.4) node {$x_{j}$}
(2.5,0.4) node {$x_{j+1}$}
(1.5,-1) node {$(i)$}
;
\draw[dashed, very thick] (1.5,0.5) -- (1.5,1.5);
\draw[red,dashed,very thick] (1.5,1.5) -- (1.5,2.5);
\draw[red] (2,2.2) node {$e'$};

\draw[black!10,step=10mm,xshift=40mm] (0,0) grid (3,3);
\fill[xshift=40mm](1,0) circle (2pt) (2,0) circle (2pt) (3,0) circle (2pt);
\draw[xshift=40mm] (1,0) -- (2,0) (2,0) -- (3,0)
(1.5,0.5) circle (1pt)
(1.5,1.5) circle (1pt)
(1.5,2.5) circle (1pt)
(2.5,0.5) circle (1pt)
(1,-0.4) node {$x_{j-1}$}
(2,-0.4) node {$x_{j}$}
(3,-0.4) node {$x_{j+1}$}
(1.5,-1) node {$(ii)$}
;
\draw[dashed,xshift=40mm,very thick] (1.5,0.5) -- (1.5,1.5);
\draw[red,dashed,very thick,xshift=40mm] (1.5,1.5) -- (1.5,2.5);
\draw[blue,dashed,xshift=40mm,very thick] (1.5,0.5) -- (2.5,0.5);
\draw[red,xshift=40mm] (2,2.2) node {$e'$};

\draw[black!10,step=7.5mm,xshift=80mm] (0,0) grid (3,3);
\fill[xshift=80mm](0.75,0.75) circle (2pt) (1.5,0.75) circle (2pt) (1.5,0) circle (2pt);
\draw[xshift=80mm] (0.75,0.75) -- (1.5,0.75) -- (1.5,0)
(1.125,1.125) circle (1pt)
(1.125,1.875) circle (1pt)
(1.125,2.625) circle (1pt)
(1.875,1.125) circle (1pt)
(1.875,0.375) circle (1pt)
(0.3,0.3) node {$x_{j-1}$}
(1.05,0.3) node {$x_{j}$}
(2,-0.4) node {$x_{j+1}$}
(1.5,-1) node {$(iii)$}
;
\draw[dashed,xshift=80mm,very thick] (1.125,1.125) -- (1.125,1.875);
\draw[red,dashed,very thick,xshift=80mm] (1.125,1.875) -- (1.125,2.625);
\draw[blue,dashed,xshift=80mm,very thick] (1.125,1.125) -- (1.875,1.125) -- (1.875,0.375);
\draw[red,xshift=80mm] (2,2.2) node {$e'$};

\draw[black!10,step=7.5mm,xshift=120mm] (0,0) grid (3,3);
\fill[xshift=120mm](0.75,0.75) circle (2pt) (1.5,0.75) circle (2pt) (1.5,1.5) circle (2pt);
\draw[xshift=120mm] (0.75,0.75) -- (1.5,0.75) -- (1.5,0)
(1.5,1.5) -- (1.5,0.75) -- (2.25,0.75)
(1.125,1.125) circle (1pt)
(1.125,1.875) circle (1pt)
(1.125,2.625) circle (1pt)

(1.05,0.3) node {$x_{j-1}$}
(1.8,0.3) node {$x_{j}$}
(2,1.05) node {$x_{j+1}$}
(1.5,-1) node {$(iv)$}
;
\draw[dashed,xshift=120mm,very thick] (1.125,1.125) -- (1.125,1.875);
\draw[red,dashed,very thick,xshift=120mm] (1.125,1.875) -- (1.125,2.625);

\draw[red,xshift=120mm] (2,2.2) node {$e'$};
\end{tikzpicture}
    \caption{Case~\eqref{i.case:n_j=2} is illustrated in sub-cases $(i), (ii), (iii)$. 
    Case~\eqref{i.case:n_j=4} is illustrated in $(iv)$. The black dashed edges are on $\gamma$. The blue dashed edges are added to form $\bar \gamma$.}
    \label{fig:picpeel}
\end{figure}
    \begin{enumerate}
        \item \label{i.case:n_j=2}
        Case: $n_j=2$. We set $x_{j+1}$ to be the endpoint other than $x_{j-1}x_j$. Since the dual of the remaining edges are open in $(\kappa^i)^*$, adding to $\gamma$ some of these dual edges forms $\bar \gamma$ such that $(x_{j}x_{j+1})^*\stackrel{\bar\gamma}{\longleftrightarrow} e'$. Therefore, $x_jx_{j+1} \in E$.
        \item \label{i.case:n_j=4} Case: $n_j = 4$. The dual of one edge $\tilde e$ shares an endpoint $y$ with $(x_{j-1}x_j)^*$ and $y$ is connected by $\gamma$ to $e'$. We take $x_{j+1}$ to be the other endpoint of $\tilde e$. Hence, $x_jx_{j+1}\in E$ as it satisfies $(x_{j}x_{j+1})^*\stackrel{\gamma}{\longleftrightarrow} e'$.
    \end{enumerate}
    \item If $x_{j+1} =x_0$, then terminate and return the loop $l=(x_0,x_1,\dots, x_j)$; otherwise, repeat the second step with $j+1$ substituted for $j$.
\end{enumerate}
This gives a loop $l$ with edges in $E$ and we put $l$ in $L^\peel_{i+1}$. Then, we repeat the same algorithm starting at a new edge $e \in E$ that is not on any loop peeled so far. This eventually exhausts all edges in $E$ and we obtain the collection $L^\peel_{i+1}$.

We define $\kappa^{i+1}$ by setting $\kappa^{i+1}_e = 0$ if $\kappa^i_e =0$ or $e \in \EE(l)$ for some $l\in L^\peel_{i+1}$; and setting $\kappa^{i+1}_e = 1$ otherwise. By this construction, for every $x\in \VV(K)$, the number of open edges, with endpoint $x$, removed from $\kappa^i$ is even. Therefore, by the definition of sources in~\eqref{e.source} and the assumption that $\partial_K \kappa^i =\emptyset$, we have that $\partial_K \kappa^{i+1} = \emptyset$ which completes the induction.

We have constructed a sequence $(L^\peel_i)_{i=1}^\infty$ of loop collections by induction. 
Set 
\begin{align}\label{e.N=inf...}
    N = \inf\Ll\{n\in \N:\: L^\peel_i=\emptyset,\,\forall i>n\Rr\},
\end{align}
which is clearly finite.

\begin{lemma}\label{l.L^peel}
The following holds:
\begin{enumerate}
    \item \label{i.L^peel_edges} the collection $\cup_{i=1}^N L^\peel_i$ satisfies the property in~\eqref{e.kappa_e=1<=>...};
    \item \label{i.L^peel_non-sefl-cro} $\cup_{i=1}^NL^\peel_i$ consists of weakly-disjoint weakly simple loops;
    
    \item \label{i.L^peel_supset_L^most} $L^\most_i\subset L^\peel_i\subset L_i$, for each $i\geq 1$.
\end{enumerate}
\end{lemma}

\begin{proof}
Part~\eqref{i.L^peel_edges} follows directly from the construction of $L^\peel_i$ and $\kappa^i$, since the loops constructed are from edges in $\kappa$.

By the construction, the loops in $L^\peel_i$ are weakly-disjoint for each $i$ and the loops in different $L^\peel_i$'s do not share edges. Step~\eqref{i.case:n_j=4} ensures that the loops in $L^\peel_i$ are weakly simple, verifying Part~\eqref{i.L^peel_non-sefl-cro}.

We now show the second inclusion in Part~\eqref{i.L^peel_supset_L^most}. The definition of $E$ gives $L^\peel_1=L_1$. Assume the inclusion up to index $i$. Notice that every $l\in L^\peel_{i+1}$ is connected to $\partial K^*$ by open edges in $\kappa^*_{i}$. Therefore, $l$ is connected by edges in $\kappa^*$ to some loop in $L^\peel_{i}$ in $\kappa$. By the induction assumption, we conclude that $l$ is of level-$(i+1)$ and thus $L^\peel_{i+1}\subset L_{i+1}$.

Then, we turn to the first inclusion in Part~\eqref{i.L^peel_supset_L^most}. By the definition of $K^*$, the edges of all the loops in $L^\most_1$ are in $E$. Then, we peel the loops in $L^\most_1$ as in Step~\eqref{i.case:n_j=4} of the algorithm. Assume the inclusion up to index $i$ for some $i\geq 1$. After removing edges in $\kappa$ to get $\kappa^i$, we see that the edges on the loops in $L^\most_{i+1}$ are not in $E$. As before, we peel them into $L^\peel_{i+1}$.
Therefore, the verification of Part~\eqref{i.L^peel_supset_L^most} is complete.
\end{proof}

\subsubsection{Concatenation of loops}\label{s.loop_construct_3}

First, we describe a concatenation of two weakly disjoint weakly simple loops, $l=(x_0,\dots,x_n)$ and $l'=(x'_0,\dots,x'_{n'})$, satisfying the following condition:
\begin{align}\label{e.concat_cond}
\begin{cases}
    \text{there is } z=x_k= x'_{k'} \in \VV(l)\cap \VV(l')\text{ such that} 
    \\
    (x_{k-1},z,x'_{k'-1})\text{ and } (x_{k+1},z,x'_{k'+1}) \text{ are on straight lines.}
\end{cases}
\end{align}
We describe a trajectory that concatenates $l$ and $l'$ into one loop: 
\begin{enumerate}
    \item start from $x_0$ and move along $l$ in order until hitting a common vertex $x_k=x'_{k'}\in \VV(l)\cap \VV(l')$;
    \item jump to $x'_{k'+1}$ and then move along $l'$ in order until reaching $x'_{k'}=x_k$;
    \item jump back to $x_{k+1}$ and then move along $l$ in order until reaching $x_0$.
\end{enumerate}

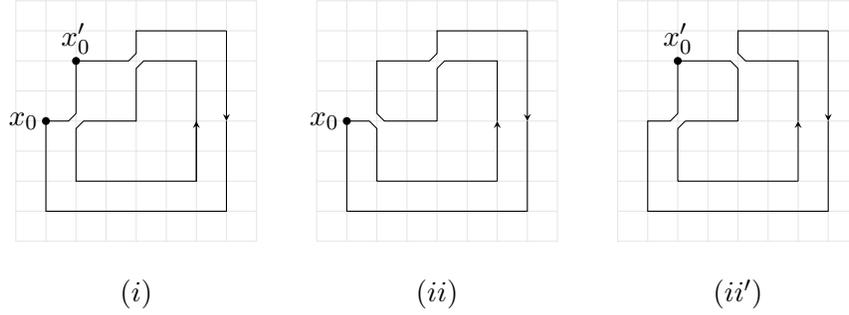
\begin{figure}
    \centering
    \begin{tikzpicture}
\draw[black!10, step=4mm] (0.4,0.4) grid (3.6,3.6)
;
\draw[] (1.2,1.9) -- (1.2,1.2) -| (2.8,2.8) -- (2.1,2.8) (1.9,2.8) -| (1.2,2.1)
(1.1,2) -- (1.2,2.1)
(1.2,1.9) -- (1.3,2)
(2,2.9) |- (3.2,3.2) 

(3.2,2) |- (0.8,0.8) |- (1.1,2)
(1.3,2) -| (2,2.7)
(1.9,2.8) -- (2,2.9)
(2,2.7) -- (2.1,2.8)
(0.5,2) node{$x_0$}
(1.2,3.1) node{$x_0'$}
(2,-0.3) node{$(i)$}
;
\fill[](0.8,2) circle (1.5pt)
(1.2,2.8) circle (1.5pt)
;

\draw[black!10, step=4mm,xshift=40mm] (0.4,0.4) grid (3.6,3.6);

\fill[xshift=40mm](0.8,2) circle (1.5pt);
\draw[xshift=40mm]
(0.5,2) node{$x_0$}
(1.2,1.2) -- (2.8,1.2)
(2.8,2) -- (2.8,2.8) -- (2.1,2.8)
(2.1,2.8) -- (2,2.7)
(1.9,2.8) -- (2,2.9)
(1.9,2.8) -| (1.2,2.1)

(1.2,1.9) -- (1.2,1.2)

(2,2) -- (2,2.7) 

(2,2.9) |- (3.2,3.2) 

(3.2,2) |- (0.8,0.8) |- (1.1,2)
(1.2,1.9) -- (1.1,2)
(1.3,2) -- (1.2,2.1)
(1.3,2) -- (2,2)
(2,-0.3) node{$(ii)$}
;
\draw[black!10, step=4mm,xshift=80mm] (0.4,0.4) grid (3.6,3.6)
;
\draw[-stealth,very thin](2.8,1.2) -- (2.8,2);
\draw[-stealth,very thin](3.2,3.2) -- (3.2,2);
\draw[-stealth,very thin, xshift=40mm](2.8,1.2) -- (2.8,2);
\draw[-stealth,very thin,xshift=40mm](3.2,3.2) -- (3.2,2);
\draw[xshift=80mm]
(1.2,3.1) node{$x_0'$}
(1.2,1.2) -- (2.8,1.2)
(2.8,2) -- (2.8,2.8) -- (2.1,2.8)
(2,2.9) -- (2.1,2.8)
(1.9,2.8) -- (2,2.7)
(1.9,2.8) -| (1.2,2.1)

(1.2,1.9) -- (1.2,1.2)

(2,2) -- (2,2.7) 
(1.1,2) -- (1.2,2.1)
(1.2,1.9) -- (1.3,2)
(2,2.9) |- (3.2,3.2) 

(3.2,2) |- (0.8,0.8) |- (1.1,2)

(1.3,2) -- (2,2)
(2,-0.3) node{$(ii')$}
;
\fill[xshift=80mm](1.2,2.8) circle (1.5pt)
;
\draw[-stealth,very thin,xshift=80mm](2.8,1.2) -- (2.8,2);
\draw[-stealth,very thin,xshift=80mm](3.2,3.2) -- (3.2,2);
\end{tikzpicture}
    \caption{Different starting vertices $x_0$ and $x'_0$ in $(i)$ give two different concatenations: $(i)\to(ii)$ and $(i)\to(ii')$.}
    \label{fig:picdifconc}
\end{figure}

We denote by $l\oplus l'$ a \textbf{concatenation} of $l$ and $l'$, which is not unique since the procedure depends on the starting vertex $x_0$ (see the second row in Figure~\ref{fig:picdifconc}). 
Notice that $l\oplus l'$ traverses all vertices in $l$ and $l'$ and preserves the order on each edge in $l$ and $l'$, which implies that
\begin{align}\label{e.l'insidel=>orientation}
    \text{$l'$ is inside $l$}\quad\Longrightarrow\quad\text{$l\oplus l'$ has the same orientation as $l$}.
\end{align}
Clearly, that $l\oplus l'$ is weakly simple.

We use this procedure to concatenate loops in $\cup_{i=1}^N L^\peel_i$ (see Figure~\ref{fig:peelconc}). For every $l^i\in L_i^\most\subset L_i^\peel$ (see Lemma~\ref{l.L^peel}~\eqref{i.L^peel_supset_L^most}), we inductively construct the following loops:
\begin{enumerate}

    \item We start by setting $l^{i,i}=l^i$.
    \item Assuming that we have constructed $l^{i,k}$ for some $k\in\{i,i+1,\dots,N-1\}$, we let $\{l'_j\}_{1\leq j\leq m_{k+1}}$, for some $m_{k+1}\in\N$, be loops in $L_{k+1}^\peel$ satisfying $\VV(l'_j)\cap \VV(l^{i,k})\neq \emptyset$.
    \item \label{i.l=(...(l+} We set $l^{i,k+1} =\Big(\cdots\big((l^{i,k}\oplus l'_1)\oplus l'_2\big)\cdots\oplus l'_{m_{k+1}}\Big)$.
    \item Repeat the second step with $k+1$ substituted for $k$ until we get $l^{i,N}$.
\end{enumerate}
Let us comment on Step~\eqref{i.l=(...(l+}. Step~\eqref{i.l=(...(l+} requires the weak disjointness in Lemma~\ref{l.L^peel}~\eqref{i.L^peel_non-sefl-cro} and the choice~\eqref{e.orientation} for loops in $L^\peel_i$, which guarantees condition~\eqref{e.concat_cond}. By the weak disjointness, $l'_{j+1}$ is inside the big loop concatenated $l'_j$. This along with~\eqref{e.l'insidel=>orientation} implies that $l^{i,k}$ has the same orientation as $l^{i,k+1}$. Iteratively, $l^{i}$ has the same orientation as $l^{i,N}$. Since $\oplus$ is not associative, we must concatenate by order.

For each $i$, we define a new collection
\begin{align*}
    L_i^\conc = \Ll\{l^{i,N}:l^i\in L_i^\most\Rr\},
\end{align*}
of which we record some properties.

\begin{lemma}\label{l.L^conc}
The following holds:
\begin{enumerate}
    \item \label{i.L^conc_edge} the collection $\cup_{i=1}^NL^\conc_i$ satisfies the property in~\eqref{e.kappa_e=1<=>...};
    \item \label{i.L^conc_inside} $l^{i,N}$ is inside $l^i$ for every $l^i\in L_i^\most$;
    \item \label{i.L^conc_diam} 
    $\diam(l^{i,N}) = \diam(l^i)$ for every $l^i\in L_i^\most$;
    \item \label{i.L^conc_disjoint_non-self-cro} $\cup_{i=1}^NL^\conc_i$ consists of disjoint weakly simple loops;
    \item \label{i.L^conc_biject} the map $l^i\mapsto l^{i,N}$ is a bijection from $L_i^\most$ to $L_i^\conc$.
\end{enumerate}
\end{lemma}

We remark that the bijectivity in Part~\eqref{i.L^conc_biject} is used to show a loop decomposition of the critical Ising interface converges to $\cle(3)$ on the basis of the results in \cite{benoist2019scaling}, where the convergence is proven for the leftmost Ising loops.
\begin{proof}
In the concatenation procedure, no edge is added or deleted from the loops in $\cup_{i=1}^N L^\peel_i$. This along with Lemma~\ref{l.L^peel}~\eqref{i.L^peel_edges} yields Part~\eqref{i.L^conc_edge}.

The loops concatenated into $l^{i,i}$ are those sharing vertices with $l^i$. Since $l^i$ is outmost, these loops must be inside $l^i$. By the definition of outmost loops, $l^i$ is surrounded by a loop $\gamma^i\in L^*$. Iteratively, we see that all the ensuing concatenations happen inside $l^i$, which proves Part~\eqref{i.L^conc_inside}. Part~\eqref{i.L^conc_diam} also follows in a straightforward way.

Since $\gamma^i$ are disjoint for different $l^i$ and $l^{i,N}$ are inside $\gamma^i$, we have that $l^{i,N}$ are disjoint. The facts that the loops in $L^\peel_{k+1}$ are weakly simple (Lemma~\ref{l.L^peel}~\eqref{i.L^peel_non-sefl-cro}) and $\oplus$ preserves this property prove Part~\eqref{i.L^conc_disjoint_non-self-cro}.

The definition of $L^\conc_i$ ensures that the map $l^i\mapsto l^{i,N}$ is surjective.
Part~\eqref{i.L^conc_edge} and the disjointness of $l^i$  imply the injectivity, proving Part~\eqref{i.L^conc_biject}.
\end{proof}

Recall from Lemma~\ref{l.L^peel}~\eqref{i.L^peel_edges} that the loops in $L^\peel_i$ are level-$i$ and that levels are defined relative to the loops in $L^\mathrm{all}$.
After concatenating, the loops in $L^\conc_i$ are of level-$i$ with respect to the loops in $L^\mathrm{all}$. However, the level with respect to the loops in $\cup_{i=1}^NL^\conc_i$ is less than or equal to $i$.

\subsubsection{Conclusion of constructions}

We define
\begin{align}\label{e.kappa^loop}
    \kappa^\lop = \cup_{n=1}^N L_i^\conc.
\end{align} 

\begin{lemma}\label{l.kappa^lop_loop_decomposition}
If $\kappa$ is a percolation configuration on $K$ satisfying $\partial_K \kappa=\emptyset$, then the collection $\kappa^\lop$ defined in~\eqref{e.kappa^loop} is a loop decomposition of $\kappa$ on $K$.
\end{lemma}

\begin{proof}
This is a consequence of Lemma~\ref{l.L^conc}~\eqref{i.L^conc_edge} and~\eqref{i.L^conc_disjoint_non-self-cro}.
\end{proof}

We emphasize that a loop decomposition of $\kappa$ is in general not unique because self-touching loops can be traversed in different ways as demonstrated in Figure~\ref{fig:picdifconc}.

\begin{lemma}\label{l.lop_unique}
Up to the identification by $=_{\R^2}$, $\kappa^\lop$ is unique.
\end{lemma}

\begin{proof}
Let $L$ and $L'$ be two loop decompositions of $\kappa$. Let $l \in L$ and set $L'(l) = \{l'\in L':\EE(l')\cap \EE(l) \neq \emptyset\}$. By the property~\eqref{e.kappa_e=1<=>...}, for every $e\in \EE(l)$, we have $ \kappa_e =1$ and thus $e\in \EE(l')$ for some $l'\in L'$. Hence, $L'(l)$ is not empty and we further deduce that $l \subset \cup_{l'\in L'(l)}l'$. Since the loops in $L'$ are disjoint, the loops in $L'(l)$ are disjoint when viewed as subsets of $\R^2$. As $l$ is connected, $L'(l)$ must be a singleton and we denote the element by $l'$, satisfying $l\subset l'$. The other direction follows similarly using disjointedness and connectedness and thus $l=_{\R^2} l'$. Hence, we conclude that $L\subset_{\R^2} L'$ and the other inclusion can be deduced in the same way.
\end{proof}

\begin{remark}\label{r.extend_varsigma}
Let $R$ be a subgraph of $\D_\delta$ and let $\varsigma\in\{0,1\}^{\EE(R)}$. Suppose that $\partial_R \varsigma =\emptyset$. 
Let $\bar\varsigma$ be the trivial extension of $\varsigma$ (see~\eqref{e.trivial_ext}) and let $\bar\varsigma_{\D_\delta}$ be the restriction of $\bar\varsigma$ to $\D_\delta$ (see~\eqref{e.kappa_K}).
The definition of sources in~\eqref{e.source} implies that $\partial_{\D_\delta} \bar\varsigma_{\D_\delta}=\emptyset$. Denote by $\varsigma^\lop$ a loop decomposition of $\bar\varsigma_{\D_\delta}$ given by Lemma~\ref{l.kappa^lop_loop_decomposition}.  We view $\varsigma^\lop$ as the loop decomposition of $\varsigma$. Clearly, the loops in $\varsigma^\lop$ are in $R$.
\end{remark}

\begin{remark}\label{r.eta^lop}
For each $\delta$, under $\P^\emptyset_{\D_\delta}$, every realization $\eta_\delta$ satisfies $\partial_{\D_\delta} \eta_\delta=\emptyset$, which is clear from the definition of $\P^\emptyset_{\D_\delta}$ in~\eqref{e.P^emptyset_G}. By Lemma~\ref{l.kappa^lop_loop_decomposition}, every realization $\eta_\delta$ admits a loop decomposition. In Sections~\ref{s.compare_cross} and~\ref{s.id_limits}, we denote by $\eta^\lop_\delta$ some loop decomposition of $\eta_\delta$. While in the application to random current in Section~\ref{s.app_random_current}, we denote by $\eta^\lop_\delta$ the loop decomposition constructed in~\eqref{e.kappa^loop}, because we need the bijection in Lemma~\ref{l.L^conc}~\eqref{i.L^conc_biject} and convergence results in \cite{benoist2019scaling} to show the convergence of $\eta^\lop_\delta$ to $\cle(3)$.
\end{remark}

\section{Comparison of crossing probabilities}\label{s.compare_cross}

\subsection{A digression on exploration processes}
Instead of giving an algorithm, we adopt a static approach by describing successful outcomes when part of the configuration is explored. 
We work with a slightly more general setting as in Section~\ref{s.kappa_prop}.
Let $G$ be a finite graph and let $\kappa\in\{0,1\}^{\EE(G)}$ be a percolation configuration on $G$. Recall the notation~\eqref{e.kappa_K} for restriction $\kappa_R$ of $\kappa$ to a subgraph $R$.

Let $\mathcal{E}$ be a collection of subgraphs of $G$. We interpret each $R\in\mathcal{E}$ as the subgraph on which we explore (or reveal) the configuration $\kappa$.

For each $R\in\mathcal{E}$, we associate a set $S_R\subset \{0,1\}^{\EE(R)}$. We interpret $S_R$ as the collection of the states that are admissible after revealing $\kappa_R$. Indeed, in an algorithmic description of the exploration process, reaching any state in $S_R$ after exploring $R$ terminates the algorithm.

Writing $S = (S_R)_{R\in\mathcal{E}}$, we call $(\mathcal{E}, S)$ an \textbf{exploration process}. We associate with $(\mathcal{E}, S)$ the following event:
\begin{align}\label{e.B}
    B = \bigcup_{R\in\mathcal{E}}\{\kappa:\kappa_R \in S_R\}.
\end{align}
Let $\mathbf{P}$ be any probability measure on $\{0,1\}^{\EE(G)}$.
We call $(\mathcal{E},S)$ \textbf{efficient under $\mathbf{P}$} if 
\begin{align}\label{e.efficient}
    R\neq R' \quad \Longrightarrow\quad \mathbf{P}\Ll(\{\kappa:\kappa_R \in S_R\}\cap \{\kappa:\kappa_{R'} \in S_{R'}\}\Rr) = 0.
\end{align}
This means that explorations on different sets of edges (i.e.\ $R\neq R'$) will not double count the same configuration.
It is easy to see that $(\mathcal{E},S)$ is efficient if and only if $B$ is a disjoint union.

\begin{lemma}\label{l.eff_exp}
If $(\mathcal{E},S)$ is an exploration process efficient under $\mathbf{P}$and event $B$ is defined as in~\eqref{e.B}, then for every event~$A$,
\begin{align*}
    \mathbf{P}(A|B) = \frac{\sum_{R\in\mathcal{E}}\sum_{\varsigma\in S_R}\mathbf{P}(A|\kappa_R =\varsigma)\mathbf{P}(\kappa_R=\varsigma)}{\sum_{R\in\mathcal{E}}\sum_{\varsigma\in S_R}\mathbf{P}(\kappa_R=\varsigma)}.
\end{align*}
\end{lemma}
\begin{proof}
We have
\begin{align*}
    \sum_{R\in\mathcal{E}}\sum_{\varsigma\in S_R}\mathbf{P}(A|\kappa_R =\varsigma)\mathbf{P}(\kappa_R=\varsigma)
    &= \sum_{R\in\mathcal{E}}\sum_{\varsigma\in S_R}\mathbf{P}(A\cap \{\kappa_R =\varsigma\}) 
    \\
    &= \sum_{R\in\mathcal{E}}\mathbf{P}(A\cap \{\kappa_R \in S_R\}) = \mathbf{P}(A\cap B)
\end{align*}
where the last equality follows from the efficiency of the exploration process.
Replacing $A$ by the entire probability space, we see that the denominator is equal to $\mathbf{P}(B)$.
\end{proof}

This lemma implies that a bound on $\mathbf{P}(A|\kappa_R=\varsigma)$, uniformly over $\varsigma$ and $R$, also holds for $\mathbf{P}(A|B)$.

\subsection{Exploration of loops from outside}\label{s.exp_loop_from_outside}

Fix any $\delta>0$ and let $\gamma\subset \D_\delta$ be a $\delta$-discrete non-self-crossing loop.
Let $[\gamma]$ be the $\delta$-discrete disc with boundary $\gamma$.
For any $\delta$-discrete loop $l$ in $\D_\delta$, we define $l^+$ to be the subgraph of $\D_\delta$ with
\begin{align}\label{e.l^+=}
    \VV\Ll(l^+\Rr) = \VV\Ll(l\Rr),\qquad \EE\Ll(l^+\Rr) = \Ll\{e \in \EE(\D_\delta):\: \text{$e$ has an endpoint in $\VV(l)$} \Rr\}.
\end{align}
In words, $l^+$ is obtained from the graph of $l$ by adding edges with endpoints in $\VV\Ll(l\Rr)$ (see Figure~\ref{fig:picl+}).
For a collection $L$ of disjoint weakly simple loops, we define the following union of graphs
\begin{align}\label{e.U(L)=}
    U(L) = \Ll( \D_\delta \setminus [\gamma]\Rr)\cup\Ll(\cup_{l\in L:\, l\cap(\D_\delta\setminus[\gamma])\neq \emptyset}\,l^+\Rr).
\end{align}
Set
\begin{align}\label{e.E}
    \mathcal{E} = \Ll\{R\subset \D_\delta: \text{$R=U(L)$ for some $L$} \Rr\}.
\end{align}
For each $R\in\mathcal{E}$, set
\begin{align}\label{e.mathcal(D)_R}
    \mathcal{D}_R = \Ll\{\varsigma\in\{0,1\}^{\EE(R)}:\: \partial_R \varsigma =\emptyset\Rr\}.
\end{align}
For each $\varsigma\in \mathcal{D}_R$, we denote by $\varsigma^\lop$ the loop decomposition of $\varsigma$ given by Remark~\ref{r.extend_varsigma}.

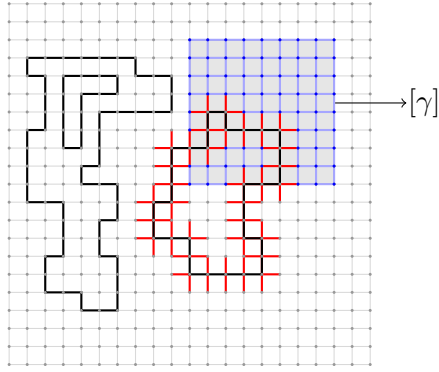
\begin{figure}
    \centering
    
\begin{tikzpicture}[scale=1.2]

\fill[black!10] (2,2) rectangle (3.6,3.6);
\draw[black!15, step=2mm] (0,0) grid (4,4);

\draw[black, thick]
(0.8,0.6) -| (1.2,1.2) -| (1,1.6) -| (1.2,2) -| (0.8,2.2) -| (1,2.8) -| (1.8,3.2) -| (1.4,3.4) -| (0.2,3.2) -| (0.4,2.6) -| (0.2,1.8) -| (0.6,1.2) -| (0.4,0.8) -| (0.8,0.6)
(0.6,2.4) -| (0.8,3) -| (1.2,3.2) -| (0.6,2.4);
 \draw[black, opacity=0.5, very thin] (2,2) rectangle (3.6,3.6);
 \draw[red, thick]
(1.4,1.6) -- (1.8,1.6) (1.4,1.4) -- (1.6,1.4) (1.4,1.8) -- (2,1.8) (1.6,2) -- (2,2) (1.6,2.2) -- (2,2.2) (1.6,2.4) -- (2,2.4) (2,2.6) -- (2.4,2.6) (2,2.8) -- (2.6,2.8) (2,1.4) -- (2.2,1.4) (1.8,1) -- (2,1) (1.8,1.2) -- (2.2,1.2)
(1.8,1.2) -- (1.8,1.6) (2,1.4) -- (2,1.6) (1.6,1.2) -- (1.6,2) (1.8, 1.6) -- (1.8, 2.6)
(2,2.2) -- (2,2.6) (2.2,2.2) -- (2.2,3) (2.4,2.4) -- (2.4,3) (2.6,2.4) -- (2.6,2.8) (2.8,2.4) -- (2.8,2.8) (3,1.8) -- (3,2.8) (2.8,1.8) -- (2.8,2.2) (2.6,0.8) -- (2.6,2.2) (2.4,0.8) -- (2.4,1.2) (2.2,0.8) -- (2.2,1.2) (2,0.8) -- (2,1) (2.8,0.8) -- (2.8,1.6)
(2.2,2.4) -- (2.4,2.4)
(2.8,1) -- (3,1) (2.6,1.2) -- (3,1.2) (2.4,1.4) -- (3,1.4) (2.4,1.6) -- (2.8,1.6) (2.4,1.8) -- (2.8,1.8) (2.4,2) -- (3.2,2) (2.8,2.2) -- (3.2,2.2) (2.8,2.4) -- (3.2,2.4) (3,2.6) -- (3.2,2.6);
\draw[black,thick] (1.6,1.4) |- (1.8,1.8) |- (2.2,2.4) |- (2.4,2.8) |- (3,2.6) |- (2.6,2) |- (2.8,1.4) |- (2,1) |- (1.6,1.4);
\draw[->] (3.6,2.9) -- (4.4,2.9);
\draw[] (4.6,2.9) node{$[\gamma]$};

\draw[blue!35, thick] (2,3.6) -- (3.6,3.6)
(2,3.4) -- (3.6,3.4)
(2,3.2) -- (3.6,3.2)
(2,3) -- (3.6,3)
(2.6,2.8) -- (3.6,2.8)
(3.2,2.6) -- (3.6,2.6) (3.2,2.4) -- (3.6,2.4) (3.2,2.2) -- (3.6,2.2) (3.2,2) -- (3.6,2)
(2,2) -- (2.4,2) (2,2.2) -- (2.8,2.2) (2.4,2.4) -- (2.8,2.4)
(2,2) -- (2,2.2) (2.2,2) -- (2.2,2.2) (2.4,2) -- (2.4,2.4) (2.6,2.2) -- (2.6,2.4) (2.8,2.2) -- (2.8,2.4)
(2,2.6) -- (2,3.6) (2.2,3) -- (2.2,3.6) (2.4,3) -- (2.4,3.6) (2.6,2.8) -- (2.6,3.6) (2.8,2.8) -- (2.8,3.6) (3,2.8) -- (3,3.6) (3.2,2) -- (3.2,3.6) (3.4,2) -- (3.4,3.6) (3.6,2) -- (3.6,3.6)
;
\foreach \i in {0,0.2,...,4}
 {\foreach \j in {0,0.2,...,4}
 \node[fill=black!40,circle,scale=0.2, inner sep=2pt]  at (\i,\j) {};
 }
\foreach \i in {2,2.2,...,3.6}
 {\foreach \j in {3,3.2,3.4,3.6}
 \node[fill=blue,circle,scale=0.2, inner sep=2pt]  at (\i,\j) {};
 }
 \foreach \i in {3.2,3.4,3.6}
 {\foreach \j in {2,2.2,...,2.8}
 \node[fill=blue,circle,scale=0.2, inner sep=2pt]  at (\i,\j) {};
 }
  \foreach \i in {2,2.2,...,2.8}
 {\node[fill=blue,circle,scale=0.2, inner sep=2pt]  at (\i,2.2) {};
 }
  \foreach \i in {2,2.2,2.4}
 {\node[fill=blue,circle,scale=0.2, inner sep=2pt]  at (\i,2) {};
 }
  \foreach \i in {2.6,2.8,3}
 {\node[fill=blue,circle,scale=0.2, inner sep=2pt]  at (\i,2.8) {};
 }
  \foreach \i in {2.4,2.6,2.8}
 {\node[fill=blue,circle,scale=0.2, inner sep=2pt]  at (\i,2.4) {};
 }
  \foreach \i in {2.6,2.8}
 {\node[fill=blue,circle,scale=0.2, inner sep=2pt]  at (2,\i) {};
 }
\end{tikzpicture}
    \caption{Edges colored in black are open in $\varsigma$, which form two loops in this example. The red edges in the figure are in $l^+\setminus l$. In the definition $U(L)$, we only need $l^+$ that intersects $D_\delta\setminus [\gamma]$. The shaded region is $[\gamma]$. The blue edges and blue vertices form the unexplored graph $D_\delta\setminus R$, which is the union of two $\delta$-discs in this example.}
    \label{fig:picl+}
\end{figure}

For each $R\in\mathcal{E}$, we define
\begin{align}\label{e.S}
    S_R =\Ll\{\varsigma \in \mathcal{D}_R : R =U\Ll(\varsigma^\lop\Rr) \Rr\}.
\end{align}
Formally speaking, an exploration on $R$ is successful when $R$ is exactly the union of the outside of $[\gamma]$, the loops intersecting $\partial_1 A$ in $\varsigma^\lop$ and their neighboring edges. The exploration process is described by the following algorithm:
\begin{enumerate}
    \item start by setting $R_0 = \D_\delta\setminus [\gamma]$;
    \item \label{i.2nd_step} for $i\in\N\cup\{0\}$, consider two cases:
    \begin{enumerate}
        \item if $\partial_{R_i}\eta_\delta =\emptyset $, then return $R=R_i$ and $\varsigma= (\eta_\delta)_{R_i}$;
        \item otherwise, let $R_{i+1}$ be a subgraph of $\D_\delta$ with $\VV(R_{i+1}) = \VV(R_i)\cup \partial_{R_i}\eta_\delta$ and
        \begin{align*}
            \EE(R_{i+1})=\Ll\{e\in \EE(\D_\delta): \:\text{$e$ has an endpoint in $\VV(R_{i+1})$}\Rr\},
        \end{align*}
        and repeat Step~\eqref{i.2nd_step} with $i+1$ substituted for $i$.
    \end{enumerate}
\end{enumerate}
This exploration process is based on finding sources $x$ on the revealed graph then further reveal $\eta_\delta$ on edges with endpoint $x$ until there is no source. Since this algorithm is not used in our proofs, we omit the verification of the equivalence between the algorithmic description and the static description.

\begin{lemma}\label{l.const_exp}
The exploration process $(\mathcal{E}, S)$ defined by~\eqref{e.E} and~\eqref{e.S} is efficient under $\P^\emptyset_{\D_\delta}$. Moreover, $\P^\emptyset_{\D_\delta}(B)=1$ for $B$ defined in~\eqref{e.B}.
\end{lemma}

\begin{proof}

\textit{Part~1.}
Under $\P^\emptyset_{\D_\delta}$, $\eta_\delta$ satisfies $\partial_{\D_\delta} \eta_\delta$ a.s. Fix any such realization $\eta_\delta$ and any loop decomposition $\eta^\lop_\delta$ as in Remark~\ref{r.eta^lop}. To show the efficiency defined by~\eqref{e.efficient}, it suffices to show
\begin{align}\label{e.RneqR'=>}
     R\neq R' \quad \Longrightarrow\quad \{\eta_\delta:(\eta_\delta)_R \in S_R\}\cap \{\eta_\delta:(\eta_\delta)_{R'} \in S_{R'}\}=\emptyset.
\end{align}
Recall the relations $\subset_{\R^2}$ and $=_{\R^2}$ for loop collections defined above Lemma~\ref{l.lop_unique}.
Note that~\eqref{e.RneqR'=>} follows from the claim:

\noindent\textit{If $(\eta_\delta)_R = \varsigma$ for some $R\in\mathcal{E}$ and some $\varsigma\in S_R$, then
\begin{align}\label{e.L(eta)=L_R(varsigma)}
    L(\eta_\delta) =_{\R^2} \varsigma^\lop
\end{align}
where we have set
\begin{align}\label{e.L(eta_delta)}
    L(\eta_\delta)= \Ll\{l\in \eta^\lop_\delta: \: l\cap \Ll(\D_\delta\setminus[\gamma]\Rr)\neq \emptyset\Rr\}.
\end{align}
}

Indeed, assuming the claim is true, if $\eta_\delta$ belongs to the intersection in~\eqref{e.RneqR'=>}, then~\eqref{e.L(eta)=L_R(varsigma)} implies that $\varsigma^\lop =_{\R^2} (\varsigma')^\lop$ for some $\varsigma\in S_R$ and $\varsigma'\in S_{R'}$. This along with~\eqref{e.S} yields $R=R'$.

Now, we prove the claim. Let $R$ and $\varsigma$ be given as in the condition. Let $l$ be any loop in $\varsigma^\lop$ and we have $l\subset l^+\subset R$. Since $\eta^\lop_\delta$ is a loop decomposition of $\eta_\delta$ and $(\eta_\delta)_R = \varsigma$, we have that $l\subset \cup \tilde l$ where the union is taken over loops $\tilde l$ intersecting $l$ in $L(\eta_\delta)$. The loops in $L(\eta_\delta)$ are disjoint and $l$ is connected, so $l\subset \tilde l$ (inclusion as graphs) for some $\tilde l\in L(\eta_\delta)$. We argue that we must have 
\begin{align}\label{e.tilde_l_setminus}
    \tilde l \setminus l =\emptyset
\end{align}
because otherwise, there is an edge $e \in l^+\setminus l$ contained in $\tilde l$, in which case $e$ is closed in $\varsigma$ and open in $\eta_\delta$ contradicting $(\eta_\delta)_R = \varsigma$. Therefore, we conclude that $l$ and $\tilde l$ are equal as graphs (i.e. $l=_{\R^2}\tilde l$). We have thus proven $\varsigma^\lop\subset_{\R^2} L(\eta_\delta)$.

We turn to the other inclusion. Let $\tilde l \in L(\eta_\delta)$. We show that
\begin{align}\label{e.E(tilde_l)cap}
    \EE(\tilde l )\cap \EE\Ll(\D_\delta\setminus [\gamma]\Rr) \neq \emptyset.
\end{align}
Suppose otherwise, then by the definition of $L(\eta_\delta)$, there exists some $x\in \VV(\tilde l)\cap \VV(\D_\delta\setminus [\gamma])$. This implies that there is an edge $e$ with one endpoint being $x$ such that $e\in \EE(\tilde l) \subset \EE([\gamma])$. Since $[\gamma]$ is a disc, we must have $x\in \VV([\gamma])$, contradicting the fact that $x\in \VV(\D_\delta\setminus [\gamma])$. Hence,~\eqref{e.E(tilde_l)cap} must hold. 

Let $e$ be any edge in the intersection in~\eqref{e.E(tilde_l)cap}. Since $(\eta_\delta)_R=\varsigma$ and $e\in R$ in $\eta_\delta^\lop$, we must have $e$ is on some $l\in \varsigma^\lop$ and thus $l\cap \tilde l\neq\emptyset$.
Due to $l\subset R$, the assumption $(\eta_\delta)_R = \varsigma$ implies that the edges on $l$ are open in $\eta_\delta$, meaning that $l$ is a connected subset consisting of open edges in $\eta_\delta$. As $\tilde l$ is a connected component in $\eta_\delta$, we have $ l\subset \tilde l$. Then, we derive~\eqref{e.tilde_l_setminus} using the same argument below that display. We conclude that $l =_{\R^2} \tilde l$ and thus $L(\eta_\delta)\subset_{\R^2} \varsigma^\lop$. We have shown~\eqref{e.L(eta)=L_R(varsigma)} and verified the claim.

\textit{Part~2.}
It remains to show that $\P^\emptyset_{\D_\delta}(B)=1$. In view of the definition of $B$ in~\eqref{e.B}, it suffices to find $R\in\mathcal{E}$ and $\varsigma\in S_R$ such that $(\eta_\delta)_R =\varsigma$. We take
\begin{align}\label{e.R=(...)}
    R = U(L(\eta_\delta))
\end{align}
for $L(\eta_\delta)$ given in~\eqref{e.L(eta_delta)} and $U$ defined in~\eqref{e.U(L)=}. As a subcollection of $\eta_\delta^\lop$, $L(\eta_\delta)$ is collection of disjoint weakly simple loops. Hence, we have $R\in \mathcal{E}$ by~\eqref{e.E}.

Then, we take $\varsigma\in \{0,1\}^{\EE(R)}$ by setting $\varsigma_e = 1$ if $e\in \EE(l)$ for some $l\in L(\eta_\delta)$ and $\varsigma_e=0$ otherwise. 
Since the loops in $\eta_\delta^\lop$ are disjoint and weakly simple, so are the loops in $L(\eta_\delta)$. 
Then, $L(\eta_\delta)$ is a loop decomposition of $\varsigma$ as described in Remark~\ref{r.extend_varsigma}. By~\eqref{e.decompsoable=>sourceless}, we have $\partial_R \varsigma=0$ and thus $\varsigma\in \mathcal{R}$ given in~\eqref{e.mathcal(D)_R}. Let $\varsigma^\lop$ be the loop decomposition given in Remark~\ref{r.extend_varsigma}. By Lemma~\ref{l.lop_unique}, we recover~\eqref{e.L(eta)=L_R(varsigma)}.
Plugging this into~\eqref{e.R=(...)} and comparing it with~\eqref{e.S}, we get $\varsigma\in S_R$. Lastly, it is clear from~\eqref{e.L(eta_delta)} and~\eqref{e.R=(...)} that $\varsigma = (\eta_\delta)_R$, which completes the proof.
\end{proof}

We need the lemma concerning the geometry of the explored region and the probability measure conditioned on the exploration.

\begin{lemma}\label{l.explore_markov}
Let $\varsigma\in S_R$ for some $R\in\mathcal{E}$. The following holds:
\begin{enumerate}
    \item there is some $n\in\N$ and a collection $(K_i)_{i=1}^n$ of disjoint $\delta$-discrete discs in $\D_\delta$ such that $\D_\delta\setminus  R = \cup_{i=1}^n K_i$;
    \item under~\ref{i.H_basics}, for every $i\in \{1,\dots,n\}$ and every event $\mathscr{E}$ depending only on $(\eta_\delta)_{K_i}$, we have $\P^\emptyset_{\D_\delta}(\mathscr{E}\,|\,(\eta_\delta)_R = \varsigma)= \P^\emptyset_{K_i}(\mathscr{E})$.
\end{enumerate}
\end{lemma}
\begin{proof}
For brevity, we write $\eta =\eta_\delta$.

\textit{Part~1.}
By~\eqref{e.S}, we have $R=U(\eta^\lop)$. Recall the definition of $U$ in~\eqref{e.U(L)=}. Then, we can derive
\begin{align}\label{e.DsetminusR=[gamma]...}
    \D_\delta \setminus R= [\gamma] \cap \Ll(\cap_{l\in \varsigma^\lop}\Ll(l^+\Rr)^\complement\Rr).
\end{align}
The definition of $l^+$ in~\eqref{e.l^+=} implies that $K\cap \Ll(l^+\Rr)^\complement$ is a union of disjoint discs for any disc $K$. Use this argument inductively and see that the right-hand side of~\eqref{e.DsetminusR=[gamma]...} is a disjoint union of discs. This gives the first result.

\textit{Part~2.}
Set $\hat K = (\cup_{j=1}^n K_j)\setminus K_i$. We first show
\begin{align}\label{e.partialpartial=>partial}
    \partial_R \eta =\emptyset,\ \partial_{\D_\delta} \eta =\emptyset \quad\Longrightarrow \quad \partial_{R\cup \hat K} \eta = \emptyset.
\end{align}
By the first part, we have 
\begin{align}\label{e.D-K=RUK}
    \D_\delta \setminus K_i  = R\cup\hat K.
\end{align}
By the definition of discs (see Section~\ref{s.discrete_objects}), we see that
\begin{align}\label{e.xy=>x}
    xy \in \EE(K_j) \qquad \Longrightarrow \qquad x\in \VV(K_j).
\end{align}
for every $j$, every $x\in \VV(\D_\delta)$, and every $y\sim x$. To show the right-hand side of~\eqref{e.partialpartial=>partial}, by~\eqref{e.D-K=RUK} and the definition of sources in~\eqref{e.source},
it suffices to verify that
\begin{align}\label{e.s(x)=even}
    s(x) = \sum_{y:\:xy\in \EE(\D_\delta\setminus K_i)}\eta_{xy} \ \in 2\Z, \quad\forall x\in \VV(\D_\delta).
\end{align}
We consider three cases depending on the location of $x$. In the following, $y$ is a neighboring vertex of $x$.

\textit{Case~1: $x\in \VV(\D_\delta)\setminus \VV(K_i)$.}
For every $xy \in \EE(\D_\delta)$, we must have $xy \in \EE(\D_\delta\setminus K_i)$ because otherwise $xy \in \E(K_i)$ which by~\eqref{e.xy=>x} implies $x\in \VV(K_i)$. Hence, along with $\partial_{\D_\delta} \eta=\emptyset$, we have
\begin{align*}
    s(x) = \sum_{y:\:xy\in \E(\D_\delta)}\eta_{xy}\ \in 2\Z.
\end{align*}

\textit{Case~2: $x\in \VV(K_i)\setminus\partial K_i$.}
Recall the definition of the inner boundary in~\eqref{e.boundary}, which implies that $y \in \VV(K_i)$ and thus $xy\in \EE(K_i)$ for every $y\sim x$. Hence, we have $s(x)=0$.

\textit{Case~3: $x\in \partial K_i$.}
Note that $xy \not\in \EE(K_j)$ for every $j\neq i$ because otherwise~\eqref{e.xy=>x} would imply $x\in K_j$ for some $j\neq i$.
Using this and~\eqref{e.D-K=RUK}, for every $xy\in \EE(\D_\delta\setminus K_i)$, we must have $xy \in \EE(R)$. Therefore, along with $\partial_R \eta = \emptyset$, we have
\begin{align*}
    s(x) = \sum_{y:\: xy\in \EE(R)}\eta_{xy}\ \in 2\Z.
\end{align*}

Combining these, we have shown~\eqref{e.s(x)=even} which yields~\eqref{e.partialpartial=>partial}. Next, we apply~\eqref{e.partialpartial=>partial} and the Markov property~\eqref{e.markov_prop} given by~\ref{i.H_basics} to finish the proof. Let $\mathscr{E}$ be given as in the statement. We compute
\begin{align*}
    \P^\emptyset_{\D_\delta}\Ll(\mathscr{E}\,\big|\, \eta_R = \varsigma\Rr) = \sum_{\kappa \in \{0,1\}^{\EE(\hat K)}}\P^\emptyset_{\D_\delta}\Ll(\mathscr{E}\,\big|\, \eta_R = \varsigma,\, \eta_{\hat K}=\kappa\Rr)\P^\emptyset_{\D_\delta}\Ll(\eta_R = \varsigma,\, \eta_{\hat K}=\kappa\big|\, \eta_R = \varsigma\Rr).
\end{align*}
Due to $\partial_{\D_\delta} \eta=\emptyset$ under $\P^\emptyset_{\D_\delta}$, the condition $\eta_R=\varsigma\in S_R$, and the definition of $S_R$ in~\eqref{e.S}, the left-hand side of~\eqref{e.partialpartial=>partial} is satisfied a.s. Hence, we use~\eqref{e.partialpartial=>partial}, the Markov property~\eqref{e.markov_prop}, and~\eqref{e.D-K=RUK} to get
\begin{align*}
    \P^\emptyset_{\D_\delta}\Ll(\mathscr{E}\,\big|\, \eta_R = \varsigma,\, \eta_{\hat K}=\kappa\Rr) =  \P^\emptyset_{\D_\delta\setminus (R\cup \hat K)}\Ll(\mathscr{E}\Rr) = \P^\emptyset_{K_i}\Ll(\mathscr{E}\Rr).
\end{align*}
Inserting this to the previous display gives the desired result.
\end{proof}

\subsection{Annulus crossing probability}

Under our assumptions, we show that, in the limit, if $\eta_\delta$ does not cross an annulus $A$, then neither does $\omega_\delta$.

\begin{lemma}\label{l.discon_bdy}
Suppose that $(\eta_\delta,\omega_\delta)$ satisfies~\ref{i.H_basics}--\ref{i.H_bdy_conn}. Then, for every dyadic annulus~$A$,
\begin{align}\label{e.limlimsupP<->}
    \lim_{r\to 0}\limsup_{\delta\to0}  \P^\emptyset_{\D_\delta}\Ll(\Ll\{ (\partial_0 A)_r \stackrel{\omega_\delta\cap A}{\longleftrightarrow} \Ll(\partial_1 A\Rr)_r  \Rr\}\cap \Ll(\circledcirc_A^{\eta_\delta}\Rr)^\complement\Rr) = 0.
\end{align}

\end{lemma}

\begin{proof}
For simplicity, we omit $\delta$ from the subscript and write $\omega$ and $\eta$ instead. 
Let $A^\eps\subset A$ be given as in~\ref{i.H_eta_cross_similar}. By the first bullet point in~\ref{i.H_eta_cross_similar}, there is $c_\eps>0$ such that $d(\partial_0 A^\eps, \partial_0 A)>c_\eps$ where $d$ is the Euclidean distance between subsets of $\R^2$.

Denote the event in~\eqref{e.limlimsupP<->} by $\{\cdots\}$, we have by~\ref{i.H_eta_cross_similar} that
\begin{align*}
    \P^\emptyset_{\D_\delta}\{\cdots\} \leq \P^\emptyset_{\D_\delta}\Ll(\Ll\{ (\partial_0 A)_r \stackrel{\omega\cap A}{\longleftrightarrow} \Ll(\partial_1 A\Rr)_r  \Rr\}\cap \Ll(\circledcirc_{A^\eps }^\eta\Rr)^\complement\Rr) + o_{\eps,\delta}(1),
\end{align*}
where $\lim_{\eps\to0}\limsup_{\delta\to 0} o_{\eps,\delta}(1) =0$.

We introduce the discrete approximations of $A$ and $A^\eps$. Since $A$ is dyadic, for sufficiently small $r$, we can find a $\delta$-discrete annulus $A_\delta$ satisfying $A\subset A_\delta$ and $d(\partial_i A,\partial_i A_\delta)\leq \delta$ for both $i=0,1$. For each $r>0$, we can find $\delta(r)>0$ such that for all $\delta<\delta(r)$,
\begin{align*}
    (\partial_0 A)_r \stackrel{\omega\cap A}{\longleftrightarrow} \Ll(\partial_1 A\Rr)_r \qquad \Longrightarrow\qquad \Ll(\partial_0 A_\delta\Rr)_{2r} \stackrel{\omega\cap A_\delta}{\longleftrightarrow} \Ll(\partial_1 A_\delta\Rr)_{2r}.
\end{align*}
For each $\eps$, we can find $\delta(\eps)$ such that, for all $\delta<\delta(\eps)$, there is a $\delta$-discrete annulus $A^\eps_\delta$ with $\partial_1 A^\eps_\delta = \partial_1 A_\delta$, $\partial_0 A^\eps_\delta$ inside $\partial_0 A^\eps$, and $d(\partial_0 A^\eps, \partial_0 A^\eps_\delta)<\delta$. Then, we see that 
$A^\eps\subset A^\eps_\delta$ and $\circledcirc_{A^\eps_\delta }^\eta\subset  \circledcirc_{A^\eps }^\eta$.
The last property along with the above display implies that
\begin{align}\label{e.cdots<<->}
    \P^\emptyset_{\D_\delta}\{\cdots\} \leq \P^\emptyset_{\D_\delta}\{\longleftrightarrow\} + o_{\eps,\delta}(1),\qquad\forall \delta <\delta(\eps,r).
\end{align}
where we have set $\delta(\eps,r)= \min\{\delta(\eps),\delta(r)\}$ and
\begin{align*}
    \{\longleftrightarrow\} = \Ll\{\Ll(\partial_0 A_\delta\Rr)_{2r} \stackrel{\omega\cap A_\delta}{\longleftrightarrow} \Ll(\partial_1 A_\delta\Rr)_{2r} \Rr\}\cap \Ll(\circledcirc_{A^\eps_\delta }^\eta\Rr)^\complement.
\end{align*}
By making $\delta(\eps,r)$ smaller, we can make sure
\begin{align}\label{e.d()>c_eps}
    d\Ll(\partial_0 A^\eps_\delta, \partial_0 A_\delta\Rr)>c_\eps/2,\qquad\forall \delta <\delta(\eps,r).
\end{align}

We apply results from previous sections. Recall the exploration process that starts outside a loop denoted by $\gamma$ in Section~\ref{s.exp_loop_from_outside}. Here, we set $\gamma =\partial_1 A_\delta$. Let $(\mathcal{E},S)$ be an exploration process defined by~\eqref{e.E} and~\eqref{e.S}.
Lemma~\ref{l.const_exp} implies that $(\mathcal{E},S)$ is efficient under $\P^\emptyset_{\D_\delta}$, while the second part of the lemma implies 
$ \P^\emptyset_{\D_\delta}\{\longleftrightarrow\} =  \P^\emptyset_{\D_\delta}\Ll(\{\longleftrightarrow\}\,|\, B\Rr)$ for $B$ defined in~\eqref{e.B} for $(\mathcal{E},S)$. Now, we use  Lemma~\ref{l.eff_exp} to get
\begin{align}\label{e.P(longrightleftarrow)}
    \P^\emptyset_{\D_\delta}\{\longleftrightarrow\} =\frac{\sum_{R\in\mathcal{E}}\sum_{\varsigma\in S_R}\P^\emptyset_{\D_\delta}(\{\longleftrightarrow\}|\eta_R =\varsigma)\P^\emptyset_{\D_\delta}(\eta_R=\varsigma)}{\sum_{R\in\mathcal{E}}\sum_{\varsigma\in S_R}\P^\emptyset_{\D_\delta}(\eta_R=\varsigma)}.
\end{align}

By the definitions of $S_R$ in~\eqref{e.S} and $U$ in~\eqref{e.U(L)=},
\begin{align}\label{e.R=A...}
    R = \Ll(\D_\delta \setminus [\partial_1 A_\delta]\Rr)\cup \Ll(\cup_{l\in \varsigma^\lop:\, l\cap \partial_1 (\D_\delta \setminus [\partial _1 A_\delta])\neq\emptyset}\,l^+\Rr)
\end{align}
where $[\partial_1 A_\delta]$ is the disc with boundary $\partial_1 A_\delta$. Recall the definition of $l^+$ above~\eqref{e.U(L)=}. Here, we interpret $l^+$ as $l$ with additional edges, each of which has an endpoint in $\VV(l)$ and closed in $\varsigma$.
Due to $\eta_R = \varsigma$, we know all the loops in $\eta$ intersecting the outside of $\partial_1 A_\delta$ are a part of $R$. Therefore, we deduce that $R\cap \partial_0 A^\eps_\delta \neq \emptyset$ implies $\circledcirc_{A^\eps_\delta}^\eta$ and thus
\begin{align}\label{e.RcapAneqempty}
    R\cap \partial_0 A^\eps_\delta \neq \emptyset \qquad\Longrightarrow \qquad \P^\emptyset_{\D_\delta}(\{\longleftrightarrow\}|\eta_R =\varsigma) =0, \quad\forall \varsigma \in S_R.
\end{align}

Now, we focus on $R\in \mathcal{E}$ satisfying
\begin{align}\label{e.RcapA=empty}
    R \cap \partial_0 A^\eps_\delta =\emptyset.
\end{align}
Using this and Lemma~\ref{l.explore_markov}, there is a disc $K$ satisfying $\partial_0 A^\eps_\delta \subset K \subset \D_\delta \setminus R$ and
\begin{align}\label{e.P^emptyset=P_K}
    \P^\emptyset_{\D_\delta}(\mathscr{E}\,|\,(\eta_\delta)_R = \varsigma)= \P^\emptyset_K(\mathscr{E})
\end{align}
for any event $\mathscr{E}$ depending only on $\eta_K$. We define an annulus $\tilde A$ by setting
\begin{align}\label{e.tilde-A}
    \partial_0\tilde A = \partial_0 A_\delta,\qquad \partial_1 \tilde A =\partial K.
\end{align}
Since $\D_\delta \setminus R\subset [\partial_1 A_\delta]$ from~\eqref{e.R=A...}, we deduce that
\begin{align}\label{e.partial_0Aconnects...}
    \Ll(\partial_0 A_\delta\Rr)_{2r} \stackrel{\omega\cap A_\delta}{\longleftrightarrow} \Ll(\partial_1 A_\delta\Rr)_{2r}\quad \Longrightarrow\quad \Ll(\partial_0 \tilde A\Rr)_{2r} \stackrel{\omega\cap \tilde A}{\longleftrightarrow} \Ll(\partial_1 \tilde A\Rr)_{2r},
\end{align}
from the property of $K$.
As $\partial K$ is outside $\partial_0 A^\eps_\delta$,~\eqref{e.tilde-A} and~\eqref{e.d()>c_eps} imply $d\Ll(\partial_0\tilde A,\partial_1\tilde A\Rr)> \frac{1}{2} c_\eps$. Using this and~\eqref{e.partial_0Aconnects...}, we have that under~\eqref{e.RcapA=empty},
\begin{align} \label{e.I+II}
    \P^\emptyset_{\D_\delta}(\{\longleftrightarrow\}|\eta_R =\varsigma)\leq \P^\emptyset_K\Ll(\Ll(\partial_0\tilde  A\Rr)_{2r} \stackrel{\omega\cap\tilde A}{\longleftrightarrow} \Ll(\partial_1\tilde  A\Rr)_{2r}  ,\ d\Ll(\partial_0\tilde A,\partial_1\tilde A\Rr)>c_\eps/2 \Rr)
\end{align}
where we used~\eqref{e.P^emptyset=P_K} since the event only depends on $\eta_K$ due to $\tilde A\subset K$ by~\eqref{e.tilde-A}.

We set $\lambda_\eps = 10^{-2}c_\eps$ and let $N_\eps =\sup_{\delta\in(0,1]}\Ll|\D_\delta\cap\lambda_\eps \Z^2\Rr| $, which is finite.
Set $\Gamma=\Gamma(\eps,\tilde A)$ to be the set of points in $\tilde A$ of distance $3\lambda_\eps$ to $\partial_1 \tilde A$.
There is an integer $k = k(\tilde A)\leq N_\eps$ and $\{x_1,\dots,x_k\} \in \D_\delta\cap \lambda_\eps \Z^2$ such that
\begin{itemize}
    \item $\Gamma \subset \cup_{i=1}^k \Lambda_{\lambda_\eps}(x_i)$,
    \item $\Lambda_{2\lambda_\eps}(x_i) \subset \tilde A$ for all $i\in\{1,\dots,k\}$;
    \item $\Lambda_{3\lambda_\eps}(x_i) \not\subset \tilde A$ for all $i\in\{1,\dots,k\}$.
\end{itemize}
By these and~\eqref{e.I+II}, we obtain
\begin{align*}
    \P^\emptyset_{\D_\delta}(\{\longleftrightarrow\}|\eta_R =\varsigma) \leq \sum_{i=1}^k \P^\emptyset_K\Ll(\Lambda_{\lambda_\eps}(x_i) \stackrel{\omega\cap\tilde A}{\longleftrightarrow}\Ll(\partial_{1}\tilde A\Rr)_{2r}\Rr),\quad\forall r<10^{-2}\lambda_\eps.
\end{align*}
Applying~\ref{i.H_bdy_conn} to the disc $K$, we get
\begin{align*}
    \sup_{i\in\{1,\dots,k\}}\P^\emptyset_K\Ll(\Lambda_{\lambda_\eps}(x_i) \stackrel{\omega\cap\tilde A}{\longleftrightarrow}\Ll(\partial_{1}\tilde A\Rr)_{2r}\Rr) = o^\eps_{r,\delta}(1),
\end{align*}
for some error term $o^\eps_{r,\eps}(1)$ that depends on $\eps$ and satisfies $\lim_{r\to0}\limsup_{\delta\to0} o^\eps_{r,\delta}(1)=0$ for each fixed $\eps$.
The above two displays yield
\begin{align}\label{e.P(...|eta=sigma)<No}
    \P^\emptyset_{\D_\delta}(\{\longleftrightarrow\}|\eta_R =\varsigma)  \leq N_\eps o^\eps_{r,\delta}(1),\qquad\forall r<10^{-2}\lambda_\eps.
\end{align}
Note that this bound is independent of $R$ or $\varsigma\in S_R$.

Recall that~\eqref{e.P(...|eta=sigma)<No} holds under the condition~\eqref{e.RcapA=empty}. Due to~\eqref{e.RcapAneqempty}, this bound trivially holds in the other case. Therefore, we apply the bound~\eqref{e.P(...|eta=sigma)<No} to~\eqref{e.P(longrightleftarrow)} and use~\eqref{e.cdots<<->} to conclude that
\begin{align*}
    \P^\emptyset_{\D_\delta}\Ll(\cdots\Rr) \leq N_\eps o^\eps_{r,\delta}(1) +o_{\eps,\delta}(1),\qquad\forall \delta<\delta(\eps,r),\ \forall r<10^{-2}\lambda_\eps.
\end{align*}
Sending $\delta\to0$, $r\to0$, and $\eps\to0$ in order, we get the desired result.
\end{proof}

\section{Topologies}\label{s.topologies}

In the first two subsections, we introduce the annulus-crossing space and recall the topology on loop collections. The former has a topology similar to the quad-crossing topology first introduced in \cite{SchrammSmirnov}. In the third subsection, we study the properties of a natural map from loop collections to annulus-crossings.

\subsection{Annulus-crossing topology}

We closely follow \cite[Section~3]{SchrammSmirnov}.

\subsubsection{Space of annuli}
Recall the definition of topological annuli in Section~\ref{s.def_topo_obj}.
Slightly abusing the notation, we denote by $A$ the equivalence class of $A$ modulo any homeomorphism $\phi$ from $\mathbb{S}^1\times [0,1]$ onto itself that maps $\mathbb{S}^1\times \{i\}$ to $A(\mathbb{S}^1\times \{i\})$ and preserves the orientation on $\mathbb{S}^1\times \{i\}$ for both $i=0,1$. 
Let $\mathcal{A}_\D$ be the space of (equivalent classes of) the annuli in $\D$, equipped with metric
\begin{align}\label{e.d_(A_D)}
    d_{\mathcal{A}_\D}(A_1,A_2) =\inf_\phi \sup_{z\in\mathbb{S}^1\times\{0,1\}}|A_1(z) - A_2(\phi(z))|.
\end{align}

Recall the inner and outer boundaries of annuli in Section~\ref{s.def_topo_obj}. Due to the Annuli Theorem (see item~7 in Section~A of \cite[Chapter~2]{rolfsen2003knots}), given two simple loops $\gamma_0,\gamma_1$ with $\gamma_0$ inside $\gamma_1$ (see Section~\ref{s.loop_properties}), there is an annulus $A$ such that $\partial_i A =\gamma_i$ for $i=0,1$.

\begin{lemma}\label{l.d<r,in()_r}
If $d_{\mathcal{A}_\D}(A,A')\leq r$ for some $r>0$, then $\partial_i A'\subset (\partial_i A)_r$ for both $i=0,1$.
\end{lemma}
\begin{proof}
Suppose that $\partial_i A'\not\subset (\partial_i A)_r$ for some $i\in\{0,1\}$. Then there is $t\in \mathbb{S}^1$ such that $|A'(t,i)-A(s,i)|> r$ for all $s\in\mathbb{S}^1$. In particular, $|A'(t,i) - A(\phi(t,i))|> r$ for all admissible $\phi$. This implies that $d_{\mathcal{A}_\D}(A,A')\geq r$.
\end{proof}

\subsubsection{Orders on annuli}\label{s.orders_on_annuli}
To define the topology, we need a partial order defined in terms of crossing annuli. It is also useful to consider the dual of crossing: separation.
Hence, we define:
\begin{itemize}
    \item A \textbf{crossing} of $A$ is a compact connected subset of $A(\mathbb{S}^1\times[0,1])$ that intersects both $\partial_0 A$ and $\partial_1 A$;
    \item A \textbf{separation} of $A$ is a compact connected subset of $A(\mathbb{S}^1\times[0,1])$ whose complement in $\R^2$ contains $\partial_0 A$ and $\partial_1 A$ in different open connected components. 
\end{itemize}
Accordingly, a subset $R$ of $\R^2$ is said to \textbf{cross} (resp.\ \textbf{separate}) $A$ if $R$ contains a crossing (resp.\ separation) of $A$.
Note that the event $\circledcirc_A^S$ in~\eqref{e.circledcirc_A^s} is equivalent to that $S$ crosses $A$.

Then, we define the corresponding partial orders:
\begin{itemize}
    \item $A_1\leq_\cro A_2$ if every crossing of $A_2$ contains a crossing of $A_1$;
    \item $A_1\leq_\sep A_2$ if every separation of $A_2$ contains a separation of $A_1$.
\end{itemize}
The annulus-crossing order is defined analogously to the quad-crossing order in \cite{SchrammSmirnov}. As illustrated in \cite[Figure~2.1]{garban2013pivotal}, the quad-crossing order is not given by set inclusion. In this aspect, the annuli setting is much simpler.

A loop $l$ is said to \textbf{circulate} an annulus $A$ if $\partial_0 A$ is inside $l$ and $l$ is inside $\partial_1 A$. Note that $l$ separates $A$ if and only if $l$ circulates $A$ and does not intersect $\partial A$. An annulus $A'$ is said to \textbf{circulate} another annulus $A$ if $\partial_i A'$ circulate $A$ as a loop for both $i=0,1$.

\begin{lemma}\label{l.ineq_sep}
Let $A_1,A_2$ be two annuli. Then, 
\begin{align*}
    A_2\leq_\sep A_1 \quad \Longleftrightarrow \quad  \text{$A_1$ circulates $A_2$} \quad  \Longleftrightarrow\quad   A_1\leq_\cro A_2 
\end{align*}
\end{lemma}
\begin{proof}
In the first equivalence relation, the deduction direction is obvious. To see the other direction, we suppose that $A_1$ does not circulate $A_2$. In this case, $\partial_i A_1$ is not contained in $A_2$ for some $i=1,2$. Then, we find a simple loop $l$ circulating $A_1$ and close to $\partial_i A_1$ so that $l\not\subset A_2$. Hence, $l$ separates $A_1$ but not $A_2$.
So $A_2\leq_\sep A_1$ does not hold, proving the first equivalence relation.

The implication direction in the second equivalence relation is obvious. To prove the other direction, it suffices to show that $A_1\leq_\cro A_2  $ implies $A_2\leq_\sep A_1$. Let $s_1$ be any separation of $A_1$, then $s_1$ intersects every crossing of $A_1$.
Since every crossing of $A_2$ contains a crossing of $A_1$, $s_1$ intersects every crossing of $A_2$. 
Then $\partial_0 A_2$ and $\partial_1 A_2$ are in different components of $\R^2\setminus s_1$, because otherwise there are $x_i\in \partial_iA_2$ for both $i=0,1$, in the same component of $\R^2\setminus s_1$ which allows us to find a crossing of $A_2$ that does not intersect $s_1$. Therefore, $s_1$ separates $A_2$ and thus $A_2\leq_\sep A_1$.
\end{proof}

\subsubsection{Topology on the collection of hereditary sets}

We describe a topology on a suitable subcollection of $\{\mathcal{S}:\mathcal{S}\subset \mathcal{A}_\D\}$. We write $A_1<_\cro A_2$ if there are open neighborhoods $\mathcal{U}_1$ of $A_1$ and $\mathcal{U}_2$ of $A_2$ in $\mathcal{A}_\D$ such that $A'_1\leq_\cro A'_2$ for every $A'_1\in \mathcal{U}_1$ and every $A'_2\in\mathcal{U}_2$.

A subset $\mathcal{S}\subset \mathcal{A}_\D$ is called \textbf{hereditary} (as in \cite{garban2013pivotal}; or \textbf{lower} as in \cite{SchrammSmirnov}) if
\begin{align}\label{e.hereditary}
    A\in \mathcal{S},\ A'\in \mathcal{A}_\D,\ A'<_\cro A \quad \Longrightarrow \quad A'\in \mathcal{S}.
\end{align}
We define $\mathscr{H}_{\mathcal{A}_\D} = \{\mathcal{S}\subset\mathcal{A}_\D:\text{$\mathcal{S}$ is hereditary}\}$.

Then, we describe a topology $\mathbf{T}_{\mathcal{A}_\D}$ on $\mathscr{H}_{\mathcal{A}_\D}$ by specifying a subbase. For every $A\in\mathcal{A}_\D$ and every open subset $\mathcal{U}\subset \mathcal{A}_\D$, we define
\begin{align}\label{e.def_open_sets}
    \circledcirc_A = \{\mathcal{S}\in \mathscr{H}_{\mathcal{A}_\D}:A\in\mathcal{S}\},\qquad \mathscr{V}_\mathcal{U} = \{\mathcal{S}\in \mathscr{H}_{\mathcal{A}_\D}:\mathcal{S}\cap \mathcal{U}\neq\emptyset\}.
\end{align}
Let $\mathbf{T}_{\mathcal{A}_\D}$ be the minimal topology containing every such $(\circledcirc_A)^\complement$ and $\mathscr{V}_\mathcal{U}$.

Note that a percolation configuration $\kappa$ on $\D_\delta$ is naturally associated with a $\mathcal{S}$ in $\mathscr{H}_{\mathcal{A}_\D}$: the set of annuli crossed by the union of open edges of $\kappa$ when viewed as a subset of $\R^2$. 
Heuristically, $\circledcirc_A$ corresponds to the set of $\kappa$ crossing $A$; and for a neighborhood $\mathcal{U}$ of $A$, $\mathcal{S}\in \mathscr{V}_{\mathcal{U}}$ corresponds to the event that $\kappa$ crosses an annulus close to $A$.

Recall the definition of $\mathcal{A}^k_\D$ and $\mathcal{A}_{\D}^{\mathrm{dya}}$ in Section~\ref{s.dyadic_annuli}.
It is easy to check that $\mathcal{A}_\D$ satisfies the conditions in \cite[Theorem~3.19]{SchrammSmirnov}, which yields the following lemma. 

\begin{lemma}\label{l.basics_topology}
The topological space $(\mathscr{H}_{\mathcal{A}_\D},\mathbf{T}_{\mathcal{A}_\D})$ is Hausdorff, compact, and metrizable. Moreover, the $\sigma$-algebra generated by $\{\circledcirc_{A}\}_{A\in \mathcal{A}_{\D}^{\mathrm{dya}}}$ is the Borel $\sigma$-algebra of $(\mathscr{H}_{\mathcal{A}_\D},\mathbf{T}_{\mathcal{A}_\D})$.
\end{lemma}

\subsubsection{Metric and neighborhoods}
Although $\mathscr{H}_{\mathcal{A}_\D}$ is metrizable, we do not know a useful explicit metric.
As a remedy, we describe the closedness in metric using open neighborhoods, following \cite{garban2013pivotal}.

For $A\in \mathcal{A}^k_\D$ for some $k\in\N$, we take $\mathcal{B}^k_A$ to be the open $d_{\mathcal{A}_\D}$-metric ball centered at $A$ with radius $2^{-k-10}$; and let $\bar A^k \in \mathcal{A}^{k+10}_\D$ be the smallest annulus in terms of $\leq_\cro$ in $\{A' \in \mathcal{A}_\D^{k+10}: A<_\cro A'\}$. For every $\mathcal{S}\in \mathscr{H}_{\mathcal{A}_\D}$, we define the open neighborhood
\begin{align}\label{e.O^k(omega)}
    \mathscr{O}^k(\mathcal{S}) = \Ll(\bigcap_{A\in \mathcal{A}^k_\D:\: A\not\in \mathcal{S}} \circledcirc_{\bar A^k}^\complement\Rr)\bigcap \Ll(\bigcap_{A\in\mathcal{A}^k_\D:\:A\in\mathcal{S}} \mathscr{V}_{\mathcal{B}^k_A}\Rr)
\end{align}
which is open because $\mathcal{A}^k_\D$ is finite.
If $\mathcal{S}' \in \mathscr{O}^k(\mathcal{S})$, then we have the following:
\begin{itemize}
    \item if $\mathcal{S}$ does not cross $A$ (i.e.\ $A\not\in\mathcal{S}$), then $\mathcal{S}'$ does not cross $\bar A^k$ (i.e.\ $\mathcal{S}'\in \circledcirc_{\bar A^k}^\complement$) which is slightly harder to cross than $A$;
    \item if $\mathcal{S}$ crosses $A$ (i.e.\ $A\in\mathcal{S}$), then there is $A'$ close to $A$ (i.e.\ $A'\in \mathcal{B}^k_A$) which is also crossed by $\mathcal{S}'$ (i.e.\ $\mathcal{S}'\in \mathscr{V}_{\mathcal{B}^k_A}$).
    
\end{itemize}
Loosely speaking, if $\mathcal{S}'\in\mathscr{O}^k(\mathcal{S})$ and $\mathcal{S}\in\mathscr{O}^k(\mathcal{S}')$, then $\mathcal{S}$ and $\mathcal{S}'$ cross similar annuli in~$\mathcal{A}^k_\D$.

We fix any metric $d_{\mathscr{H}_{\mathcal{A}_\D}}$ on $\mathscr{H}_{\mathcal{A}_\D}$ that generates $\mathbf{T}_{\mathcal{A}_\D}$.
The next lemma, which is a straightforward adaptation of \cite[Lemma~2.5]{garban2013pivotal}, describes the metric in terms of the open neighborhoods introduced above.
\begin{lemma}\label{l.metric_dya}
There is a function $\mathfrak{r}:\N\to [0,\infty)$ such that, for every $\mathcal{S},\mathcal{S}'\in\mathscr{H}_{\mathcal{A}_\D}$, if $d_{\mathscr{H}_{\mathcal{A}_\D}}(\mathcal{S},\mathcal{S}') \leq \mathfrak{r}(k)$ for some $k\in\N$, then $\mathcal{S}'\in\mathscr{O}^k(\mathcal{S})$ and $\mathcal{S}\in\mathscr{O}^k(\mathcal{S}')$.
\end{lemma}

\subsection{Space of loop collections}

Recall in Section~\ref{s.loop_properties} that we identify loops up to orientation-preserving reparametrizations $\phi$.

\subsubsection{Space of loops}
We measure the distance between two loops $l$ and $l'$ by
\begin{align}\label{e.d_mathcal_L}
    d_{\mathcal{L}_\D}\Ll(l,l'\Rr) = \inf_\phi \sup_{t\in \S^1} \Ll|l(t)- l'(\phi(t))\Rr|.
\end{align}
We denote by $\mathcal{L}_\D$ the completion of the collection of simple loops in $\D$ under this metric. Hence, $\mathcal{L}_\D$ includes weakly simple loops (defined in Section~\ref{s.loop_properties}) and trivial loops reduced to points. We call $\mathcal{L}_\D$ the \textbf{space of weakly simple loops in $\D$}.
\begin{lemma}\label{l.d(l,l')<delta}
If $d_{\mathcal{L}_\D}(l,l')\leq \delta$ for some $\delta>0$, then $l\subset (l')_\delta$.
\end{lemma}
\begin{proof}
Suppose that $l\not\subset (l')_\delta$. Then, there is $t$ such that $|l(t)-l'(t')|>\delta$ for every $t'$. This implies that $|l(t)-l'(\phi(t))| >\delta$ for every reparametrization $\phi$ that preserves the orientation. Hence, we have $d_{\mathcal{L}_\D}(l,l')\geq \delta$.
\end{proof}

\subsubsection{Metric on loop collections}\label{s.metric_on_loop_coll}
Set $\mathscr{L}_\D$ to be the \textbf{space of loop collections} consisting of $\{l_i\}_{i\in I}$ with $I\subset \N$ satisfying
\begin{itemize}
    \item $l_i\in \mathcal{L}_\D$ and $l_i$ is not a point for every $i\in I$;
    \item for every $\eps>0$, $\{i\in I:\diam(l_i)\geq \eps\}$ is finite, where $\diam$ is the Euclidean diameter.
\end{itemize}
Loops in $\{l_i\}_{i\in I}$ can appear with multiplicity and $\mathscr{L}_\D$ includes the empty collection.
This definition is the same as in \cite[Section~2.7]{benoist2019scaling}.

For two index sets $I, J$, we call $\pi\subset I\times J$ a \textbf{matching} if for each $i\in I$ there is at most one $j\in J$ such that $(i,j)\in \pi$ and vice versa. Note that $\pi$ can be empty. Given a matching $\pi$, we set $I^\pi = \{ i\in I: \nexists j\in J,\, (i,j)\in \pi\}$ to be the set of unmatched indices in $I$ and set $J^\pi$ similarly. 
For $\{l_i\}_{i\in I}$, $\{l'_j\}_{j\in J}$ and a matching $\pi\subset I\times J$, we define
\begin{align*}
    d_\pi \Ll(\{l_i\}_{i\in I},\, \{l'_j\}_{j\in J}\Rr) = \max\Ll(\sup_{(i,j)\in\pi} d_{\mathcal{L}_\D}(l_i,l'_j),\, \sup_{i\in I^\pi} \frac{1}{2}\diam(l_i) ,\, \sup_{j\in J^\pi} \frac{1}{2}\diam(l'_j)\Rr)
\end{align*}
We equip $\mathscr{L}_\D$ with the following metric:
\begin{align}\label{e.d_mathscr_L}
    d_{\mathscr{L}_\D} \Ll(\{l_i\}_{i\in I},\, \{l'_j\}_{j\in J}\Rr) = \inf_{\pi} d_\pi \Ll(\{l_i\}_{i\in I},\, \{l'_j\}_{j\in J}\Rr)
\end{align}
where the infimum is taking over all matching $\pi\subset I\times J$.
The definition of $d_{\mathscr{L}_\D}$ is similar to the distance in \cite{benoist2019scaling}, whose results on the convergence of Ising loops are used later. One difference is that we scale the diameters by $\frac{1}{2}$, making the following verification easier.
\begin{lemma}
The function $d_{\mathscr{L}_\D}:\mathscr{L}_\D\times \mathscr{L}_\D \to [0,\infty)$ is a metric.
\end{lemma}
\begin{proof}
We write $L= \{l_i\}_{i\in I}$ and $L' = \{l'_j\}_{j\in J}$. Clearly, $d_{\mathscr{L}_\D}(L,L)=0$. Now assuming $d_{\mathscr{L}_\D}(L,L')=0$, we fix any $l_i\in L$. Then we have that $J_0=\{j\in J:d_{\mathcal{L}_\D}(l_i,l'_j)<\frac{1}{4}\diam(l_i)\}$ is not empty, for otherwise any matching $\pi$ would satisfy $d_\pi(L,L')\geq \frac{1}{4}\diam(l_i)$. Using the inequality
\begin{align}\label{e.diam<2d+diam}
    \diam(l) \leq 2 d_{\mathcal{L}_\D}(l,l') + \diam(l'),\quad\forall l,l'\in \mathcal{L}_\D,
\end{align}
we see that $\diam(l'_j) > \frac{1}{2}\diam(l_i)$ for every $j\in J_0$. Hence, $J_0$ is finite. For every $\eps \in (0, \frac{1}{4}\diam(l_i))$, due to $d_{\mathscr{L}_\D}(L,L')<\eps$, there exists $j_\eps\in J_0$ such that $d_{\mathcal{L}_\D}(l_i,l'_{j_\eps})<\eps$. Since $J_0$ is finite, sending $\eps\to0$, there is $j\in J_0$ such that $l_i = l'_j$. Therefore, we have $L\subset L'$, and the other direction is analogous, proving $L=L'$.

Since $d_{\mathscr{L}_\D}$ is symmetric, it only remains to verify the triangle inequality. 
Let $L = \{l_i\}_{i\in I}$, $L'=\{l'_j\}_{j\in J}$, and $L''=\{l''_k\}_{k\in K}$. 
For any $\pi\subset I\times J$ and $\pi'\subset J\times K$,
we construct the matching $\bar\pi = \{(i,k)\in I\times K: \text{$(i,j)\in \pi$ and $(j,k)\in \pi'$ for some $j\in J$}\} $. Then, we have
\begin{align*}
    \sup_{(i,k)\in \bar\pi} d_{\mathcal{L}_\D}(l_i, l''_k) \leq \sup_{(i,j)\in \pi} d_{\mathcal{L}_\D}(l_i, l'_j) + \sup_{(j,k)\in \pi'} d_{\mathcal{L}_\D}(l'_j, l''_k) \leq d_{\pi}(L,L')+ d_{\pi'}(L',L'').
\end{align*}
For any unmatched index $i\in I^{\bar\pi}$, there are two case: $i\in I^{\pi}$ or $i\not\in I^{\pi}$. In the first case, we have $\frac{1}{2}\diam(l_i) \leq d_\pi(L,L')$. In the second case, there is $j$ such that $(i,j)\in \pi$ and $j$ is unmatched in $\pi'$. Along with~\eqref{e.diam<2d+diam} for $l_i$ and $l'_j$, we have 
\begin{align*}
    \frac{1}{2}\diam(l_i) \leq d_\pi(L,L') + d_{\pi'}(L',L''),
\end{align*}
which holds for all $i\in I^{\bar\pi}$. A similar bound holds for $l_i$ replaced by $l'_j$, for any unmatched $j\in J^{\bar\pi}$. Combining the above two displays, we obtain $d_{\bar\pi}(L,L') \leq d_\pi(L,L') + d_{\pi'}(L',L'')$. Optimizing over $\pi,\pi'$ gives the triangle inequality.
\end{proof}

\subsubsection{Space of collections of simple loops}
\label{s.coll_simple}
Given any collection $L=\{l_i\}_{i\in I}$ of distinct loops, we define the level of loop $l\in L$ by $\mathrm{level}_L(l) =1 + |\{i\in I: \text{$l$ inside $l_i$}\}|$. 
Since there are finitely many loops with diameter greater than that of $l$, $\mathrm{level}_L(l)$ is finite.

Recall the notion of disjoint loops in Section~\ref{s.loop_properties}.
We define $\mathscr{L}_\D^\good$ to be the subspace of $\mathscr{L}_\D$ satisfying the following:
\begin{itemize}
    \item loops in $L$ are disjoint, simple and do not intersect $\partial \D$;
    \item for every $l\in L$, $l$ is oriented clockwise if $\mathrm{level}_L(l)$ is odd and counterclockwise if $\mathrm{level}_L(l)$ is even.
\end{itemize}
We record a basic result, extracted from \cite[Lemma~5]{benoist2019scaling}, to conclude this part.

\begin{lemma}
The metric space $\mathscr{L}_\D$ is complete and separable. The subset $\mathscr{L}^\good_\D$ is measurable with respect to the Borel $\sigma$-algebra on $\mathscr{L}_\D$.
\end{lemma}

\subsection{Mapping loops to crossings}

Define $F:\mathscr{L}_\D\to\mathscr{H}_{\mathcal{A}_\D}$ by setting
\begin{align}\label{e.F(L)}
    F(L) = \Ll\{A\in \mathcal{A}_\D:\: \text{$A$ is crossed by some $l\in L$}\Rr\},\quad\forall L\in\mathscr{L}_\D.
\end{align}
Obviously, $F$ is well-defined, namely, $F(L)\in \mathscr{H}_{\mathcal{A}_\D}$.

\begin{lemma}\label{l.F_cts}
The function $F:\mathscr{L}_\D\to\mathscr{H}_{\mathcal{A}_\D}$ is continuous.
\end{lemma}

\begin{proof}
Due to~\eqref{e.def_open_sets}, there are two types of open subsets of $\mathscr{H}_{\mathcal{A}_\D}$ to consider: $\circledcirc_A^\complement$ for $A \in \mathcal{A}_\D$ and $\mathscr{V}_\mathcal{U}$ for an open set $\mathcal{U}\subset \mathcal{A}_\D$.
It suffices to show that their preimages under $F$ are open in $\mathscr{L}_\D$.

First, fixing any annulus $A$, we show that $F^{-1}(\circledcirc_A^\complement)$ is open.
Let $L =\{l_i\}_{i\in I} \in F^{-1}(\circledcirc_A^\complement)$ which means that no loop in $L$ crosses $A$. We fix some $\eps>0$ small so that, for any loop $l$,
\begin{align}\label{e.diam<eps=>no_cross}
    \diam(l)<\eps \qquad \Longrightarrow\qquad \text{$l$ does not cross $A$}.
\end{align}
We set $I_0=\{i\in I: \diam(l_i)\geq \frac{\eps}{2}\}$ to collect larger loops in $L$. Since $I_0$ is finite and $L\in F^{-1}(\circledcirc_A^\complement)$, by Lemma~\ref{l.d(l,l')<delta}, we can choose $\delta\in(0,\frac{\eps}{4})$ small to ensure that, for any loop $l$,
\begin{align}\label{e.inf<delta=>no_cross}
    \inf_{i\in I_0}d_{\mathcal{L}_\D}(l, l_i)\leq \delta\qquad\Longrightarrow\qquad \text{$l$ does not cross $A$}.
\end{align}

We show that, for any $L'=\{l'_j\}_{j\in J}$ satisfying $d_{\mathscr{L}_\D}(L,L')<\frac{\delta}{2}$, no loop in $L'$ crosses $A$, meaning that $L'\in F^{-1}(\circledcirc_A^\complement)$. We fix any matching $\pi\in I\times J$ such that $d_\pi (L,L') <\delta$. For any $j\in J$, we consider three cases:
\begin{enumerate}
    \item if there is $i\in I_0$ such that $(i,j)\in \pi$, then we have $d_{\mathcal{L}_\D}(l_i,l'_j) <\delta$, which by~\eqref{e.inf<delta=>no_cross} implies that $l'_j$ does not cross $A$;
    \item if there is $i\not\in I_0$ such that $(i,j)\in \pi$, then by~\eqref{e.diam<2d+diam}, we have $\diam(l'_j) \leq 2d_{\mathcal{L}_\D}(l_i,l'_j) + \diam(l_i)< 2\delta + \frac{\eps}{2}<\eps$, implying that $l'_j$ does not cross $A$ by~\eqref{e.diam<eps=>no_cross};
    \item if index $j$ is unmatched, then we have $\diam(l'_j) <2\delta<\frac{\eps}{2}$ and thus $l'_j$ does not cross $A$ due to~\eqref{e.diam<eps=>no_cross}.
\end{enumerate}
We conclude that $L'\in F^{-1}(\circledcirc_A^\complement)$ and therefore the open metric ball centered at $L$ of radius $\frac{\delta}{2}$ is contained in $F^{-1}(\circledcirc_A^\complement)$, proving $F^{-1}(\circledcirc_A^\complement)$ is open.

Next, we show that $F^{-1}(\mathscr{V}_\mathcal{U})$ is open, for any fixed open set $\mathcal{U}\subset\mathcal{A}_\D$. Let $L \in F^{-1}(\mathscr{V}_\mathcal{U})$, then by the definition of $\mathscr{V}_\mathcal{U}$ in~\eqref{e.def_open_sets}, we have $F(L)\cap \mathcal{U}\neq\emptyset$, meaning there exists $l\in L$ and $A\in \mathcal{U}$ such that $l$ crosses $A$. Since $\mathcal{U}$ is open, there is $\eps>0$ such that the open $d_{\mathcal{A}_\D}$-metric ball $\mathcal{B}$ centered at $A$ with radius $\eps$ is a subset of $\mathcal{U}$. 

Using Lemma~\ref{l.d<r,in()_r} and Lemma~\ref{l.d(l,l')<delta}, we can find $c>0$ sufficiently small such that, for every loop $l'$,
\begin{align}\label{e.d<c=>cross_A'}
    d_{\mathcal{L}_\D}(l',l)<c\qquad \Longrightarrow\qquad \text{$l'$ crosses some $A' \in \mathcal{B}$}.
\end{align}
We fix $\delta>0$ satisfying $\delta < \min\{\frac{1}{2}c, \frac{1}{8}\diam(l)\}$. Let $L'$ be any collection satisfying $d_{\mathscr{L}_\D}(L',L)<\delta$. We can fix a matching $\pi$ between $L'$ and $L$ such that $d_\pi (L',L)< 2\delta$. Then, $l$ is matched because otherwise we would have $d_\pi(L',L)>4\delta$ by the choice of $\delta$. Hence, $l$ is matched with some $l'\in L'$. This gives $d_{\mathcal{L}_\D}(l',l)<2\delta <c$ which together with~\eqref{e.d<c=>cross_A'} implies that $l'$ crosses some $A'\in \mathcal{B}$. 
We get $F(L')\cap \mathcal{U}\neq\emptyset$, or equivalently $L' \in F^{-1}(\mathscr{V}_\mathcal{U})$. We deduce that the open $d_{\mathscr{L}_\D}$-metric ball centered at $L$ with radius $\delta$ is contained in $F^{-1}(\mathscr{V}_\mathcal{U})$, proving that $F^{-1}(\mathscr{V}_\mathcal{U})$ is open.
\end{proof}

We turn to the injectivity of $F$ on $\mathscr{L}_\D^\good$. For this goal, we need the following lemma.
Recall the definition of separation in Section~\ref{s.orders_on_annuli}.
For every $L\in\mathscr{L}_\D$, we define
\begin{align*}
    \mathcal{G}(L) = \Ll\{A\in\A_\D: A\text{ is separated by some }l\in L \Rr\}.
\end{align*}

\begin{lemma}\label{l.F=G}
If $F(L) = F(L')$ for some $L,L'\in\mathscr{L}^\good_\D$, then $\mathcal{G}(L) = \mathcal{G}(L')$.
\end{lemma}
\begin{proof}

Assuming $F(L)=F(L')$, we want to show $\mathcal{G}(L)\subset \mathcal{G}(L')$. The other direction follows by symmetry.
We argue by contradiction and suppose that there is $A\in \mathcal{G}(L)\setminus \mathcal{G}(L')$. Let $l\in L$ be a loop that separates $A$. We consider two cases (see Figure~\ref{fig:pic4.9cases}). Recall the notation and convention for loops viewed as subsets of $\R^2$ in Section~\ref{s.loop_subset_R^2}. 

\textit{Case~1}:
$l'\cap l \neq \emptyset$ for some $l'\in L'$. 
Since $l$ is simple and $l$ separates $A$, the closure of either of the two components of $A\setminus l$ is an annulus. For each $i\in\{0,1\}$, we denote the one containing $\partial_i A$ by $A_i$.

Due to $A\not \in \mathcal{G}(L')$, we have $l'\neq l$ and there is some $x\in l'\setminus l$ in the interior of $A_i$ for some $i\in\{0,1\}$.
Then, we can find a loop separating $A_i$ and $x\in \tilde l$.
Let $\tilde A$ be the annulus with the boundary loops $l$ and $\tilde l$. Since the loops in $L$ are disjoint, $\tilde A$ is not crossed by any loop in $L$. However, $\tilde A$ is crossed by $l'$, which implies that $\tilde A\in F(L')\setminus F(L)$, contradicting $F(L)=F(L')$.

\textit{Case~2}: $l'\cap l = \emptyset$ for all $l'\in L'$. 
Fix any $\eps\in(0, \frac{1}{7}\diam (l)]$ and set $L'_\eps=\{l'\in L' : \diam (l')>\eps\}$. Fixing any $x\in l$, we set $B_r(x) = \{z\in\R^2:|z-x|\leq r\}$ for any $r>0$.
Since $L'_\eps$ is finite, we can find $\delta\in(0,\eps)$ sufficiently small such that
\begin{align}\label{e.B_delta(x)cap=emptyset}
    B_\delta(x)\cap \Ll(\cup_{l'\in L'_\eps} l'\Rr) =\emptyset.
\end{align}
Set $A^\star$ to be the closure of $B_{\delta+2\eps}(x)\setminus B_{\delta}(x)$, which is an annulus. We have that $l$ crosses $A^\star$, because otherwise $\diam(l)\leq 6\eps$, which contradicts the choice of $\eps$. Thus, $A^\star \in F(L)$.

Then, we show that no loop in $L'$ crosses $A^\star$. If $l'$ crosses $A^\star$, then $\diam(l')\geq 2\eps$. Hence, no loop in $L'\setminus L'_\eps$ can cross $A^\star$. Also by the choice of $\delta$ as in~\eqref{e.B_delta(x)cap=emptyset}, no loop in $L'_\eps$ crosses $A^\star$. Therefore, we get $A^\star \not\in F(L')$, contradicting $F(L)=F(L')$.

\begin{figure}
    \centering
    
\begin{tikzpicture}[scale=0.6]
\fill[blue!10]
(3,2) to[in=-60,out=-45] (3,-2) to[in=30, out=120] (-0.5,-3) to[in=-100, out=-150] (-3.5,1.5) to[in=135,out=70] (3,2);
\fill[blue!20] (-0.3,1.8) to[out=20,in=130] (3,1.8) to[out=-50,in=80] (2.4,-0.4) to[out=-100,in=40] (-0.8,-2.8) to[out=-140,in=-80] (-2.6,-0.4) to[out=100,in=160] (-0.3,1.8);
\fill[blue!10] (-2,0) to[out=-100,in=-175] (-1.5,-1.5) to[out=15, in=170] (0,-2) to[out=-10,in=-170] (1,-1.2) to[out=10, in=-105] (2.5,1.2) to[out=75, in=-5] (-0.2,1.4) to[out=175,in=80] (-2,0);
\fill[white]
(0,-1) to[out=70,in=-95] (1,0.5) to[out=85,in=15] (-0.5,0.5) to[out=-165,in=105] (-0.25,0) to[out=-85,in=-110] (0,-1);

\draw (3,2) to[in=-60,out=-45] (3,-2) to[in=30, out=120] (-0.5,-3) to[in=-100, out=-150] (-3.5,1.5) to[in=135,out=70] (3,2)
(0,-1) to[out=70,in=-95] (1,0.5) to[out=85,in=15] (-0.5,0.5) to[out=-165,in=105] (-0.25,0) to[out=-85,in=-110] (0,-1);
\draw[very thick] (-2,0) to[out=-100,in=-175] (-1.5,-1.5) to[out=15, in=170] (0,-2) to[out=-10,in=-170] (1,-1.2) to[out=10, in=-105] (2.5,1.2) to[out=75, in=-5] (-0.2,1.4) to[out=175,in=80] (-2,0);
\draw[blue,very thick] (-4,3) to[out=-60,in=140] (-1.8,0.6) to[out=-40,in=-105] (-0.4,1.4) to[out=75,in=-25] (-0.6,2.8) to[out=155,in=120] (-4,3);
\draw[red] (-0.6,2.2) node{\large$x$};
\draw[dashed, thick] (-0.3,1.8) to[out=20,in=130] (3,1.8) to[out=-50,in=80] (2.4,-0.4) to[out=-100,in=40] (-0.8,-2.8) to[out=-140,in=-80] (-2.6,-0.4) to[out=100,in=160] (-0.3,1.8);
\fill[red] (-0.3,1.8) circle (2pt);
\draw[blue] (-2.6,3.1) node{\large$l'$};
\draw[-stealth] (2.4,-0.4) -- (2.8,-1);
\draw[-stealth] (0,-2) -- (-0.5,-1.2);
\draw[] (3,-1.2) node{\large$\tilde l$}
(-0.6,-0.9) node{\large$l$}
;

\fill[blue!10,xshift=100mm]
(3,2) to[in=-60,out=-45] (3,-2) to[in=30, out=120] (-0.5,-3) to[in=-100, out=-150] (-3.5,1.5) to[in=135,out=70] (3,2);

\fill[blue!10,xshift=100mm] (-2,0) to[out=-100,in=-175] (-1.5,-1.5) to[out=15, in=170] (0,-2) to[out=-10,in=-170] (1,-1.2) to[out=10, in=-105] (2.5,1.2) to[out=75, in=-5] (-0.2,1.4) to[out=175,in=80] (-2,0);
\fill[white,xshift=100mm]
(0,-1) to[out=70,in=-95] (1,0.5) to[out=85,in=15] (-0.5,0.5) to[out=-165,in=105] (-0.25,0) to[out=-85,in=-110] (0,-1);
\draw[xshift=100mm] (3,2) to[in=-60,out=-45] (3,-2) to[in=30, out=120] (-0.5,-3) to[in=-100, out=-150] (-3.5,1.5) to[in=135,out=70] (3,2)
(0,-1) to[out=70,in=-95] (1,0.5) to[out=85,in=15] (-0.5,0.5) to[out=-165,in=105] (-0.25,0) to[out=-85,in=-110] (0,-1);
\draw[xshift=100mm]
(-1.5,-1.5) circle (26pt);
\fill[pattern=north west lines, pattern color=blue!60,xshift=100mm]
(-1.5,-1.5) circle (26pt);
\fill[blue!10,xshift=100mm]
(-1.5,-1.5) circle (16pt);
\draw[xshift=100mm]
(-1.5,-1.5) circle (16pt);
\draw[black, very thick,xshift=100mm] (-2,0) to[out=-100,in=-175] (-1.5,-1.5) to[out=15, in=170] (0,-2) to[out=-10,in=-170] (1,-1.2) to[out=10, in=-105] (2.5,1.2) to[out=75, in=-5] (-0.2,1.4) to[out=175,in=80] (-2,0);
\fill[red,xshift=100mm] (-1.5,-1.5) circle (2pt);
\draw[red,xshift=100mm] (-1.5,-1.8) node{\large$x$};

\draw[] (0,-4) node{\large Case $1$} (10,-4) node{\large Case $2$};
\draw[-stealth,xshift=100mm] (-2,-2.1) -- (-2.6,-3);
\draw[xshift=100mm] (-3,-3.2) node{ $A^\star$};
\end{tikzpicture}

    \caption{In Case $1$ where $l'\cap l\neq \emptyset$, we can find an annulus $\tilde A$, shaded by darker blue. In Case $2$ where $l'\cap l= \emptyset$, we construct annulus $A^\star$. Both annuli give a contradiction.}
    \label{fig:pic4.9cases}
\end{figure}
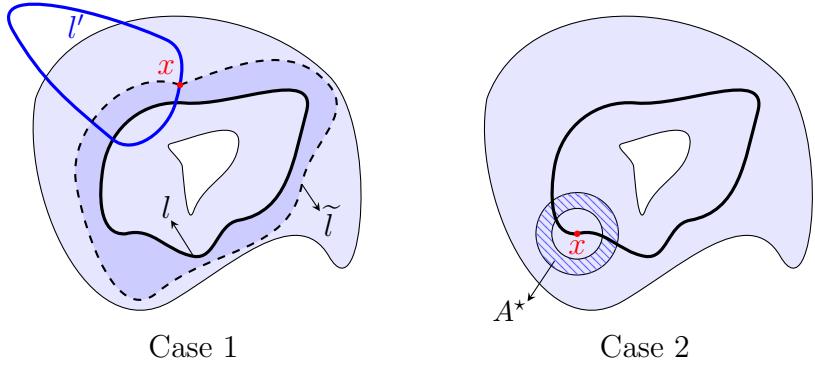

We reach a contradiction in both cases. Therefore, we conclude that $\mathcal{G}(L)\subset \mathcal{G}(L')$.
\end{proof}
\begin{lemma}\label{l.bijection}
The function $F|_{\mathscr{L}_\D^\good}:\mathscr{L}_\D^\good \to \mathscr{H}_{\mathcal{A}_\D}$ is injective.
\end{lemma}

\begin{proof}
For brevity, we write $F = F|_{\mathscr{L}_\D^\good}$ here.
Suppose that $F(L)= F(L')$ for some $L,L' \in \mathscr{L}_\D^\good$. By Lemma~\ref{l.F=G}, $L$ and $L'$ separate the same set of annuli.
Recall the notation $=_{\R^2}$ for loop collections from Section~\ref{s.loop_subset_R^2}.
First, we want to show that $L=_{\R^2}L'$.

We fix any $\tilde l \in L$. Let $\tilde A$ be an annulus separated by $\tilde l$, then $\tilde l$ is contained in the interior of $\tilde A$. 
There is a constant $c>0$ such that, for any loop $l$,
\begin{align}\label{e.diam<c=>no_sep}
    \diam(l) < c\qquad\Longrightarrow \qquad \text{$l$ does not separate $\tilde A$}.
\end{align}
Set $L_c=\{l\in L: \diam(l) \geq c\}$ and $L'_c=\{l'\in L': \diam(l') \geq c\}$. Since both $L_c$ and $L'_c$ are finite and contain disjoint loops, there exists an annulus $A\subset \tilde A$ satisfying
\begin{enumerate}
    \item $\tilde l$ separates $A$ and is contained in the interior of $A$;
    \item $A$ does not intersect any $l \in L_c\setminus\{\tilde l\}$;
    \item $A$ contains at most one loop in $L'_c$.
\end{enumerate}
The first property with $\mathcal{G}(L)=\mathcal{G}(L')$ implies the existence of $\tilde l'\in L'$ that separates $A$. 
Due to~\eqref{e.diam<c=>no_sep}, we have $\tilde l'\in L'_c$. The third property of $A$ shows that $\tilde l'$ is the unique loop in $L'$ that separates $A$. 

Recall the definitions of $\leq_\sep$ and circulation in Section~\ref{s.orders_on_annuli}. We argue that
\begin{align}\label{e.hat_A...=>...}
    \text{$\hat A$ circulates $A$;}\quad \text{$\tilde l$ separates $\hat A$}\qquad\Longrightarrow\qquad \text{$\tilde l'$ separates $\hat A$.}
\end{align}
Assuming the left-hand side, by $\mathcal{G}(L)=\mathcal{G}(L')$, there is $\hat l' \in L'$ separating $\hat A$. Since Lemma~\ref{l.ineq_sep} implies $A\leq_\sep \hat A$, we have that $\hat l'$ separates $A$. The uniqueness of $\tilde l'$ implies that $\tilde l'= \hat l'$, verifying~\eqref{e.hat_A...=>...}.

Using~\eqref{e.hat_A...=>...} for $\hat A$ arbitrarily close to $\tilde l$, we deduce that $\tilde l' \subset \tilde l$. Both of them are simple loops, so $\tilde l'=_{\R^2} \tilde l$ and thus $L\subset_{\R^2} L'$. By symmetry, we get $L=_{\R^2} L'$. Hence, the loops in $L$ and $L'$ are identified with each other up to orientation. Due to $L,L'\in\mathscr{L}_\D^\good$, the loops in $L$ and $L'$ are uniquely oriented. Therefore, we conclude that $L=L'$.
\end{proof}

\subsubsection{A continuous bijection}\label{s.cts_bij}
We restrict $F:\mathscr{L}_\D \to \mathscr{H}_{\mathcal{A}_\D}$ defined in~\eqref{e.F(L)} to
\begin{align}\label{e.F_biject}
    \hat F : \mathscr{L}_\D^\good \to F(\mathscr{L}_\D^\good).
\end{align}
Namely, we set $\hat F=F$ on $\mathscr{L}_\D^\good$ and restrict its range to $F(\mathscr{L}_\D^\good)$.
We endow $\mathscr{L}_\D^\good$ and $F(\mathscr{L}_\D^\good)$ with the subspace topologies induced by $\mathscr{L}_\D$ and $\mathscr{H}_{\mathcal{A}_\D}$, respectively.
\begin{proposition}\label{p.cts_bijection}
The function $\hat F$ is a continuous bijection.
\end{proposition}
\begin{proof}
This follows from Lemma~\ref{l.F_cts} and Lemma~\ref{l.bijection}.
\end{proof}
We denote the inverse of $\hat F$ by $\hat F^{-1}: F(\mathscr{L}_\D^\good)\to \mathscr{L}_\D^\good$. 
We do not expect $\hat F^{-1}$ to be continuous because the topology on $\mathscr{H}_{\mathcal{A}_\D}$ is coarser than that on $\mathscr{L}^\good_\D$ in the sense that $\mathscr{H}_{\mathcal{A}_\D}$ is compact by Lemma~\ref{l.basics_topology} while $\mathscr{L}^\good_\D$ is not.

\subsection{Conformal invariance}\label{s.conformal_inv}

We clarify the meaning of conformal invariance in the spaces described previously. We identify $\R^2$ with $\C$ in the obvious way. 
For any conformal automorphism $\varphi:\D\to\D$, we describe its actions on the elements in $\mathscr{H}_{\mathcal{A}_\D}$ and $\mathscr{L}_\D$. 

We define $\varphi_{\mathscr{H}}:\mathscr{H}_{\mathcal{A}_\D}\to\mathscr{H}_{\mathcal{A}_\D}$ by setting
$\varphi_{\mathscr{H}}(\mathcal{S}) = \{A \in \mathcal{A}_\D:\varphi^{-1}(A)\in \mathcal{S}\}$ for every $\mathcal{S}\in \mathscr{H}_{\mathcal{A}_\D}$. Clearly, $\varphi_{\mathscr{H}}$ is well-defined. Heuristically, for any percolation configuration $\omega$ viewed as a subset of $\R^2$, $A$ is crossed by the image of $\omega$ under $\varphi$ (i.e.\ $A\in\varphi_{\mathscr{H}}(\omega)$) if and only if $\varphi^{-1}(A)$ is crossed by $\omega$ (i.e.\ $\varphi^{-1}(A)\in\omega$). Thus, the definition of $\varphi_{\mathscr{H}}$ is natural.

We define $\varphi_{\mathscr{L}}:\mathscr{L}_\D\to\mathscr{L}_\D$ by setting $\varphi_{\mathscr{L}}(L) = \{\varphi(l)\}_{l\in\ L}$ for every $L\in\mathscr{L}_\D$. Note that $\varphi_{\mathscr{L}}$ maps $\mathscr{L}^\good_\D$ back to $\mathscr{L}^\good_\D$.

For two random variables $X_1$ and $X_2$, we write $X_1\stackrel{\d}{=}X_2$ if they have the same probability distribution. An $\mathscr{H}_{\mathcal{A}_\D}$-valued (resp.\ $\mathscr{L}_\D$-valued) random variable $X$, is said to be \textbf{conformally invariant} if $X \stackrel{\d}{=} \varphi_{\mathscr{H}}(X)$ (resp.\ $X \stackrel{\d}{=} \varphi_{\mathscr{L}}(X)$) for every conformal automorphism $\varphi$ on $\D$.

\begin{lemma}\label{l.conformal_inv}
Let $\eta$ be an $\mathscr{L}_\D$-valued random variable and define $\omega = F(\eta)$. If $\eta$ is conformally invariant, then so is $\omega$.
\end{lemma}
\begin{proof}
Fix any automorphism $\varphi$ and any annulus $A$. Using the definitions of $\varphi_{\mathscr{H}}$, $\varphi_{\mathscr{L}}$, and $F$ again, we have the following identities:
\begin{align*}
    \Ll\{A\in \varphi_{\mathscr{H}}(\omega) \Rr\} = \Ll\{\varphi^{-1}(A)\in\omega \Rr\} = \Ll\{\text{$\varphi^{-1}(A)$ is crossed by some $l\in \eta$}\Rr\} 
    \\
    = \Ll\{\text{$A$ is crossed by some $l\in \varphi_{\mathscr{L}}(\eta)$}\Rr\} = \Ll\{A \in F\Ll(\varphi_{\mathscr{L}}(\eta)\Rr) \Rr\}.
\end{align*}
Therefore, we have $\varphi_{\mathscr{H}}(\omega) = F\Ll(\varphi_{\mathscr{L}}(\eta)\Rr)$, which along with the conformal invariance of $\eta$ and the continuity of $F$ by Lemma~\ref{l.F_cts} implies the conformal invariance of $\omega$.
\end{proof}

\section{Identification of limits}\label{s.id_limits}

Recall that, for each $\delta>0$, $(\eta_\delta,\omega_\delta)$ are random percolation configurations sampled from $\P^\emptyset_{\D_\delta}$. Recall the event $\circledcirc_A^{\eta_\delta}$ defined in~\eqref{e.circledcirc_A^s}. We define
\begin{align*}
    \eta_\delta^\cro =\Ll\{A \in \mathcal{A}_\D:\: \text{$\circledcirc_A^{\eta_\delta}$ happens}\Rr\}
\end{align*}
and $\omega_\delta^\cro$ in the same way. It is straightforward to see that every realizations of $\eta^\cro_\delta$ and $\omega^\cro_\delta$ are hereditary as defined in~\eqref{e.hereditary}, which is equivalent to $\eta^\cro_\delta, \omega^\cro_\delta \in \mathscr{H}_{\mathcal{A}_\D}$. In this notation, we can rewrite the crossing event as
\begin{align}\label{e.equiv_cross_A}
    \circledcirc_A^{\eta_\delta} = \{\eta_\delta^\cro \in \circledcirc_A\}=\{A \in \eta^\cro_\delta\}
\end{align}
and similarly for $\omega_\delta$, where $\circledcirc_A\subset \mathscr{H}_{\mathcal{A}_\D}$ is given in~\eqref{e.def_open_sets}.

Throughout this section, for each $\delta>0$, we denote by $\eta^\lop_\delta$ some loop decomposition of $\eta_\delta$ whose existence is given by Lemma~\ref{l.kappa^lop_loop_decomposition} (see Remark~\ref{r.eta^lop}).
Weaker uniqueness holds for $\eta_\delta^\lop$ as described in Lemma~\ref{l.lop_unique}.
We first record a basic result.

\begin{lemma}\label{l.eta^cro=F(eta^lop)}
For every $\delta>0$ and every realization $\eta_\delta$, it holds that $\eta^\cro_\delta = F(\eta^\lop_\delta)$ for $F$ defined in~\eqref{e.F(L)}.
\end{lemma}
\begin{proof}
Let $A$ be any annulus in $\mathcal{A}_\D$. Assuming $A\in \eta^\cro_\delta$, we have that $A$ is crossed by a connected set $S$ consisting of open edges in $\eta_\delta$. By definition, $\eta^\lop_\delta$ consists of disjoint loops. Also, the property in~\eqref{e.kappa_e=1<=>...} ensures that all the open edges are on the loops in $\eta^\lop_\delta$. Hence, the connectedness of $S$ implies that $S\subset l$ for some $l\in \eta^\lop_\delta$. So $A$ is crossed by $l$ and thus $A\in F(\eta^\lop_\delta)$, showing that $\eta^\cro_\delta \subset F(\eta^\lop_\delta)$.

Now, assuming $A\in F(\eta^\lop_\delta)$, we have some $l\in\eta^\lop_\delta$ crossing $A$. Then, the event $\circledcirc_A^{\eta_\delta}$ happens, which along with~\eqref{e.equiv_cross_A} yields $A\in \eta^\cro_\delta$, giving $ F(\eta^\lop_\delta)\subset \eta^\cro_\delta$. Hence, the proof is complete.
\end{proof}

We view $\eta^\cro_\delta$ and $\omega^\cro_\delta$ as $\mathscr{H}_{\mathcal{A}_\D}$-valued random variables and $\eta^\lop_\delta$ as a $\mathscr{L}_\D$-valued random variable. Since they are induced by $(\eta_\delta, \omega_\delta)$, they are naturally coupled. 

\subsection{Convergence of annulus-crossing events}

For random variables $(X_\delta)_{\delta>0}$ and $X_0$ taking value in a common metric space $\mathscr{X}$, $X_\delta$ is said to \textbf{converge in distribution in $\mathscr{X}$} to $X_0$ if the probability distribution induced by $X_\delta$ on $\mathscr{X}$ converges weakly to that of $X_0$. 

Henceforth, for each $\delta>0$, we consider the $\mathscr{L}_\D \times \mathscr{H}_{\mathcal{A}_\D}\times \mathscr{H}_{\mathcal{A}_\D}$-valued random variable $(\eta^\lop_\delta,\eta^\cro_\delta,\omega^\cro_\delta)$ induced by $\P^\emptyset_\D$. We often assume that it converges in distribution, possibly along a subsequence, to some $(\eta^\lop_0,\eta^\cro_0,\omega^\cro_0)$. Skorokhod's representation theorem is used to embed these random variables on a common probability space, whose probability measure we denote by $\mathbf{P}$.
We denote by $\triangle$ the symmetric difference operator on sets.

We show that convergence in distribution implies convergence of probabilities of annulus-crossing events.

\begin{lemma}\label{l.ann-cro_char_lim}
Assume that $(\eta_\delta,\omega_\delta)$ satisfies~\ref{i.H_basics}--\ref{i.H_bdy_conn}. If $\omega^\cro_\delta$ converges in distribution in $\mathscr{H}_{\mathcal{A}_\D}$ to some $\omega^\cro_0$ as $\delta \to 0$, then for every $A\in \mathcal{A}_{\D}^{\mathrm{dya}}$,
\begin{align*}
    \lim_{\delta\to 0} \mathbf{P}\Ll(\Ll\{\omega^\cro_\delta\in \circledcirc_A\Rr\} \triangle \Ll\{\omega_0^\cro\in \circledcirc_A\Rr\}\Rr) =0.
\end{align*}
\end{lemma}

\begin{proof}
For brevity, we abuse the notation and write $\omega_\delta = \omega^\cro_\delta$ for $\delta\geq 0$ and $\circledcirc^{\omega_0}_A = \{\omega_0^\cro\in \circledcirc_A\}$ even though no percolation configuration $\omega_0$ induces $\omega_0^\cro$. 
By Skorokhod's representation theorem, $\omega_\delta$ converges to $\omega_0$ in $\mathscr{H}_{\mathcal{A}_\D}$ a.s.\ under $\mathbf{P}$.

First, we show that
\begin{align}\label{e.limP(omega_delta,omega_0)}
    \lim_{\delta\to 0} \mathbf{P}\Ll( \circledcirc_A^{\omega_\delta}\setminus\circledcirc_A^{\omega_0} \Rr)  =0.
\end{align}
If $\limsup_{\delta\to0} \mathds{1}_{\circledcirc_A}(\omega_\delta)=1$, then $(\omega_\delta)_{\delta>0}$ has a limit in $\circledcirc_A$ since $\circledcirc_A$ is closed by the choice of the topology $\mathbf{T}_{\mathcal{A}_\D}$. Hence, $\mathds{1}_{\circledcirc_A}(\omega_0)=1$ and thus
\begin{align}\label{e.lim1-1v0}
    \lim_{\delta\to0}\Ll(\mathds{1}_{\circledcirc_A}(\omega_\delta)-  \mathds{1}_{\circledcirc_A}(\omega_0)\Rr)\vee0= 0.
\end{align}
The only other possibility is $\limsup_{\delta\to0} \mathds{1}_{\circledcirc_A}(\omega_\delta)=0$ in which case~\eqref{e.lim1-1v0} still holds. After taking the expectation,~\eqref{e.lim1-1v0} implies~\eqref{e.limP(omega_delta,omega_0)}.

We next turn to the complement bound. For $k\in\N$ to be determined, we have
\begin{align*}
    \mathbf{P}\Ll( \circledcirc_A^{\omega_0}\setminus\circledcirc_A^{\omega_\delta} \Rr) \leq \mathbf{P}\Ll( \circledcirc_A^{\omega_0}\setminus\circledcirc_A^{\omega_\delta},\, d_{\mathscr{H}_{\mathcal{A}_\D}}(\omega_\delta,\omega_0)\leq \mathfrak{r}(k)\Rr) + \mathbf{P}\Ll(d_{\mathscr{H}_{\mathcal{A}_\D}}(\omega_\delta,\omega_0)> \mathfrak{r}(k)\Rr).
\end{align*}
Since $A$ is dyadic, there is $k_A\in \N$ such that $A \in \mathcal{A}^{k_A}_\D$. Let $k\geq k_A$.
By Lemma~\ref{l.metric_dya} and~\eqref{e.O^k(omega)}, if $d_{\mathscr{H}_{\mathcal{A}_\D}}(\omega_\delta,\omega_0) \leq \mathfrak{r}(k)$, then $\omega_0\in\circledcirc_A$ implies $\omega_\delta \in \mathscr{V}_{\mathcal{B}^k_A}$. Hence, there is $A'\in \mathcal{B}^k_A$ such that $\omega_\delta \in \circledcirc_{A'}$. This gives that
\begin{align*}
     \mathbf{P}\Ll( \circledcirc_A^{\omega_0}\setminus\circledcirc_A^{\omega_\delta},\, d_{\mathscr{H}_{\mathcal{A}_\D}}(\omega_\delta,\omega_0)\leq \mathfrak{r}(k)\Rr) \leq \mathbf{P}\Ll(\circledcirc_{A'}^{\omega_0}\setminus\circledcirc_{A}^{\omega_\delta},\,\text{for some }A'\in \mathcal{B}^k_A\Rr). 
\end{align*}
Using Lemma~\ref{l.d<r,in()_r} and the definition of $\mathcal{B}^k_A$ above~\eqref{e.O^k(omega)},
for any $r>0$, we can choose $k=k(r)\geq k_A$ sufficiently large so that for all $A'\in \mathcal{B}^k_A$, we have $\partial_i A' \in (\partial_i A)_r$ for both $i=0,1$. 
Due to $\eta_\delta\subset \omega_\delta$, we also have $  \circledcirc_{A}^{\eta_\delta} \subset \circledcirc_{A}^{\omega_\delta}$. These yield
\begin{align*}
    \Ll\{\circledcirc_{A'}^{\omega_0}\setminus\circledcirc_{A}^{\omega_\delta},\,\text{for some }A'\in \mathcal{B}^k_A \Rr\}\subset \Ll\{ (\partial_0 A)_r \stackrel{\omega_\delta\cap A}{\longleftrightarrow} \Ll(\partial_1 A\Rr)_r  \Rr\}\cap \Ll(\circledcirc_A^{\eta_\delta}\Rr)^\complement.
\end{align*}

Hence, we obtain
\begin{align*}\mathbf{P}\Ll( \circledcirc_A^{\omega_0}\setminus\circledcirc_A^{\omega_\delta} \Rr) \leq \mathbf{P}\Ll(\Ll\{ (\partial_0 A)_r \stackrel{\omega_\delta\cap A}{\longleftrightarrow} \Ll(\partial_1 A\Rr)_r  \Rr\}\cap \Ll(\circledcirc_A^{\eta_\delta}\Rr)^\complement\Rr) + \mathbf{P}\Ll(d_{\mathscr{H}_{\mathcal{A}_\D}}(\omega_\delta,\omega_0)> \mathfrak{r}(k)\Rr).
\end{align*}
First, the second term on the right-hand side vanish by sending $\delta\to0$. Then, the first term on the right vanishes when $r$ is sent to $0$ as a result of Lemma~\ref{l.discon_bdy}. This shows that $\lim_{\delta\to0}\mathbf{P}\{ \circledcirc_A^{\omega_0}\setminus\circledcirc_A^{\omega_\delta} \}=0$, which together with ~\eqref{e.limP(omega_delta,omega_0)} completes the proof.
\end{proof}

\begin{corollary}\label{c.eta_ann_cross}
Assume that $(\eta_\delta,\omega_\delta)$ satisfies~\ref{i.H_basics}--\ref{i.H_bdy_conn}. If $\eta_\delta^\cro$ converges in distribution in $\mathscr{H}_{\mathcal{A}_\D}$ to some $\eta_0^\cro$ as $\delta\to0$, then, for every $A\in\mathcal{A}_{\D}^{\mathrm{dya}}$,
\begin{align*}
    \lim_{\delta\to 0} \mathbf{P}\Ll(\Ll\{\eta^\cro_\delta\in \circledcirc_A\Rr\} \triangle \Ll\{\eta^\cro_0\in \circledcirc_A\Rr\}\Rr) =0.
\end{align*}
\end{corollary}
\begin{proof}
Substituting $\omega_\delta$ by $\eta_\delta$, we verify that the pair $(\eta_\delta,\eta_\delta)$ also satisfies~\ref{i.H_basics}--\ref{i.H_bdy_conn}.
Applying Lemma~\ref{l.ann-cro_char_lim} to $(\eta_\delta,\eta_\delta)$, we obtain the result.
\end{proof}
Lastly, we show that asymptotically $\eta^\cro_\delta$ and $\omega^\cro_\delta$ cross the same annuli.
\begin{lemma}\label{e.ann_cro_omega-eta}
Assume that $(\eta_\delta,\omega_\delta)$ satisfies~\ref{i.H_basics}--\ref{i.H_bdy_conn}.
Then for every $A \in \mathcal{A}_{\D}^{\mathrm{dya}}$,
\begin{align*}
    \lim_{\delta\to\infty}  \P^\emptyset_{\D_\delta}\Ll(\Ll\{\omega^\cro_\delta\in \circledcirc_A\Rr\}\triangle\Ll\{\eta^\cro_\delta\in \circledcirc_A\Rr\} \Rr)=0.
\end{align*}
\end{lemma}
\begin{proof}
Since $\omega_\delta$ and $\eta_\delta$ are coupled and $\omega_\delta \supset \eta_\delta$, 
the symmetric difference is equal to $\circledcirc_A^{\omega_\delta}\setminus \circledcirc_A^{\eta_\delta}$.
Note that $\circledcirc_A^{\omega_\delta}\setminus \circledcirc_A^{\eta_\delta}$ is a subset of the event in~\eqref{e.limlimsupP<->} for any $r>0$. The desired result follows from Lemma~\ref{l.discon_bdy}.
\end{proof}

\subsection{Proof of Result~\ref{resultA}}
 
For two random variables $X_1$ and $X_2$ possibly on different probability spaces, a \textbf{coupling} of $X_1$ and $X_2$ is a pair $(Y_1,Y_2)$ of random variables defined on a common probability space such that $Y_i \stackrel{\d}{=} X_i$ for $i\in\{1,2\}$.

Recall the definitions of the function $F$ in~\eqref{e.F(L)} and its restriction $\hat F$ in~\eqref{e.F_biject}.

The following is the precise restatement of Result~\ref{resultA}, which is slightly stronger since it gives the joint convergence of all three random variables.

\begin{theorem}\label{t.(eta,omega)_id_lim}
Assume that $(\eta_\delta,\omega_\delta)$ satisfies~\ref{i.H_basics}--\ref{i.H_bdy_conn}.
If $\eta_\delta^\lop$ converges in distribution in $\mathscr{L}_\D$ to some $\eta_0^\lop$, then $(\eta^\lop_\delta,\eta^\cro_\delta,\omega_\delta^\cro)$ converges in distribution in $\mathscr{L}_\D\times\mathscr{H}_{\mathcal{A}_\D}\times \mathscr{H}_{\mathcal{A}_\D}$ to $(\eta^\lop_0,F(\eta^\lop_0),F(\eta^\lop_0))$. 

If, in addition, $\eta^\lop_0\in\mathscr{L}_\D^\good$ a.s., then this gives a coupling between $\eta^\lop_0$ and the limit $\omega^\cro_0$ of $\omega^\cro_\delta$ such that $\omega^\cro_0 = \hat F(\eta^\lop_0)$ and $\hat F^{-1}(\omega^\cro_0) = \eta^\lop_0$ a.s.

\end{theorem}

\begin{proof}
We show that for every subsequence of $\{\delta\}$ there is a further subsequence along which $(\eta_\delta^\lop,\eta_\delta^\cro, \omega_\delta^\cro)$ converges to the desired limit.
By the convergence of $\eta^\lop_\delta$ and the compactness of $\mathscr{H}_{\mathcal{A}_\D}$ as in Lemma~\ref{l.basics_topology}, the marginal laws of $(\eta_\delta^\lop,\eta_\delta^\cro, \omega_\delta^\cro)$ are tight. Since marginal tightness implies joint tightness,
for any subsequence of $(\delta)$, there is a further subsequence along which $(\eta_\delta^\lop,\eta_\delta^\cro, \omega_\delta^\cro)$ converges to some $(\eta_0^\lop,\eta_0^\cro, \omega_0^\cro)$. 
By Skorokhod's representation theorem, this convergence can be improved to the almost-sure convergence under $\mathbf{P}$.

This convergence along with the identity $\eta_\delta^\cro=F(\eta^\lop_\delta)$ due to Lemma~\ref{l.eta^cro=F(eta^lop)} and the continuity of $F$ in Lemma~\ref{l.F_cts} implies that $\eta^\cro_0 = F(\eta^\lop_0)$ a.s.
For every $A \in \mathcal{A}_{\D}^{\mathrm{dya}}$, we have
\begin{align*}
    \mathbf{P}\Ll(\Ll\{\omega^\cro_0\in \circledcirc_A\Rr\} \triangle \Ll\{\eta^\cro_0\in \circledcirc_A\Rr\}\Rr) 
    &\leq \mathbf{P}\Ll(\Ll\{\omega^\cro_0\in \circledcirc_A\Rr\} \triangle \Ll\{\omega^\cro_\delta\in \circledcirc_A\Rr\}\Rr) 
    \\
    &+ \mathbf{P}\Ll(\Ll\{\omega^\cro_\delta\in \circledcirc_A\Rr\} \triangle \Ll\{\eta^\cro_\delta\in \circledcirc_A\Rr\}\Rr) 
    \\
    &+ \mathbf{P}\Ll(\Ll\{\eta^\cro_\delta\in \circledcirc_A\Rr\} \triangle \Ll\{\eta^\cro_0\in \circledcirc_A\Rr\}\Rr).
\end{align*}
Using Lemma~\ref{l.ann-cro_char_lim}, Lemma~\ref{e.ann_cro_omega-eta}, and Corollary~\ref{c.eta_ann_cross} to bound the terms on the right and sending $\delta\to0$, we obtain
\begin{align*}
    \mathbf{P}\Ll(\Ll\{\omega^\cro_0\in \circledcirc_A\Rr\} \triangle \Ll\{\eta^\cro_0\in \circledcirc_A\Rr\}\Rr) =0,\quad\forall A \in \mathcal{A}_{\D}^{\mathrm{dya}}.
\end{align*}
Then, the second part in Lemma~\ref{l.basics_topology} implies that $\omega_0^\cro =\eta_0^\cro$ a.s. 
Therefore, we conclude that $(\eta_0^\lop,\eta_0^\cro, \omega_0^\cro) = (\eta^\lop_0,F(\eta^\lop_0),F(\eta^\lop_0))$ a.s., proving the main part of the theorem.
The remaining follows from the bijectivity of $\hat F$ in Proposition~\ref{p.cts_bijection} and the definition of $\hat F$ as the restriction of $F$ to $\mathscr{L}^\good_\D$.
\end{proof}

\section{Application to the random current}\label{s.app_random_current}

\subsection{Setting and basics}
Fix any $\delta>0$.
We describe various random variables on $\D_\delta$ and its dual disc.
For simplicity, we often omit the subscript $\delta$ in the following.

\subsubsection{Random currents}
Let
\begin{align}\label{e.beta_c}
    \bbeta_\crit = \frac{1}{2}\log\Ll(1+\sqrt{2}\Rr)
\end{align}
be the critical temperature for the two-dimensional Ising model.

A \textbf{current} $\bn = (\bn_e)_{e\in\EE(\D_\delta)}:  (\D_\delta)^\EE\to\N\cup\{0\}$ is a nonnegative-integer valued function on the edges.
The \textbf{source} of $\bn$ is defined to be the set $\partial\bn = \{x\in \VV(\D_\delta) : \sum_{y:\:y\sim x}\bn_{xy}\text{ is odd}\}$.
We define the \textbf{trace} $\hat\bn = (\hat\bn_e)_{e\in \EE(\D_\delta)}$ of $\bn$ by setting $\hat\bn_e = \mathds{1}_{\bn_e \neq 0}$ which is a percolation configuration. To each $\bn$, we assign the weight
\begin{align*}
    w(\bn) = \prod_{e\in \EE(\D_\delta)}\frac{\bbeta_\crit^{\bn_e}}{\bn_e !},
\end{align*}
which induces a probability measure
\begin{align*}
    \mathbf{P}^\emptyset_{\D_\delta} \Ll(\bn \in \mathcal{E}\Rr) = \frac{\sum_{\bn:\:\bn\in\mathcal{E},\,\partial\bn=\emptyset}w(\bn)}{\sum_{\bn:\:\partial\bn=\emptyset}w(\bn)}
\end{align*}
on sourceless currents. We call the random variable $\bn$ under $\mathbf{P}^\emptyset_{\D_\delta}$ the \textbf{(critical) sourceless random current (on $\D_\delta$)}. Random current arises in an expansion of correlations in the Ising model on $\VV(\D_\delta)$, and it is a powerful tool in the study of the Ising model. We refer to \cite{duminil2016random} and the references therein for more details.

\subsubsection{High-temperature expansion}
The \textbf{high-temperature expansion} gives rise to a percolation configuration $\eta=(\eta_e)_{e\in \EE(\D_\delta)}$ with $\eta_e = \mathds{1}_{\{\bn_e \text{ is odd}\}}$. Recall the definition of sources of percolation configurations in~\eqref{e.source}, which shows that $\partial_{\D_\delta}\eta = \partial \bn$.
We denote the pushforward of $\mathbf{P}^\emptyset_{\D_\delta}$ under the map $\bn\mapsto \eta$ by $\P^\emptyset_{\D_\delta}$, which admits the representation
\begin{align}\label{e.eta_weight}
    \P^\emptyset_{\D_\delta} \Ll(\eta \in \mathcal{E}\Rr) = \frac{\sum_{\eta:\:\eta\in\mathcal{E},\,\partial_{\D_\delta}\eta=\emptyset}w(\eta)}{\sum_{\eta:\:\partial_{\D_\delta}\eta=\emptyset}w(\eta)},
\end{align}
where 
\begin{align*}
    w(\eta) = \prod_{e\in \EE(\D_\delta):\eta_e=1}\tanh(\bbeta_\crit).
\end{align*}

\subsubsection{Another coupling of $\eta$ and $\hat\bn$}

Recall from Section~\ref{s.bernoulli_pert} the notation of the Bernoulli perturbation $\bb^{\bt_\delta}_\delta$ and the short-hand $\bb^t_\delta$ for constant $t$. In this part, we write them as $\bb^\bt$ and $\bb^t$ for brevity and assume they are independent from $\eta$ and $\bn$.
Recall the notation in~\eqref{e.aveeb}. We write $ \eta\vee \bb^t  = (\eta_e\vee\bb^t_e)_{e\in \EE(\D_\delta)}$. 
By \cite[Theorem~3.2]{aizenman2019emergent}, we have that
\begin{align}\label{e.eta_vee_b}
    \eta\vee \bb^{t_\crit}\stackrel{\d}{=}\hat\bn
\end{align}
for $t_\crit = 1-\frac{1}{\cosh(\bbeta_\crit)}$, where $\bbeta_\crit$ is given in~\eqref{e.beta_c}.
By extending the probability space, we assume that $\eta$, $\hat\bn$, and $\bb^{\bt}$ are coupled under $\P^\emptyset_{\D_\delta}$ for every $\bt\in [0,1]^{\E(D_\delta)}$.

The coupling~\eqref{e.eta_vee_b} is related to the one between $\eta$ and the FK--Ising random cluster model proved in \cite{grimmett2009random} and the one between $\hat\bn$ and the FK--Ising presented in \cite{lupu2016note}.

\subsubsection{Ising model on the dual domain}\label{s.Ising_dual_graph}

Recall the definition of dual domain in Section~\ref{s.dual_graphs}. Let $\D^*_\delta$ be the dual domain of $\D_\delta$. A \textbf{spin configuration} on $\D_\delta^*$ is an element $\sigma^*= (\sigma^*_x)_{x\in \VV(\D_\delta^*)}$ in $\{-1,1\}^{\VV(\D^*_\delta)}$. Let the weight of $\sigma^*$ be given by
\begin{align}\label{e.Ising_weight}
    w(\sigma^*) =\exp\Ll(\bbeta_\crit \sum_{x y\in\EE(\D^*_\delta) } \sigma^*_x \sigma^*_y\Rr)
\end{align}
and set $\Omega^+ = \{\sigma^*:\:\sigma^*_x=+1,\,\forall x\in\partial \D^*_\delta\}$.
We define
\begin{align*}
    \P^+_{\D^*_\delta}\Ll\{\sigma^* \in \mathcal{E}\Rr\} = \frac{\sum_{\sigma^*\in\mathcal{E}\cap\Omega_\delta^+}w(\eta)}{\sum_{\sigma^*\in\Omega_\delta^+}w(\eta)}
\end{align*}
to be the \textbf{(critical) Ising measure with $+$ boundary condition (on $\D^*_\delta$)}. In the following, $\sigma^*$ is the random variable with distribution $\P^+_{\D^*_\delta}$.

Given a spin configuration $\sigma^*$, we define a percolation configuration $\eta$ on the primal domain. For every $e\in\EE(\D_\delta)$, $\eta_e=\mathds{1}_{\sigma^*_x\neq \sigma^*_y}$ where $xy=e^*$ is the dual edge of $e$.
In other words, an edge is open in $\eta$ if and only if it is on the interface of different types of spins in $\sigma^*$.
By the Kramers--Wannier duality \cite{kramers1941statistics}, the pushforward of $\P^+_{\D^*_\delta}$ under the map $\sigma^*\mapsto\eta$ is $\P^\emptyset_{\D_\delta}$.
In the following, we write $\bn_\delta$, $\eta_\delta$, $\bn^{\bt_\delta}_{\delta}$, and $\sigma^*_\delta$ to emphasize the dependence on $\delta$.

\subsubsection{Conformal loop ensembles (CLEs)}\label{s.cle}

The CLE measures introduced in \cite{sheffield2009exploration} are the probability measures on the loop collections to describe the scaling limits of the interfaces in the planar statistical mechanics models expected to be invariant under conformal maps. They are indexed by $\kappa\in (\frac{8}{3},8]$ with the corresponding measure $\cle(\kappa)$, under which the loops locally behave like $\mathrm{SLE}(\kappa)$ curves. For $\kappa\in (\frac{8}{3},4]$, $\cle(\kappa)$ admits a loop-soup construction and characterized uniquely by the Markovian restriction property \cite{sheffield2012conformal}.

For our purpose,
we follow \cite[Section~2.11]{benoist2019scaling} to define the nested version of $\cle(\kappa)$ for $\kappa\in (\frac{8}{3},4]$.
Recall the definitions of $\mathscr{L}_\D$ and $\mathscr{L}^\good_\D$ in Sections~\ref{s.metric_on_loop_coll} and~\ref{s.coll_simple}, respectively.
Let $L_1$ be the random variable in $\mathscr{L}_\D$ sampled from $\cle(\kappa)$ on the domain $D$. Due to $\kappa\leq 4$, the loops in $L_1$ are disjoint, simple, and of level $1$ as defined in Section~\ref{s.coll_simple} and away from the boundary a.s.. Assuming that $L_k$ for some $k\in\N$ is constructed and satisfies these properties, we sample independently from $\cle(\kappa)$ inside every $l \in L_k$ and collect the new loops into $L_{k+1}$. Clearly, $L_{k+1}$ satisfies the desired properties. 

Inductively, we construct a sequence $(L_{k})_{k=1}^\infty$ of random loop collections. For every $k$, we orient the loops in $L_{k}$ according to the rule in~\eqref{e.orientation}. Lastly, we set $L^{\cle(\kappa)} = \cup_{k=1}^\infty L_k$, which implies that $L^{\cle(\kappa)}$ is a random variable in $\mathscr{L}_\D$ and is $\mathscr{L}_\D^\good$-valued a.s. We call the distribution of $L^{\cle(\kappa)}$ the \textbf{nested $\cle(\kappa)$ on $\D$}.

\subsection{Ising loops and convergence}

Let $\eta^\lop_\delta$ be given as in~\eqref{e.kappa^loop} constructed in 
Sections~\ref{s.loop_construct_1},~\ref{s.loop_construct_2} and~\ref{s.loop_construct_3}.
By Lemma~\ref{l.kappa^lop_loop_decomposition} and Remark~\ref{r.eta^lop}, $\eta^\lop_\delta$ is a loop decomposition of $\eta_\delta$.
The choice of $\eta^\lop_\delta$ here is explained in Remark~\ref{r.eta^lop}.

We want to show the convergence of $\eta^\lop_\delta$ using the results from \cite{benoist2019scaling} on the convergence of the critical Ising interface to~$\cle(3)$. To do so, we match the definitions in the construction of $\eta^\lop_\delta$ to those in \cite{benoist2019scaling} for Ising loops.

\subsubsection{Ising loops}

We recall the definitions in \cite[Section~2.5]{benoist2019scaling} on Ising loops. Note that in \cite{benoist2019scaling} the Ising interface is on the dual lattice; here, $\eta_\delta$ is on the primal lattice. We adapt the definitions therein to the primal lattice in the obvious way.

An \textbf{Ising loop} $l$ is a loop on $\D_\delta$ such that
\begin{itemize}

    \item no edge is used twice,
    \item and every edge in the loop has a $+$ spin of $\sigma^*$ on its left and a $-$ spin of $\sigma^*$ on its right. 
\end{itemize}
We set $L^\mathrm{Is}$ to be the collection of all the Ising loops.
A sequence $(y_0, y_1,\dots,y_n)$ of vertices in $\D^*_\delta$ is called a \textbf{strong} (resp.\ \textbf{weak}) \textbf{loop of $\pm$ spins} if $y_0=y_n$, $|y_i-y_{i+1}|=\delta$ (resp.\ $|y_i-y_{i+1}|_\infty=\delta$), and $\sigma^*_{y_i} = \pm1$ for all $i$.
An Ising loop is called \textbf{leftmost} (resp.\ \textbf{rightmost}) if it follows on its left (resp.\ right) side a strong loop of $+$ (resp.\ $-$) spins.

We introduce the following notation: for each $i\in\N\cup\{0\}$, let
\begin{align*}
    \textit{$s(i)=+$ and $\mathbf{s}(i)=\mathrm{left}$ if $i$ is odd; $s(i)=-$ and $\mathbf{s}(i)=\mathrm{right}$ if $i$ is even}
\end{align*}
We set $L_0=\{\partial \D_\delta\}$. Inductively, for each $i\in\N\cup\{0\}$, we define $L^\mathrm{Is}_{i+1}$ to be the collection of Ising loops $l\in L^\mathrm{Is}\setminus \cup_{j=1}^i L^\mathrm{Is}_j$ such that
\begin{itemize}
    \item $l$ is inside some $l'\in L_{i}$,
    \item there is no weak loop of $s(i)$ spins separating $l$ and $l'$.
\end{itemize}
For $i\geq 1$, Ising loops in $L^\mathrm{Is}_i$ are said to be of \textbf{level-$i$}. We set $L^{\mathrm{Is},\,\most}_i$ to be the collection of loops in $L^\mathrm{Is}_i$ that are $\mathbf{s}(i)$-most.

\subsubsection{Matching definitions}

Recall the definitions of $L^\mathrm{all}$, $L_i$, $L^\most_i$ (in Section~\ref{s.loop_construct_1}), and $L^\conc_i$ (in Section~\ref{s.loop_construct_3}) constructed for $\eta_\delta$ on $\D_\delta$.

\begin{lemma}\label{l.match_Ising_loop}
For every $i\geq 1$, we have $L^\mathrm{Is}_i = L_i$ and $L^{\mathrm{Is},\,\most}_i = L^\most_i$.
\end{lemma}

\begin{proof}
The first bullet point in the definition of Ising loops ensures that all edges are distinct; and the second bullet point implies that Ising loops are weakly simple. Conversely, a weakly simple loop on $\eta_\delta$ with the correct orientation is an Ising loop. 
Therefore, $L^\mathrm{Is}$ is equal to $L^\mathrm{all}$ up to orientation, which is to be fixed according to~\eqref{e.orientation}.

We next show the first identity. The first bullet points in the definitions of $L_i$ and $L^\mathrm{Is}_i$ are the same. The second bullet point of $L^\mathrm{Is}_i$ is equivalent to requiring no Ising interface separating $l$ and $l'$, which is the second bullet point of $L_i$. Lastly, by definition, the loops in $L^\mathrm{Is}_i$ are oriented just as in~\eqref{e.orientation}. Hence, we conclude that $L^\mathrm{Is}_i = L_i$.

Each $l\in L^\most_i$ is surrounded by a simple loop $l'$ on the dual lattice. In the Ising setting, $\sigma^*$ has $s(i)$ spins on vertices of $l'$, which is the condition for $L^{\mathrm{Is},\,\most}_i$. Then, the second identity follows from the first identity.
\end{proof}

\subsubsection{Convergence of $\eta^\lop_\delta$}

Let $L^{\cle(3)}$ be a $\mathscr{L}^\good_\D$-valued random variable distributed as nested $\cle(3)$, which is described in Section~\ref{s.cle}. The main result \cite[Theorem~6]{benoist2019scaling} in \cite{benoist2019scaling} states that
\begin{itemize}
    \item as $\delta\to 0$, the collection of leftmost Ising loops converges in distribution to $L^{\cle(3)}$ in $\mathscr{L}_\D$;
    \item for every $\eps>0$, the following happens with probability tending to $1$ as $\delta\to0$: every Ising loop with diameter large than $\eps$ is within $\eps$ distance away, measured in $d_{\mathcal{L}_D}$ defined in~\eqref{e.d_mathcal_L}, from a leftmost Ising loop.
\end{itemize}
However, we cannot deduce the convergence of $\eta_\delta^\lop$ directly from these two results because the first gives the convergence of a smaller set of loops and the second is not sufficient to construct an optimal matching between the leftmost loops and $\eta^\lop_\delta$.

Hence, we look into the proof of \cite[Theorem~6]{benoist2019scaling} and use ingredients there to show the next lemma.

\begin{lemma}\label{l.eta^loop->CLE}
As $\delta\to0$, $\eta_\delta^\lop$ converges in distribution to $L^{\cle(3)}$ in $\mathscr{L}_\D$.
\end{lemma}

\begin{proof}
The proof follows by modifying the proof in \cite[Theorem~6]{benoist2019scaling}, which shows the convergence of level-$1$ leftmost Ising loops and then uses the spatial Markov property to reach higher levels. That proof consists of three paragraphs focusing on different aspects.

The first paragraph in \cite[Theorem~6]{benoist2019scaling} shows that $L_1^{\mathrm{Is},\,\most}$ converges in distribution in $\mathscr{L}_\D$ to $\cle(3)$.
We want to show the convergence of $L_1^\conc$.
For every $l^\conc\in L_1^\conc$, let $l^\most\in L_1^\most$ be the element given by the bijection in Lemma~\ref{l.L^conc}~\eqref{i.L^conc_biject}. 
The proof of \cite[Lemma~14]{benoist2019scaling} shows that for any fixed level-$1$ Ising loop $l^\mathrm{L}$ with $\diam(l^\mathrm{L})>\eps$, every level-$1$ Ising loop $l$ inside $l^\mathrm{L}$ with $\diam(l)>\eps$ satisfies $d_{\mathcal{L}_{\D}}(l,l^\mathrm{L})\leq \eps$, with probability tending to $1$ as $\delta\to0$.

By Lemma~\ref{l.match_Ising_loop} and Lemma~\ref{l.L^conc}~\eqref{i.L^conc_inside}, we can apply this to $l^\conc$ and $l^\most$. Then, Lemma~\ref{l.L^conc}~\eqref{i.L^conc_diam} implies that with probability tending to $1$, for every $l^\most$ that $\diam(l^\most)>\eps$, we have
\begin{align*}
    d_{\mathcal{L}_{\D}}\Ll(l^\conc,l^\most\Rr)\leq \eps.
\end{align*}
For any $\eps>0$, we build a matching by matching $l^\conc$ with $l^\most$ if $\diam(l^\most)>\eps$. Then, Lemma~\ref{l.L^conc}~\eqref{i.L^conc_diam} and~\eqref{e.d_mathscr_L} imply that
\begin{align*}
    d_{\mathscr{L}_\D}\Ll(L^\conc_1,L^\most_1\Rr) \leq \sup_{l^\most:\: \diam(l^\most)>\eps}d_{\mathcal{L}_\D}\Ll(l^\conc,l^\most\Rr) + \frac{1}{2}\eps.
\end{align*}
The proof of \cite[Lemma~4]{benoist2019scaling} shows that with high probability, the number of level-$1$ leftmost Ising loops of diameter greater than $\eps$ is bounded uniformly in $\delta$. Using this and the above two displays, we get
\begin{align*}
    \lim_{\delta\to0}\P^\emptyset_{\D_\delta}\Ll\{ d_{\mathscr{L}_\D}\Ll(L_1^\most,\, L_1^\conc\Rr)>\eps\Rr\}=0,
\end{align*}
which along with the convergence of $L_1^{\mathrm{Is},\,\most}$ and Lemma~\ref{l.match_Ising_loop} implies that $L_1^\conc$ converges to $\cle(3)$ in distribution.

As in the second paragraph in the proof of \cite[Theorem~6]{benoist2019scaling}, we argue that level-$2$ rightmost Ising loops inside any fixed level-$1$ leftmost Ising loop $l$ converges to $\cle(3)$.
Using the reasoning above again, we deduce the convergence of $L_2^\conc$.
Iteratively, we identify the limit of $L_{i+1}^\conc$ as independent $\cle(3)$ inside the level-$i$ loops.

The third paragraph in the proof of \cite[Theorem~6]{benoist2019scaling} implies that uniformly in $\delta$, the supremum of the diameters of the level-$i$ Ising loops converges to $0$ in probability as $i\to\infty$. This tightness result along with the convergence of $L_i^\conc$ for each $i$ implies that $\cup_{i=1}^N L_i^\conc$ converges in distribution in $\mathscr{L}_\D$ to $L^{\cle(3)}$, where $N$ is given in~\eqref{e.N=inf...}. The desired result follows from this and the definition of $\eta^\lop_\delta$ in~\eqref{e.kappa^loop}.
\end{proof}

\subsection{Proof of Result~\ref{resultB}}

Recall the notation of the Bernoulli perturbation in Section~\ref{s.bernoulli_pert}. In this subsection, we require $\bt_\delta$ to satisfy
\begin{align}\label{e.(t_delta)_e_in}
    (\bt_\delta)_e \in \Ll[0, t^\star\Rr], \quad\forall e \in \EE(\D_\delta),
\end{align}
for $t^\star = 1 - \Ll(\frac{1}{\cosh(\bbeta_\crit)}\Rr)^2$, and take $\omega_\delta$ as in~\eqref{e.omega_delta}.

To apply Theorem~\ref{t.(eta,omega)_id_lim}, we need to verify~\ref{i.H_basics},~\ref{i.H_eta_cross_similar}, and~\ref{i.H_bdy_conn}.

\subsubsection{Verification of~\ref{i.H_basics}}

\begin{lemma}
Condition~\ref{i.H_basics} is satisfied by $\eta_\delta$.
\end{lemma}

\begin{proof}
We verify the Markov property~\eqref{e.markov_prop}. Assume the condition in~\eqref{e.markov_prop}. On the event that $(\eta_\delta)_{\D_\delta\setminus K} =\kappa$, due to $\partial_{\D_\delta\setminus K
} \kappa=\emptyset$ and the map $\sigma^*\mapsto \eta$ in Section~\ref{s.Ising_dual_graph}, $K$ is surrounded by a weak loop of spins of plus or minus sign on the dual graph. Then, the Markov property follows from the Krammers--Wannier duality and the Markov property of $\sigma^*$.
\end{proof}

\subsubsection{Verification of~\ref{i.H_eta_cross_similar}}

\begin{lemma}Condition~\ref{i.H_eta_cross_similar} is satisfied by $\eta_\delta$.
\end{lemma}

\begin{proof}
For $s>0$, we set $\Lambda_s = [-s,s]^2$. For $s<s'$, we set $A_{s,s'}= \{s\leq |z|_\infty\leq s'\} = [-s',s']^2\setminus(-s,s)^2$. 
Since $A$ is dyadic, there is a finite sequence of points $(z_i)_{i=1}^n$ in $\R^2$ and some $\beta>0$ such that the sequence of squares $(z_i+\Lambda_\beta)_{i=1}^n$ satisfies the following:
\begin{itemize}
    \item all of them are contained in the closure of the bounded component of $\R^2\setminus A$;
    \item their interiors are disjoint;
    \item $\partial_0 A$ is the union of some of their edges.
\end{itemize}
For $\eps>0$, denote the outer boundary of $\cup_{i=1}^n (z_i+ \Lambda_{\beta+\eps})$ by $\Gamma^\eps$. For sufficiently small $\eps$, $\Gamma^\eps$ is a simple loop. Let $A^\eps$ be the annulus with $\partial_0 A^\eps =\Gamma^\eps$ and $\partial_1 A^\eps = \partial_1 A$.

Let $D_A$ be the domain with $\partial D_A = \partial_1 A$.
Then, we find $R,R'>0$ such that
\begin{align}\label{e.z+LambadasubsetDsubset}
    z_i+\Lambda_{R'}\subset D_A \subset z_i +\Lambda_{R},\quad\forall i\in\{1,\dots,n\}.
\end{align}
For $z\in \D$ and $\eps>0$, we define $M^\eps_\pm (z)$ to be the event that there is a vertex $x$ on the dual graph such that
\begin{itemize}
    \item $x\in z+A_{\beta,\beta+\eps}$;
    \item there is a loop of $\mp$ spins passing through $x$ and circulating around the annulus $z+A_{\beta,R}$;
    \item there is a weak path of $\pm$ spins inside $z+A_{\beta,R'}$ connecting a weak neighbor of $x$ to $\partial(z+\Lambda_{R'})$.
\end{itemize}
Then, we argue that
\begin{align*}
    \circledcirc_{A^\eps}^{\eta_\delta} \setminus \circledcirc_{A}^{\eta_\delta} \subset \bigcup_{\bullet \in \{-,+\}}\bigcup_{i=1}^n M^\eps_\bullet (z_i).
\end{align*}
Assuming the left-hand side, there is an Ising loop crossing $A^\eps$, but no Ising loop crossing $A$. Suppose this loop contains $\pm$ spins, then $\partial_1 A$ is connected to $\Gamma^\eps$ via a weak path of $\pm$ spins. The definition of $\Gamma_\eps$ ensures the existence of some $z_i$ such that this path connects a dual vertex $y$ in $z_i + A_{\beta,\beta+\eps}$ to $\partial_1 A$. On $\left(\circledcirc^{\eta_\delta}_A\right)^\complement$, we can choose $y$ whose neighbor $x$ has $-$ spin and $x$ is in $z_i+A_{\beta,\beta+\eps}$. Thus, $x$ is on a loop of $\mp$ spins circulating $A$, for otherwise the aforementioned path reaches $\partial_0 A$. By the choices of $R,R'$ as in~\eqref{e.z+LambadasubsetDsubset}, we see that $M^\eps_\pm(z_i)$ happens.

Recall that $A$ is away from the boundary of $\D_\delta$.
Using the above display and mixing properties of the critical Ising model (e.g.\ \cite[Corollary~5.4]{wu2018alternating}), for some constant $C$ arising from the mixing property, we have
\begin{align*}
    \P_{\D_\delta}^\emptyset\Ll(\circledcirc_{A^{\eps}}^{\eta_\delta} \setminus\circledcirc_{A}^{\eta_\delta}\Rr) \leq \sum_{\bullet\in\{+,-\}}\sum_{i=1}^n C\P\Ll(M^\eps_\bullet(z_i)\Rr),
\end{align*}
where $\P$ is the full-plane critical Ising measure on the dual graph of $\delta\Z^2$. 
Then, the desired result follows from
\begin{align}\label{e.limP(M)=0}
    \lim_{\eps\to0}\limsup_{\delta\to0}\P\Ll(M^\eps_\pm(z)\Rr) =0.
\end{align}
for each fixed $z\in\R^2$. We prove~\eqref{e.limP(M)=0} in the following.
\end{proof}

\begin{proof}[Proof of~\eqref{e.limP(M)=0}]
We assume the setting and use the notations from above.
By translation, it suffices to study the Ising model on the primal graph $\delta\Z^2$ and prove~\eqref{e.limP(M)=0} for $z=0$. Indeed, by substituting $\beta-\delta_0, R+\delta_0, R'-\delta_0$ for $\beta,R,R'$ with some sufficiently small $\delta_0$, we can recenter to the vertex $z_\delta$ nearest to $z$, and then we translate from $z_\delta$ to $0$. Additionally, we only consider $M^\eps_+(0)$, since $M^\eps_-(0)$ is treated similarly.

The basic idea is that on $M^\eps_\pm(0)$, there are alternating three arm events centered in $A_{\beta,\beta+\eps}$. The standard argument in the proof of \cite[Theorem~23]{nolin2008near} (see also \cite[First exercise sheet]{werner2007lectures}) proves that the probability of such a three arm event is $\mathcal{O}(\delta^2)$, due to topological reasons. Since $A_{\beta,\beta+\eps}$ has $\mathcal{O}(\eps/\delta^2)$ vertices, a union bound gives that the probability of $M^\eps_+(0)$ is $\mathcal{O}(\eps)$.

In this proof, we take $\delta<\frac{\eps}{10}$ small. We denote by $C$ a positive absolute constant that may vary from instance to instance. A point $z\in \R^2$ is said to be \textbf{in the corner} if $|z_1|=|z_2|$.
For $0<s_1<s_2, s_3$, we denote by $V^{s_1,s_2,s_3}_\pm$ the random set of vertices $z$ satisfying
\begin{itemize}
    \item $z$ is on a loop of $\mp$ spins circulating $A_{s_1,s_2}$;
    \item there is a weak path of $\pm$ spins connecting $\{|z'|_\infty=s_3\}$ to 
    \begin{itemize}
        
        \item a weak neighbor of $z$ if $z$ is in the corner,
        \item a neighbor of $z$, otherwise.
    \end{itemize}
\end{itemize}
For every vertex $x$, we define $E^{s_1,s_2,s_3}_{\pm,x}$ to be the event that $x$ is the minimizer of the function $z\mapsto|z|_\infty$ over $V^{s_1,s_2,s_3}_\pm$.
Comparing this with the definition of $M^\eps_+(0)$, we have
\begin{align}\label{e.P()<I_++I_-}
    M^\eps_+(0)\subset
     \bigcup_{x\in A_{\beta,\beta+\eps}}E^{\beta,R,R'}_{+,x}.
\end{align}

We fix a constant $\bar R$ sufficiently large such that $ \bigcup_{x\in A_{\beta,\beta+\eps}}E^{\beta,R,R'}_{+,x}$ depends only on the configuration inside $\Lambda_{\frac{1}{2}\bar R}$. For a subset $S$ of vertices, we write $\sigma_S = (\sigma_i)_{i\in S}$. Then, fixing any $\sigma_{\delta\Z^2\setminus \Lambda_{\bar R}}$, we set
\begin{align}\label{e.hatP}
    \hat{\P}(\,\cdot\,) = \P\Ll(\ \cdot\ \big|\sigma_{\delta\Z^2\setminus \Lambda_{\bar R}}\Rr).
\end{align}
This enables us to ``count'' configurations in relevant events. In the following paragraph, we write $\sigma = \sigma_{\delta\Z^2\cap\Lambda_{\bar R}}$.

We denote by $G^{s_1,s_2,s_3}_{+,x}$ the event that $x$ is the unique minimizer of the function $z\mapsto |z|_\infty$ over $V^{s_1,s_2,s_3}_+$. 
We choose a vertex $\tilde x$ in the following way (see Figure~\ref{fig:picHBCflip}):
\begin{itemize}
    \item if $x$ is in the corner, let $\tilde x$ be the unique weak neighbor of $x$ such that $|\tilde x_1|=|\tilde x_2|$ and $|\tilde x|_\infty = |x|_\infty -\delta$;
    \item otherwise, let $\tilde x$ be the unique neighbor of $x$ such that $|\tilde x|_\infty = |x|_\infty -\delta$.
\end{itemize}
For every $\sigma \in E^{\beta,R,R'}_{+,x}$, we flip at most four spins near $x$ (more precisely, at most two if $x$ is in the corner) to get a new configuration $\tilde \sigma \in G^{\beta-\delta,R,R'}_{+,\tilde x}$.
Note that every $G^{\beta,\delta,R,R'}_{+,\tilde x}$ only depends on the vertices in $\Lambda_{\frac{1}{2}\bar R}$.
Using the expression of weights in~\eqref{e.Ising_weight} for the Ising model, we see that each spin flip introduces a multiplicative factor bounded by $e^{8\bbeta_\crit}$ to the weight. This gives $w(\sigma)\leq e^{32\bbeta_\crit}w(\tilde \sigma)$.
Also, there are at most $16$ distinct $\sigma$ to be mapped to the same $\tilde \sigma$.
Therefore, it holds for all $x\in A_{\beta,\beta+\eps}$ that
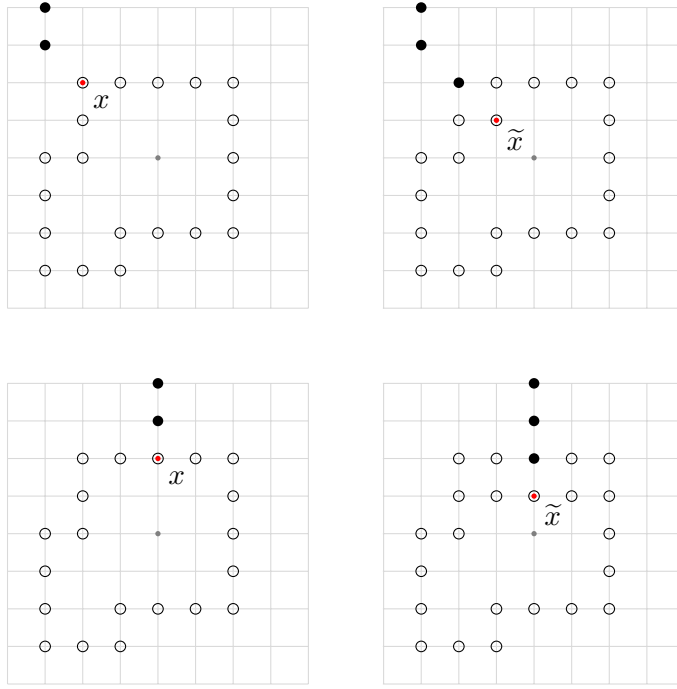
\begin{figure}
    \centering
    
\begin{tikzpicture}
\draw[black!15, step=5mm] (0,0) grid (4,4);
\draw[black!15, step=5mm] (4.99,0) grid (9,4);
\draw[black!15, step=5mm] (0,-5) grid (4,-1);
\draw[black!15, step=5mm] (4.99,-5) grid (9,-1);
\fill[gray] (2,2) circle (1pt)
(7,2) circle (1pt)
(2,-3) circle (1pt)
(7,-3) circle (1pt)
;
\foreach \i in {0.5,1,1.5}
{\draw[] (\i,0.5) circle (2pt);
} 
\foreach \i in {3,2.5,2,1.5}
{\draw[] (\i,1) circle (2pt);
}
\foreach \i in {3,2.5,2,1.5,1}
{\draw[] (\i,3) circle (2pt);
}
\foreach \j in {2,1.5,1}
{\draw[] (0.5,\j) circle (2pt);
}
\foreach \j in {2,2.5}
{\draw[] (1,\j) circle (2pt);
}
\foreach \j in {2,2.5,1.5}
{\draw[] (3,\j) circle (2pt);
}
\foreach \j in {4,3.5}
{\fill[] (0.5,\j) circle (2pt);
}
\fill[red] (1,3) circle (1pt);
\draw[] (1.25,2.75) node {$x$};

\foreach \i in {5.5,6,6.5}
{\draw[] (\i,0.5) circle (2pt);
} 
\foreach \i in {8,7.5,7,6.5}
{\draw[] (\i,1) circle (2pt);
}
\foreach \i in {8,7.5,7,6.5}
{\draw[] (\i,3) circle (2pt);
}
\foreach \j in {2,1.5,1}
{\draw[] (5.5,\j) circle (2pt);
}
\foreach \j in {2,2.5}
{\draw[] (6,\j) circle (2pt);
}
\foreach \j in {2,2.5,1.5}
{\draw[] (8,\j) circle (2pt);
}
\foreach \j in {4,3.5}
{\fill[] (5.5,\j) circle (2pt);
}
\fill[] (6,3) circle (2pt);
\draw[] (6.5,2.5) circle (2pt);
\fill[red] (6.5,2.5) circle (1pt);
\draw[] (6.75,2.25) node {$\tilde x$};

\foreach \i in {0.5,1,1.5}
{\draw[] (\i,-4.5) circle (2pt);
} 
\foreach \i in {3,2.5,2,1.5}
{\draw[] (\i,-4) circle (2pt);
}
\foreach \i in {3,2.5,2,1.5,1}
{\draw[] (\i,-2) circle (2pt);
}
\foreach \j in {-3,-3.5,-4}
{\draw[] (0.5,\j) circle (2pt);
}
\foreach \j in {-3,-2.5}
{\draw[] (1,\j) circle (2pt);
}
\foreach \j in {-3,-2.5,-3.5}
{\draw[] (3,\j) circle (2pt);
}
\foreach \j in {-1,-1.5}
{\fill[] (2,\j) circle (2pt);
}
\fill[red] (2,-2) circle (1pt);
\draw[] (2.25,-2.25) node {$x$};

\foreach \i in {5.5,6,6.5}
{\draw[] (\i,-4.5) circle (2pt);
} 
\foreach \i in {8,7.5,7,6.5}
{\draw[] (\i,-4) circle (2pt);
}
\foreach \i in {8,7.5,6.5,6}
{\draw[] (\i,-2) circle (2pt);
}
\foreach \j in {-3,-3.5,-4}
{\draw[] (5.5,\j) circle (2pt);
}
\foreach \j in {-3,-2.5}
{\draw[] (6,\j) circle (2pt);
}
\foreach \j in {-3,-2.5,-3.5}
{\draw[] (8,\j) circle (2pt);
}
\foreach \j in {-1,-1.5}
{\fill[] (7,\j) circle (2pt);
}
\foreach \i in {6.5,7,7.5}
{\draw[] (\i,-2.5) circle (2pt);
}
\fill[red] (7,-2.5) circle (1pt);
\fill[] (7,-2) circle (2pt);
\draw[] (7.25,-2.75) node {$\tilde x$};
\end{tikzpicture}

    \caption{The figures in the first row consider $x$ in the corner; the second row is for the other case.}
    \label{fig:picHBCflip}
\end{figure}
\begin{align*}
    \hat\P \Ll( E_{+,x}^{\beta,R,R'}\Rr) =\sum_{\sigma\in E^{\beta,R,R'}_{+,x}}w(\sigma) \leq \sum_{\sigma\in E^{\beta,R,R'}_{+,x}}e^{32\bbeta_\crit}w\Ll(\tilde \sigma\Rr) \leq  16e^{32\bbeta_\crit}\hat\P \Ll( G_{+,\tilde x}^{\beta-\delta,R,R'}\Rr).
\end{align*}
By construction, at most three distinct $x$ can be mapped to the same $\tilde x$ (one-to-one if $\tilde x$ is not in the corner). 
Hence, the above display along with~\eqref{e.P()<I_++I_-} and~\eqref{e.hatP} implies that
\begin{align}\label{e.I_+<CSum}
    \P\Ll(M^\eps_+(0)\Rr) \leq 48e^{32\bbeta_\crit}\sum_{x\in A_{\beta-\delta,\beta+\eps}}\P \Ll( G_{+, x}^{\beta-\delta,R,R'}\Rr).
\end{align}

We fix a constant $\gamma<10^{-1}\beta$ and consider $\eps \leq 10^{-1}\gamma$ henceforth. 
For $a, b\in \R$, we set $[a,b]_\delta = [a,b]\cap \delta \N$. 
For $r\in \N$, we set
\begin{align*}
    B^{s_1,s_2,s_3}_{+,r} = \bigcup_{x:|x|_\infty =r} G^{s_1,s_2,s_3}_{+,x}.
\end{align*}
It is useful to see that
\begin{align}\label{e.GcapG=emptyset}
    G^{s_1,s_2,s_3}_{+,x} \cap G^{s_1,s_2,s_3}_{+,x'} =\emptyset,\quad\forall x,x' \in A_{s_1,\min\{s_2,s_3\}}: x\neq x'\,.
\end{align}
For convenience, we introduce some notations. For $s\geq 0$, we set
\begin{align*}
    G^{-s}_{+,x}= G^{\beta-\gamma-s,R+s,R'-s}_{+,x},\qquad  B^{-s}_{+,r}= B^{\beta-\gamma-s,R+s,R'-s}_{+,r}.
\end{align*}
We later use the relation $G^{-s}_{+,x}\subset G^{-s'}_{+,x}$ for $s\leq s'$ many times.
Note that~\eqref{e.GcapG=emptyset} implies that
\begin{align}\label{e.BcapB=emptyset}
    B^{-s}_{+,q}\cap B^{-s}_{+,q'}=\emptyset,\quad\forall q,q'\in [\beta-\gamma-s,R'-s]_\delta,
\end{align}
which along with~\eqref{e.I_+<CSum}, ~\eqref{e.GcapG=emptyset}, and $\delta<\eps<\gamma$ yields
\begin{align}\label{e.I_+<CSumB}
    \P\Ll(M^\eps_+(0)\Rr)\leq 48e^{32\bbeta_\crit}\sum_{r\in [\beta-\eps,\beta+\eps]_\delta} \P\Ll(B^{-0}_{+,r}\Rr).
\end{align}

For any $r\in [\beta-\gamma,\beta+\gamma]_\delta$ and vertex $x$ satisfying $|x|_\infty =r$, we take $\tilde x$ to be the vertex on $\{z:|z|_\infty = \beta-\gamma\}$ closest to the line segment between $x$ and the origin. Considering the translation that moves $x$ to $\tilde x$, we have
\begin{align*}
    \P\Ll(G^{-0}_{+,x}\Rr) \leq \P\Ll(G^{-2\gamma}_{+,\tilde x}\Rr).
\end{align*}
Due to $\frac{\beta+\gamma}{\beta-\gamma}<2$, for each $r\in [\beta-\gamma,\beta+\gamma]_\delta$, at most two $ x$ on $\{z:|z|_\infty = r\}$ are mapped to the same $\tilde x$ on $\{z:|z|_\infty = \beta-\gamma\}$. Using this, the above display, and~\eqref{e.GcapG=emptyset}, we have
\begin{align}\label{e.P(B)<2P(B)}
    \P\Ll(B^{-0}_{+,r}\Rr) = \sum_{|x|_\infty = r}\P\Ll(G^{-0}_{+,x}\Rr) \leq 2\sum_{|\tilde x|_\infty = \beta-\gamma}\P\Ll(G^{-2\gamma}_{+,\tilde x}\Rr) = 2\P\Ll(B^{-2\gamma}_{+,\beta-\gamma}\Rr),\, \forall r\in  [\beta-\gamma,\beta+\gamma]_\delta.
\end{align}

Then, we apply the same argument again. For any $r'\in [\beta-2\gamma,\beta-\gamma]_\delta$ and any vertex $y$ satisfying $|y|_\infty =\beta-\gamma$, we take $\tilde y$ to be the vertex on $\{z:|z|_\infty = r'\}$ closest to the line segment between $y$ and the origin. Translating $y$ to $\tilde y$, we get
\begin{align*}
    \P\Ll(G^{-2\gamma}_{+,y}\Rr) \leq \P\Ll(G^{-(2\gamma+\beta-\gamma-r')}_{+,\tilde y}\Rr)\leq \P\Ll(G^{-3\gamma}_{+,\tilde y}\Rr).
\end{align*}
Now, due to $\frac{\beta-\gamma}{r'}<2$ for $r'\in [\beta-2\gamma,\beta-\gamma]_\delta$, at most two $y$ on $\{z:|z|_\infty = \beta-\gamma\}$ are mapped to the same $\tilde y$ on $\{z:|z|_\infty =r'\}$. This along with the above display and~\eqref{e.GcapG=emptyset} implies that
\begin{align*}
    \P\Ll(B^{-2\gamma}_{+,\beta-\gamma}\Rr) =\sum_{|y|_\infty =\beta-\gamma}\P\Ll(G^{-2\gamma}_{+,y}\Rr)\leq 2\sum_{|\tilde y|_\infty =r'} \P\Ll(G^{-3\gamma}_{+,\tilde y}\Rr)= 2\P\Ll(B^{-3\gamma}_{+,r'}\Rr)
\end{align*}
for all $r'\in  [\beta-2\gamma,\beta-\gamma]_\delta$. 
Together with~\eqref{e.BcapB=emptyset}, it shows that
\begin{align*}
    1\geq \P\Ll(\bigcup_{r'\in[\beta-2\gamma,\beta-\gamma]_\delta}B^{-3\gamma}_{+,r'}\Rr)\geq \frac{\gamma}{2\delta}\P\Ll(B^{-2\gamma}_{+,\beta-\gamma}\Rr).
\end{align*}
Applying this and~\eqref{e.P(B)<2P(B)} to~\eqref{e.I_+<CSumB}, we obtain
\begin{align*}
    \P\Ll(M^\eps_+(0)\Rr)\leq 48e^{32\bbeta_\crit} \sum_{r\in[\beta-\eps,\beta+\eps]_\delta} \frac{4\delta}{\gamma} = \frac{384e^{32\bbeta_\crit}\eps}{\gamma}.
\end{align*}
Since $\gamma$ is fixed, this implies the desired result.
\end{proof}

\subsubsection{Verification of~\ref{i.H_bdy_conn}}

To show~\ref{i.H_bdy_conn}, we cite a technical estimate of boundary connectivity probability for double currents as in~\cite[Theorem~1.2]{duminil2021conformal}.

\begin{theorem}[\cite{duminil2021conformal}]\label{t.bdy_con}
Let $\delta=1$.
There is a function $\boldsymbol{\eps}:(0,\infty)\to [0,\infty)$ satisfying $\lim_{s\to 0} \boldsymbol{\eps}(s)=0$ such that for all $r, R>0$,
\begin{align*}\mathbf{P}^\emptyset_K\otimes \mathbf{P}^\emptyset_K\Ll( [-R,R]^2 \stackrel{\hat{\bn^1+\bn^2}\cap K}{\longleftrightarrow} (\partial K)_r \Rr) \leq \boldsymbol{\eps} \Ll(\frac{r}{R}\Rr)
\end{align*}
uniformly for $1$-discrete disc $K$ satisfying $[-2R,2R]^2 \subset K \not\subset [-3R,3R]^2$, where $\bn^1$ and $\bn^2$ are two independent sourceless random currents on $K$ under $\mathbf{P}^\emptyset_K\otimes \mathbf{P}^\emptyset_K$; and $\hat{\bn^1+\bn^2}$ is the trace of $\bn^1+\bn^2$, namely, $(\hat{\bn^1+\bn^2})_e = \mathds{1}_{(\bn^1+\bn^2)_e\neq 0}$ for every edge $e$. 
\end{theorem}

Originally in \cite[Theorem~1.2]{duminil2021conformal}, $K$ is a $1$-discrete domain. Considering the largest $1$-discrete domain contained in $K$, we see that the same holds for $1$-discrete discs.

Recall that we have taken $\omega_\delta$ as in~\eqref{e.omega_delta} for $(\bt_\delta)_\delta$ satisfying~\eqref{e.(t_delta)_e_in}.

\begin{lemma}
Condition~\ref{i.H_bdy_conn} is satisfied by $(\eta_\delta,\omega_\delta)$.
\end{lemma}
\begin{proof}
Clearly, Theorem~\ref{t.bdy_con} holds on the scaled lattice $\delta\Z^2$ for all $\delta$ sufficiently smaller than $r$ and $K$ translated along the lattice.
The coupling between $\eta_\delta$ and $\hat \bn_\delta$ in~\eqref{e.eta_vee_b} ensures that $\hat{\bn^1_\delta+\bn^2_\delta}$ has the same distribution of $\eta_\delta \vee \bb^{t_\crit}_\delta \vee \tilde \eta_\delta \vee \tilde \bb^{t_\crit}_\delta$ where $(\tilde \eta_\delta,\tilde \bn^{t_\crit}_\delta)$ is a copy of $(\eta_\delta,\bn^{t_\crit}_\delta)$, and the four percolation configurations are independent. Note that $\bb^{t_\crit}_\delta\vee \tilde \bb^{t_\crit}_\delta \stackrel{\d}{=} \bb^{t^\star}_\delta$ for $t^\star$ given below~\eqref{e.(t_delta)_e_in}, due to $t^\star = 1 -(1-t_\crit)^2$.
Recall that $\omega_\delta$ is taken as in~\eqref{e.omega_delta} under the condition~\eqref{e.(t_delta)_e_in}. Hence, there is a coupling under which $\omega_\delta \subset \eta_\delta\vee \bb^{t^\star}_\delta \subset \hat{\bn^1_\delta+\bn^2_\delta}$.
This relation and Theorem~\ref{t.bdy_con} yield that
\begin{align*}
    \P_K^\emptyset\Ll( x+[-R,R]^2 \stackrel{\omega_\delta\cap K}{\longleftrightarrow} (\partial K)_r \Rr) \leq\mathbf{P}^\emptyset_{ K}\otimes \mathbf{P}^\emptyset_{ K}\Ll( [-R,R]^2 \stackrel{\hat{\bn^1_\delta+\bn^2_\delta}\cap  K}{\longleftrightarrow} \Ll(\partial  K\Rr)_r \Rr) \leq \boldsymbol{\eps} \Ll(\frac{r}{R}\Rr)
\end{align*}
for $\delta$ sufficiently smaller than $r$, which gives~\ref{i.H_bdy_conn}.
\end{proof}

\subsubsection{Derivation of the main result}

Since $(\eta_\delta,\omega_\delta)$ satisfies~\ref{i.H_basics}--\ref{i.H_bdy_conn},
we are ready to prove the main result rigorously stated below.
Recall the function $\hat F$ defined in~\eqref{e.F_biject}.
Recall from Section~\ref{s.conformal_inv} the notion of conformal invariance for an $\mathscr{H}_{\mathcal{A}_\D}$-valuable random variable.

\begin{theorem}\label{t.cvg_rc_cle}
For each $\delta>0$, let $\eta_\delta$ be distributed as in~\eqref{e.eta_weight} and let $\omega_\delta$ be given as in~\eqref{e.omega_delta} with $\bt_\delta$ satisfying~\eqref{e.(t_delta)_e_in}.

Then, as $\delta\to0$, $(\eta^\lop_\delta,\omega^\cro_\delta)$ converges in distribution in $\mathscr{L}_\D\times\mathscr{H}_{\mathcal{A}_\D}$ to $(L^{\cle(3)},\omega^\cro_0)$ where $\omega^\cro_0$ satisfies $\omega^\cro_0 = \hat F(L^{\cle(3)})$ and $\hat F^{-1}(\omega^\cro_0) = L^{\cle(3)}$ a.s.
In particular, $\omega^\cro_0$ is conformally invariant.

\end{theorem}

\begin{proof}
The main part of the theorem follows from Theorem~\ref{t.(eta,omega)_id_lim} and Lemma~\ref{l.eta^loop->CLE}. The conformal invariance of $\omega^\cro_0$ follows from the conformal invariance of $L^{\cle(3)}$, Lemma~\ref{l.conformal_inv}, and the fact that $\hat F$ is a restriction of $F$.
\end{proof}

Recall that $\hat\bn_\delta$ admits the representation in~\eqref{e.eta_vee_b}. Hence, $\hat\bn_\delta \stackrel{\d}{=} \omega_\delta$ when all the entries in $\bt_\delta$ are equal to $t_\crit$. 
Comparing $t^\star$ and $t_\crit$ defined below~\eqref{e.(t_delta)_e_in} and~\eqref{e.eta_vee_b}, respectively, we have $t_\crit<t^\star$, satisfying~\eqref{e.(t_delta)_e_in}. Therefore, the following corollary for the sourceless random current, the rigorous version of Result~\ref{resultB}, is immediate.

\begin{corollary}\label{c.cvg_rc}
As $\delta\to0$, $(\eta^\lop_\delta,\hat\bn^\cro_\delta)$ converges in distribution in $\mathscr{L}_\D\times\mathscr{H}_{\mathcal{A}_\D}$ to $(L^{\cle(3)},\hat\bn^\cro_0)$ where $\hat\bn^\cro_0$ satisfies $\hat\bn^\cro_0 = \hat F(L^{\cle(3)})$ and $\hat F^{-1}(\hat\bn^\cro_0) = L^{\cle(3)}$ a.s.
In particular, $\hat\bn^\cro_0$ is conformally invariant.

\end{corollary}

\small
\bibliographystyle{abbrv}

\end{document}